\documentclass{amsart}

\usepackage{amsmath}
\usepackage{amsfonts}
\usepackage{amsthm}
\usepackage{graphicx}
\usepackage{todonotes}

\usepackage{marginnote}

\newtheorem{theorem}{Theorem}[section]
\newtheorem{lemma}[theorem]{Lemma}
\newtheorem{proposition}[theorem]{Proposition}

\newtheorem{corollary}[theorem]{Corollary}
\newtheorem{definition}[theorem]{Definition}
\newtheorem{remark}[theorem]{Remark}
\newtheorem{example}[theorem]{Example}

\newcommand{\R}{\mathbb{R}}

\newcommand{\F}{\mathcal{F}}
\newcommand{\D}{\mathcal{D}}
\newcommand{\Id}{\mathrm{id}}
\renewcommand{\d}{\mathrm{d}}

\renewcommand{\epsilon}{\varepsilon}
\newcommand*{\mailto}[1]{\href{mailto:#1}{\nolinkurl{#1}}}

\numberwithin{equation}{section}

\begin{document}
\title[Lipschitz stability for $\alpha$-dissipative Hunter--Saxton]{Existence and Lipschitz stability for $\alpha$-dissipative solutions of the two-component Hunter--Saxton system }
\allowdisplaybreaks

\author[K. Grunert]{Katrin Grunert}
\address{Department of Mathematical Sciences\\
  NTNU\\Norwegian University of Science and Technology\\
  7491 Trondheim\\ Norway}
\email{\mailto{katring@math.ntnu.no}}
\urladdr{\url{http://www.math.ntnu.no/~katring/}}

\author[A. Nordli]{Anders Nordli}
\address{Department of Mathematical Sciences\\ NTNU\\ Norwegian University of Science and Technology\\ NO-7491 Trondheim\\ Norway}
\email{\mailto{andenors@math.ntnu.no}}

\thanks{Research supported by the grant {\it Waves and Nonlinear Phenomena (WaNP)} from the Research Council of Norway.}

\begin{abstract}
We establish the concept of $\alpha$-dissipative solutions for the two-component Hunter--Saxton system under the assumption that ei\-ther $\alpha(x)=1$ or $0\leq \alpha(x)<1$ for all $x\in \R$. Furthermore, we investigate the Lipschitz stability of solutions with respect to time by introducing a suitable parametrized family of metrics in Lagrangian coordinates. This is necessary due to the fact that the solution space is not invariant with respect to time.
\end{abstract}

\maketitle

\section{Introduction}\label{section intro}

In this paper we investigate the existence and Lipschitz stability of solutions of the initial value problem of the two-component Hunter--Saxton (2HS) system on the line, which is given by
\begin{subequations}
\label{hunter-saxton-system}
\begin{align}
u_t(x,t)+uu_x(x,t) &= \frac 14\bigg(\int_{-\infty}^x(u_x(z,t)^2+\rho(z,t)^2)\:\d z\nonumber\\ &\qquad\qquad\qquad  -\int_x^{\infty}(u_x(z,t)^2+\rho(z,t)^2)\:\d z\bigg),\\
\rho_t(x,t) +(u\rho)_x(x,t) &= 0.
\end{align}
\end{subequations}
It has been derived by Pavlov as a model of non-dissipative dark matter \cite{P}, but can also be viewed as a high frequency limit of the two-component Camassa--Holm system describing water waves \cite{DP98, W10}.
Moreover, it is a generalization of the well known Hunter--Saxton equation
\begin{equation}
\label{hunter-saxton}
u_t(x,t)+uu_x(x,t) = \frac 14\bigg(\int_{-\infty}^x u_x(z,t)^2\:\d z-\int_x^{\infty}u_x(z,t)^2\:\d z\bigg),
\end{equation}
which has been introduced by Hunter and Saxton as a model of the director field of a nematic liquid crystal \cite{HS}. 

Solutions of the 2HS system develop singularities in finite time, even for smooth initial data \cite{N,W10}. The appearance of singularities, known as wave breaking, means that $u_x$ tends pointwise to $-\infty$ while $u$ remains bounded and continuous. The phenomenon is illustrated in the following example.
\begin{example}
\label{example solution}
Let $t\in [0,2)$ and let the functions $u$ and $\rho$ be defined by
\begin{align*}
u(x,t) &=
	\begin{cases}
	-\frac 12t+1,& x\leq -(1-\frac12 t)^2,\\
	-\frac{1}{-\frac 12t+1}x,& -(1-\frac12 t)^2\leq x \leq 0,\\
	\frac{t}{\frac 12t^2+2}x, &	0\leq x \leq \frac 14 t^2+1,\\
	\frac 12 t, & \frac 14t^2+1 \leq x,
	\end{cases}\\
\rho(x,t) &=
	\begin{cases}
	0,& x\leq 0,\\
	\frac{1}{\frac 14t^2+1}, &0< x \leq \frac 14 t^2+1,\\
	0, & \frac 14 t^2+1<x.
	\end{cases}
\end{align*}
Then $(u,\rho)$ is a weak solution of \eqref{hunter-saxton-system} for $t\in[0,2)$. Note that $u_{x}(0,t)\rightarrow -\infty$ as $t\rightarrow 2^-$, which in particular means that wave breaking occurs. We can define the energy of the system at time $t$ to be given by
\begin{equation}
\int_{\R} \big(u_x^2(x,t)+\rho^2(x,t)\big)\:\d x = 2,
\end{equation}
which is constant in time, even up to the point $t=2$. The energy contained in the interval $-\frac 14 t^2 + t -1\leq x\leq 0$, given by $\int_{-\frac 14 t^2 + t -1}^0 (u_x^2+\rho^2)\:\d x = 1$, is also conserved. Thus a finite amount of energy is being concentrated in a single point as $t\rightarrow 2^-$.
\begin{figure}
\includegraphics[width=8cm]{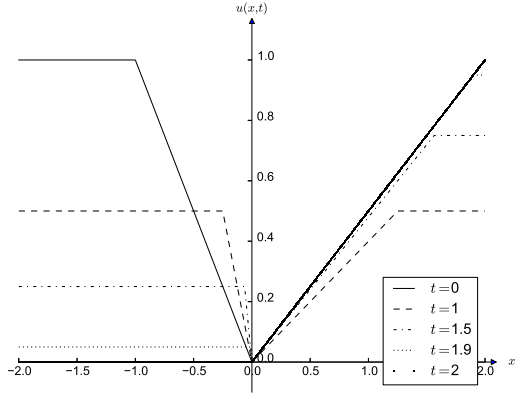}
\caption{A plot of $u$ in Example \ref{example solution} for $t=0,1,1.5,1.9,2$.}
\label{fig example intro}
\end{figure}
\end{example}
As seen in Example \ref{example solution} part of the energy $\int_\R(u_x^2+\rho^2)\:\d x$ is concentrated at a single point at wave breaking. This illustrates that the energy density in general is not absolutely continuous, but a positive, finite Radon measure. Nevertheless, the total energy remains constant in time as $t\rightarrow 2^-$. Hence $u_x$ and $\rho$ remain in $L^2(\R)$ even if $u_x$ tends to minus infinity and the energy can be described by the cumulative distribution function of a positive, finite Radon measure.

To continue solutions past wave breaking is a delicate issue since weak solutions are not unique afterwards. A property that the 2HS system shares with the Hunter--Saxton equation \cite{BC,BHR,BZZ,HS,HZ,ZZ}, the Camassa--Holm equation \cite{BC07cons,BC07diss,CH,HR07}, and the two-component Camassa--Holm system \cite{GHR}. To be more precise, there are infinitely many ways to extend a local solution to a global one past wave breaking by manipulating the concentrated energy at breaking time. In particular, there has been shown much interest in two classes of weak solutions, namely dissipative solutions and conservative solutions. On the one hand one could ignore the part of the energy that concentrates on a set of measure zero in the continuation, which yields dissipative solutions. On the other hand one could continue by letting the concentrated energy back into the system, which would give conservative solutions. In practice that would amount to defining $u$ and $\rho$ by the formulas in Example \ref{example solution} even for $t>2$. Thus it is essential to include the energy in our sets of variables, when constructing global conservative solutions. Existence of dissipative solutions has been proven in \cite{W10}, while existence and stability of conservative solutions has been shown in \cite{N}. Here we generalize and provide a unified approach to the notions of dissipative and conservative weak solutions.

We will work within the novel concept of $\alpha$-dissipative solutions, which has been introduced in \cite{GHR15} in the context of the two-component Camassa--Holm system for the case $\alpha$ being constant. Here we consider a more general case. Given a Lipschitz continuous function of space, $\alpha:\R\to [0,1]$, an $\alpha$-dissipative solution will dissipate an $\alpha$-fraction of the energy concentrated on a set of measure zero at wave breaking. Dissipative and conservative solutions are recovered as special cases with $\alpha = 1$ and $\alpha = 0$, respectively. Here we will construct $\alpha$-dissipative solutions and prove Lipschitz stability of the constructed solutions. 

We will solve the 2HS system by a generalized method of characteristics. The method will be similar to the one for conservative solutions \cite{N} and dissipative solutions \cite{W10}. As long as solutions $(u,\rho)$ stay smooth we can define the corresponding Lagrangian variables $(y,U,r,V)$ by
\begin{equation}
y_t(\xi,t) = u(y(\xi,t),t),\\
\end{equation}
and
\begin{align}
U(\xi,t) &= u(y(\xi,t),t),\\
V(\xi,t) &= \int_{-\infty}^{y(\xi,t)}u_x(x,t)^2+ \rho(x,t)^2\:dx = 0,\\
r(\xi,t) &= \rho(y(\xi,t),t)y_\xi(\xi,t),
\end{align}
where $\xi$ is a parameter that determines the initial value of the characteristics. Then
\begin{align}
U_t(\xi,t) &= \frac 12 V(\xi,t) - \frac 14\lim_{\xi\rightarrow\infty}V(\xi,t),\\
V_t(\xi,t) &= 0,\\
r_t(\xi,t) &= 0,
\end{align}
Wave breaking happens precisely where different characteristic curves $y(\xi,t),y(\xi',t)$ meet, and one can choose which solution to obtain by manipulating the energy density $V_\xi(\xi,t)$ at wave breaking. The time of wave breaking $\tau(\xi)$ can be determined initially and is given by
\begin{equation}
\tau(\xi) = -2\frac{y_{0,\xi}(\xi)}{U_{0,\xi}(\xi)}.
\end{equation}
 
To obtain dissipative solutions one sets $V_\xi(\xi,t) = 0$ after wave breaking, which implies that the characteristics stick together after they meet. One can solve the resulting system explicitly as shown by Wunsch \cite{W10}. If one on the other hand keeps $V_\xi(\xi,t)$ constant across wave breaking one gets conservative solutions \cite{N}. For dissipative solutions one can always choose $y(\xi,0) = \xi$, while for conservative solutions this choice is not possible since energy can then be concentrated on a set of measure zero initially. The main difficulty with $\alpha$-dissipative solutions compared to conservative or dissipative solutions is that it is not known \textit{a priori} how the energy density is manipulated at wave breaking. After wave breaking we have $V_\xi(\xi,t) = (1-\alpha(y(\xi,\tau(\xi)))) V_{0,\xi}(\xi)$. Since the time evolution of $y(\xi,t)$ depends heavily on the total amount of energy in the system, the points $y(\xi,\tau(\xi))$ where wave breaking happens will depend heavily on all other points where wave breaking has previously happened and cannot be computed explicitly. In Section \ref{section:exist lagrange} we reformulate the 2HS system in Lagrangian coordinates and prove the existence of $\alpha$-dissipative solutions in Lagrangian coordinates.

Similar to the case of conservative solutions, for $\alpha$-dissipative solutions parts of the energy can concentrate on a set of measure zero and later be given back into the system, hence one has to be careful when going from the Eulerian to the Lagrangian description of the system. Since we cannot choose $y(\xi,0) = \xi$ we need a systematic way of going from Eulerian coordinates $(u,\rho,\mu)$ to Lagrangian variables $(y,U,r,V)$, and vice versa. We will use the mappings developed for the Camassa--Holm equation in \cite{GHR,GHR15,HR07}, and modify them to the solution spaces for the 2HS system \cite{BHR,N}. Even though we cannot use $y(\xi,0) = \xi$, there is no unique way to define $y(\xi,0) = y_0(\xi)$. Thus there is some redundancy in the Lagrangian variables, and we can identify a group of homeomorphisms $\xi\mapsto f(\xi)$ that acts on the Lagrangian variables and identifies equivalence classes which are in one-to-one correspondence to the Eulerian coordinates. We will develop the mappings between Eulerian and Lagrangian coordinates, and use those mappings to prove existence of $\alpha$-dissipative solutions in Eulerian coordinates in Section \ref{section:exist euler}.

In Section \ref{section lipschitz metric}, we turn to the Lipschitz stability of solutions. Here the redundancy will cause some problems since two solutions may differ in Lagrangian coordinates, but coincide in Eulerian coordinates. To overcome this obstacle we must create a metric in Lagrangian coordinates that gives zero distance between solutions that map to the same Eulerian solution. Such a metric was created in \cite{BHR} for conservative solutions of \eqref{hunter-saxton}. Moreover, for metrics in Lagrangian coordinates to induce metrics in Eulerian coordinates we need that there is a bijection from Lagrangian coordinates to Eulerian coordinates. For this to be the case we must restrict $\alpha$ to either $\alpha\equiv 1$, that is dissipative solutions, or $\alpha:\R\rightarrow [0,1)$.

Since the size of the discontinuity in the energy at wave breaking is unknown initially, it is more difficult to establish stability of solutions in this case than in the conservative one. To do so we will need to keep track of the amount of energy initially. Thus in addition to the energy variable $\mu$ we will need an energy variable $\nu$ such that $\mu\leq\nu$, and $\mu_0 = \nu_0$. In Lagrangian coordinates this corresponds to a variable $H$ such that $V_\xi\leq H_\xi$ and $V_{0,\xi} = H_{0,\xi}$ and $H_t(\xi,t) = 0$. Hence the solution space will not be invariant with respect to time and therefore we will introduce a parametrized family of metrics $d_{\D_0^{\alpha,M}}(t,\cdot,\cdot)$ on sets with the total energy bounded by $M$. The main result of Section \ref{section lipschitz metric} is that the $\alpha$-dissipative solutions constructed in Section~\ref{section:exist euler} are Lipschitz continuous in time with respect to the initial data.

\section{Existence of solutions in Lagrangian coordinates}\label{section:exist lagrange}

In \cite{N} the conservative solutions to \eqref{hunter-saxton-system} have been constructed by rewriting the 2HS system as a system of differential equations in a suitable Banach space. Here we are aiming at constructing so-called $\alpha$-dissipative solutions where one takes out an $\alpha$-fraction of the concentrated energy every time wave breaking occurs. The description of these solutions is also based on a reformulation of \eqref{hunter-saxton-system} in Lagrangian coordinates via a generalized method of characteristics. In contrast to \cite{GHR15}, where the concept of $\alpha$-dissipative solutions has been introduced for $\alpha$ being a constant in $[0,1]$ in the context of the two-component Camassa--Holm system, we are considering functions
\begin{subequations}
\label{cond:alpha}
\begin{equation}
\alpha(x) \in W^{1, \infty}(\R)
\end{equation}
such that 
\begin{equation}
\text{ either }\quad 0\leq \alpha(x)<1 \quad \text{or} \quad \alpha(x)=1 \quad \text{ for all }x\in\R.
\end{equation}
\end{subequations}
In this section we will first introduce the concept of $\alpha$-dissipative solutions in Lagrangian coordinates and then establish their existence in this setting. 

Let $(u_0,\rho_0)$ be some smooth initial data for \eqref{hunter-saxton-system}, such that $u_0(x)\in L^\infty(\R)$ and $u_{0,x}(x)$, $\rho_0(x)\in L^2(\R)$. In addition, in this case the initial characteristic $y(\xi,0)$ can be chosen to be equal to the identity or more general to any strictly increasing function $y_0(\xi)$ belonging to the set of relabeling functions $G$, which will be introduced in Definition~\ref{definition_G}. Applying the method of characteristics, we obtain a local in time solution, by solving the initial value problem
\begin{subequations}
\begin{align}\label{lag:sm}
y_t(\xi,t)&=U(\xi,t),\\
U_t(\xi,t)& = \frac12 H(\xi,t)-\frac14 \lim_{\xi\to\infty}H(\xi,t),\\
H_t(\xi,t)& =0,\\
r_t(\xi,t)&=0
\end{align}
\end{subequations}
with initial data 
\begin{subequations}\label{ini:sm}
\begin{align}
y(\xi,0)& =y_0(\xi),\\
U(\xi,0)& =u_0(y(\xi,0))=u_0(y_0(\xi)),\\
H(\xi,0)&=\int_{-\infty}^{y(\xi,0)}\left(u_{0,x}^2(z)+\rho_0^2(z)\right)\:\d z=\int_{-\infty}^{y_0(\xi)} \left(u_{0,x}^2(z)+\rho_0^2(z)\right)\:\d z,\\
r(\xi,0)& = \rho_0(y(\xi,0))y_{\xi}(\xi,0)=\rho_0(y_0(\xi))y_{0,\xi}(\xi).
\end{align}
\end{subequations}
This system coincides with the one for conservative solutions and is valid until wave breaking occurs for the first time and energy concentrates on sets of measure zero. This happens when $y_\xi(\xi,t)=0$ for some $\xi$. To study this phenomenon in detail we differentiate the above system with respect to $\xi$ and get for each $\xi\in \R$ the following closed system of linear ordinary differential equations
\begin{subequations}
\label{eq:sm diff}
\begin{align}
y_{\xi,t}(\xi,t)& = U_\xi(\xi,t)\\
U_{\xi,t}(\xi,t)& = \frac12 H_\xi(\xi,t)\\
H_{\xi,t}(\xi,t)& = 0,
\end{align}
\end{subequations}
which can be solved explicitly. In particular, one has 
\begin{equation}
y_\xi(\xi,t)=y_{0,\xi}(\xi)+U_{0,\xi}(\xi)t+\frac14 H_{0,\xi}(\xi)t^2.
\end{equation}
Since $y_{0}(\xi)$ is strictly increasing and both $y_{0,\xi}(\xi)$ and $H_{0,\xi}(\xi)$ are positive for all $\xi\in\R$, wave breaking can only occur when $U_{0,\xi}(\xi)<0$. A closer look reveals that, if wave breaking occurs the breaking time $\tau(\xi)$ is given by 
\begin{equation}
\tau(\xi)=2\frac{-U_{0,\xi}(\xi)\pm \sqrt{U_{0,\xi}(\xi)^2-y_{0,\xi}(\xi)H_{0,\xi}(\xi)}}{H_{0,\xi}(\xi)}.
\end{equation}
and since wave breaking can occur only once, we must have 
\begin{equation}\label{comb:yUHr}
y_{0,\xi}(\xi)H_{0,\xi}(\xi)=U_{0,\xi}(\xi)^2.
\end{equation}
Comparing \eqref{comb:yUHr} with \eqref{ini:sm} yields, in addition, that wave breaking will occur at all points $\xi\in \R$ such that $U_{0,\xi}(\xi)<0$ and $r_0(\xi)=0$. The corresponding wave breaking time $\tau(\xi)$ is given by 
\begin{equation}
\tau(\xi)=-2\frac{U_{0,\xi}(\xi)}{H_{0,\xi}(\xi)}=-2\frac{y_{0,\xi}(\xi)}{U_{0,\xi}(\xi)}.
\end{equation}
Furthermore, it should be noted that 
\begin{equation}
\label{eq:lim U_xi y_xi}
U_\xi(\xi,t)\uparrow 0 \quad \text{ and } \quad y_\xi(\xi,t)\downarrow 0 \quad \text{as }t\to \tau(\xi)-.
\end{equation}

We can now turn to the continuation of the solution past the first time wave breaking occurs. If wave breaking occurs at a point $(\xi,t)$ in Lagrangian coordinates then it occurs in Eulerian coordinates at the point $(y(\xi,t),t)$, Moreover, the energy density at any point $(\xi,t)$ in the conservative case is given by $H_\xi(\xi,t)$. Thus taking out an $\alpha$-fraction of the energy concentrated at a point $(\xi, \tau(\xi))$ corresponds to replacing $H_\xi(\xi,t)$ by 
\begin{equation}
V_\xi(\xi,t)=\begin{cases}H_{0,\xi}(\xi), & \quad \text{ if }t<\tau(\xi),\\
(1-\alpha(y(\xi, \tau(\xi)))H_{0,\xi}(\xi), & \quad \text{ if } t\geq \tau(\xi),
\end{cases}
\end{equation}
on the right hand side of \eqref{lag:sm}

Before defining $\alpha$-dissipative solutions in Lagrangian coordinates rigorously, it is important to note that our introduction of the solution concept is heuristic, in that sense that we assumed that the initial data is smooth and no wave breaking occurs at time $t=0$. This limitations will be overcome in the next section, where we focus on the interplay between Eulerian and Lagrangian coordinates.

\vspace{0.2cm}
Let $E_1$ be the Banach space defined by
\begin{equation}
E_1 = \{ f\in L^{\infty}(\R) \mid f'\in L^2(\R) \text{ and }\lim_{x\rightarrow -\infty}f(x) = 0\},
\end{equation}
equipped with the norm $\|f\|_{E_1} = \|f\|_{\infty}+\|f'\|_{2}$, and $E_2$ the Banach space given by
\begin{equation}
E_2 = \{ f\in L^{\infty}(\R) \mid f'\in L^2(\R)\},
\end{equation}
equipped with the norm $\|f\|_{E_2} = \|f\|_{\infty}+\|f'\|_{2}$. 

In addition, let $B$ be the Banach space $B = E_2\times E_2\times E_1\times L^2(\R)\times E_1$ with the norm
\begin{equation}
\|(f_1,f_2,f_3,f_4,f_5)\|_B = \|f_1\|_{E_2} + \|f_2\|_{E_2} + \|f_3\|_{E_1} + \|f_4\|_2 + \|f_5\|_{E_1}.
\end{equation}
Then the set of Lagrangian coordinates $\mathcal{F}^\alpha$ is given as follows.

\begin{definition}[Lagrangian coordinates]
\label{def_F}
The set $\mathcal{F}^\alpha$ consists of all \hspace{2cm} \phantom{XX} $X=(y,U,H,r,V)$, such that $(y-\Id,U,H,r,V) \in B$, and 
\begin{align*}
(i)&\;\;\; y-\Id,U,H,V \in W^{1,\infty}(\R), r \in L^{\infty}(\R), \\
(ii)&\;\;\; 0\leq y_{\xi}, 0\leq H_{\xi}, 0<c< H_{\xi}+y_{\xi} \;\;\;\text{ almost everwhere},\\
(iii)&\;\;\; y_{\xi}V_{\xi} = U_{\xi}^2 + r^2\;\;\;\text{ almost everywhere},\\
(iv)&\;\;\; 0\leq V_{\xi} \leq H_{\xi} \;\;\;\text{ almost everywhere},\\
(v)& \;\;\; \text{If } 0\leq\alpha(x)<1 \text{ for all } x\in\R, \nonumber\\
	&\;\;\;\text{then there exists }\kappa:\R\rightarrow(0,1] \text{ such that } V_\xi(\xi)=\kappa(y(\xi))H_\xi(\xi) \text{ almost}\nonumber\\ &\;\;\; \text{everywhere}, \text{ and } \kappa(y(\xi)) = 1 \text{ whenever } U_\xi(\xi)<0 \text{ or } r(\xi) \neq 0,\\
(vi)& \;\;\; \text{If } \alpha(x)=1 \text{ for all } x\in \R,  \text{ then } y_\xi(\xi)=0 \text{ implies  } V_\xi(\xi)=0,\\
	& \;\;\; \text{and } y_\xi(\xi)>0 \text{ implies that } V_\xi(\xi) = H_\xi(\xi) \text{ almost everywhere}.
\end{align*}
\end{definition}

The $\alpha$-dissipative solution $X(t)\in \F^\alpha$ with initial data $X(0)=X_0\in \F^\alpha$ to \eqref{hunter-saxton-system} in Lagrangian coordinates is then given by
\begin{subequations}
\label{characteristic_system_a_dissipative}
\begin{align}
y_t(\xi,t) &= U(\xi,t),\\
U_t(\xi,t) &= \frac 12V(\xi,t)-\frac 14\lim_{\xi\to\infty}V(\xi,t),\\
H_t(\xi,t) &= 0,\\
r_t(\xi,t) &= 0,
\end{align}
\end{subequations}
where 
\begin{equation}\label{def:V}
V(\xi,t) = \int_{-\infty}^{\xi}H_{0,\xi}(\eta)\big(1-\mathbf{1}_{\{t\geq\tau(\eta)\}}\alpha(y(\eta,\tau(\eta)))\big)\:d\eta.
\end{equation}
Here the blow-up time $\tau(\xi)$ is given by 
\begin{equation}\label{def:tau}
\tau(\xi)=\begin{cases}
0, &\quad  \text{if } y_{0,\xi}(\xi)=0,\\
-\frac{2U_{0,\xi}(\xi)}{H_{0,\xi}(\xi)}, &\quad  \text {if } r_0(\xi)=0 \text{ and } U_{0,\xi}(\xi)< 0,\\
\infty, & \quad \text{otherwise.}
\end{cases}
\end{equation}
Note that $\tau(\xi)=0$ for some $\xi\in \R$, implies that $y_{0,\xi} (\xi)= 0$ for some $\xi\in\R$ and hence wave breaking occurs at time $t=0$. 

In \eqref{def:V} it is implicitly assumed that $V_{0,\xi}(\xi) = H_{0,\xi}(\xi)$ for a.e. $\xi\in\R$. This is a stricter condition than $(iv)$ in Definition \ref{def_F}. Thus we define the set of admissible initial data, $\F_i^\alpha$ as follows.
\begin{definition}
\label{def:F_init}
Let $\F_i^\alpha$ be the set of all $X\in\F^\alpha$ such that
\begin{equation}
V(\xi) = H(\xi) \quad \text{ for all }\xi\in\R.
\end{equation}
\end{definition}

Here, it should be pointed out that all of Section \ref{section:exist lagrange} and \ref{section:exist euler} remain true for initial data in $\F^\alpha$. The reason for this restriction will be clarified in Section \ref{section lipschitz metric}. 

We now turn our attention to the existence and uniqueness of solutions to the system \eqref{characteristic_system_a_dissipative} with initial data in $\F_i^\alpha$.

\begin{lemma}
\label{solution_operator_lagrangian}
To any initial data $X_0\in \F_i^\alpha$ there exists a unique solution $X\in C(\R_+, \F^\alpha)$ of \eqref{characteristic_system_a_dissipative} with $X(0)=X_0$.
\end{lemma}

\begin{proof} Let $X_0\in \F_i^\alpha$.
First we prove existence and uniqueness of solutions to \eqref{characteristic_system_a_dissipative} in $B$. We are going to construct a sequence of approximate solutions $X_n(t)$ that converge to the $\alpha$-dissipative solution $X(t)$ which satisfies $X(0)=X_0$. To that end define $X_1(t)=X_0$ for all $t\geq 0$. For $n\geq 2$, we take advantage of the fact that for each $\xi\in\R$ the breaking time $\tau(\xi)$ can be computed from the initial data, see \eqref{def:tau}. Indeed we define, for $n+1\geq 2$, 
\begin{equation}
\label{approx alpha}
\beta_{n+1}(\xi,t) =
	\begin{cases}
	0, &\tau(\xi) = \infty,\\
	\alpha\big(y_n(\xi,t)\big), &t<\tau(\xi)<\infty,\\
	\alpha\big(y_n(\xi,\tau(\xi))\big), &\tau(\xi)\leq t,
	\end{cases}
\end{equation}
and $X_{n+1}(t)$ as the solution to 
\begin{subequations}
\label{approx X}
\begin{align}
y_{{n+1},t}(\xi,t) &= U_{n+1}(\xi,t),\\
U_{n+1,t}(\xi,t) &= \frac 12V_{n+1}(\xi,t)-\frac 14V_{n+1,\infty}(t),\\
H_{n+1,t}(\xi,t) &= 0,\\
r_{n+1,t}(\xi,t) &= 0,\\
V_{n+1}(\xi,t) &= \int_{-\infty}^{\xi}H_{0,\xi}(\eta)(1-\mathbf{1}_{\{t\geq\tau(\eta)\}}\beta_{n+1}(\eta,t))\:d\eta,
\end{align}
\end{subequations}
where $V_{n+1,\infty}(t)=\lim_{\xi\to\infty}V_{n+1}(\xi,t)$, with initial data $X_{n+1}(0)=X_0$. That $X_n(t)\in B$ for all $t\geq 0$ and $n\geq1$ follows inductively by integrating the system.

To establish the existence of a limiting function $X(t)$ we use a fix point argument. First, estimate the difference between $y_n(t)$ and $y_{n+1}(t)$ for $0\leq t\leq T$. From \eqref{approx alpha} we get that
\begin{equation}\label{diff:beta}
|\beta_{n+1}(\xi,t)-\beta_n(\xi,t)| \leq \|\alpha'\|_{\infty}\sup_{t\in[0,T]}\|y_n(t)-y_{n-1}(t)\|_\infty,
\end{equation}
 which implies together with \eqref{approx X} that
\begin{subequations}
\begin{align}
\label{diff:V}
|V_{n+1}(\xi,t)-V_n(\xi,t)| & \leq \int_{-\infty}^{\infty}H_{0,\xi}(\eta)|\beta_{n+1}(\eta,t)-\beta_n(\eta,t)|\mathbf{1}_{\{t\geq\tau(\eta)\}}\:d\eta\\ \nonumber
	&\leq  \|H_0\|_{\infty}\|\alpha'\|_{\infty}\sup_{t\in[0,T]}\|y_n(t)-y_{n-1}(t)\|_\infty,\\ 
\label{diff:U}
|U_{n+1}(\xi,t)-U_n(\xi,t)| 
	&\leq \frac 14 \|H_0\|_{\infty}\|\alpha'\|_{\infty}\sup_{t\in[0,T]}\|y_n(t)-y_{n-1}(t)\|_\infty t,\\ 
	\label{diff:y}
|y_{n+1}(\xi,t)-y_n(\xi,t)| &\leq \frac 18 \|H_{0}\|_{\infty}\|\alpha'\|_{\infty}\sup_{t\in[0,T]}\|y_n(t)-y_{n-1}(t)\|_\infty t^2.
\end{align}
\end{subequations}
Thus,
\begin{equation*}
\sup_{t\in[0,T]}\|y_{n+1}(t)-y_{n}(t)\|_\infty \leq \frac 18 \|H_0\|_{\infty}\|\alpha'\|_{\infty}\sup_{t\in[0,T]}\|y_{n}(t)-y_{n-1}(t)\|_\infty T^2,
\end{equation*}
and if $T$ is chosen such that $T <\sqrt{\frac{8}{\|\alpha'\|_{\infty}\|H_0\|_{\infty}}}$, we have a contraction on the \mbox{Banach} space $C([0,T], L^\infty(\R))$. In particular, there exists a unique function $y\in C([0,T], L^\infty(\R))$ such that $y_n(\cdot, t)\rightarrow y(\cdot, t)$ in $L^{\infty}(\R)$ for all $t\in [0,T]$. Moreover, \eqref{diff:beta}, \eqref{diff:V}, and \eqref{diff:U} imply that there exist unique $\beta$, $V$, and $U$  in $C([0,T], L^\infty(\R))$, such that $\beta_n(\cdot,t)\rightarrow \beta(\cdot,t)$, $V_n(\cdot,t)\to V(\cdot,t)$, and $U_n(\cdot, t)\rightarrow U(\cdot,t)$ in $L^{\infty}(\R)$ for all $t\in [0,T]$ as well. As far as $r_n(\xi,t)$ and $H_n(\xi,t)$ are concerned, observe that  $r_n(\xi,t)=r_0(\xi)$ and $H_n(\xi,t)=H_0(\xi)$ for all $n\geq 1$ and hence $r_n(t,\cdot)\to r_0(\cdot)$ in $L^2(\R)\cap L^\infty(\R)$ and $H_n(\cdot,t)\to H_0(\cdot)$ in $L^\infty(\R)$ for all $t\in \R$.

The system for the derivatives is given by
\begin{subequations}
\begin{align}
y_{n+1,\xi, t}(\xi,t) &= U_{n+1,\xi}(\xi,t),\\
U_{n+1,\xi, t}(\xi,t) &= \frac 12(1-\beta_{n+1}(\xi,t)\mathbf{1}_{\{t\geq\tau(\xi)\}})H_{0,\xi}(\xi),
\end{align}
\end{subequations}
and hence
\begin{subequations}
\begin{align}
|U_{n+1,\xi}(\xi,t)-U_{n,\xi}(\xi,t)|& \leq \frac12 t |\beta_{n+1}(\xi,t)-\beta_n(\xi,t)|H_{0,\xi}(\xi)\\
& \leq \frac12 t H_{0,\xi} \|\alpha'\|_\infty \sup_{t\in[0,T]} \|y_n(t)-y_{n-1}(t)\|_\infty,\nonumber \\
|y_{n+1,\xi}(\xi,t)-y_{n,\xi}(\xi,t)| & \leq \frac 14H_{0,\xi}t^2\|\alpha'\|_{\infty}\sup_{t\in[0,T]}\|y_{n}(t)-y_{n-1}(t)\|_\infty.
\end{align}
\end{subequations}
Since $H_{0,\xi}(\cdot)\in L^2(\R)\cap L^\infty(\R)$ it follows that $U_{n,\xi}(\cdot,t)\to U_{\xi}(\cdot,t)$ and $(y_{n,\xi}-1)(\cdot,t)\to (y_\xi-1)(\cdot,t)$ in $L^2(\R)\cap L^\infty(\R)$ for all $t\in [0,T]$. Thus we showed that to each $X_0\in \F^\alpha_i$ there exists a local solution $X(t)$ to \eqref{characteristic_system_a_dissipative} in $B$. In addition, we obtained that both $(y_\xi-1)(\cdot, t)$, $U_\xi(\cdot,t)$, and $r(\cdot,t)$ belong to $L^\infty(\R)$ for all $t\in [0,T]$.

The uniqueness of the solution $X(t)$ to \eqref{characteristic_system_a_dissipative} with $X(0)=X_0\in \F^\alpha_i$ on the interval $[0,T]$ follows by a standard contraction argument, based on estimates similar to the ones established so far. 

To extend $X(t)$ from a local to a global solution observe that the upper bound on $T$ only depends on $\| H_0\|_\infty$ and $\|\alpha '\|_\infty$, which are both independent of time. Thus one can repeat the above construction, with slight modifications, with initial data $X(\frac12 T)$ to obtain the unique solution on $[0, \frac32 T]$.  Continuing inductively finally yields the unique, global solution $X(t)\in B$ to \eqref{characteristic_system_a_dissipative}. 

The method is illustrated in Example \ref{example iteration} where the iteration \eqref{approx X} is applied to some multipeakon initial data.

It is left to show that $X(t)$ satisfies properties $(ii)$--$(vi)$ in Definition~\ref{def_F}.
By definition, see \eqref{def:V}, we have that 
\begin{equation}
V_\xi(\xi,t) =
	\begin{cases}
	H_{0,\xi}(\xi), &\text{if }t<\tau(\xi),\\
	\left(1-\alpha(y(\xi,\tau(\xi)))\right)H_{0,\xi}(\xi), &\text{if } t\geq \tau(\xi),
	\end{cases}
\end{equation}
and hence properties $(iv)$--$(vi)$ in Definition~\ref{def_F} hold since wave breaking happens when $U_\xi(\xi,t) = 0$ and $r(\xi,t) = 0$. It should be pointed out that in the case $\alpha=1$ we have that if $y_\xi(\xi,t) = 0$, then $y_\xi(\xi,t')= 0$ for all $t'>t$.

Direct calculations yield for a.e. $\xi\in\R$, that 
\begin{align}
\frac{d}{d t} \frac{1}{y_{\xi}(\xi,t)+H_{0,\xi}(\xi)} &= -\frac{U_\xi(\xi,t)}{(y_{\xi}(\xi,t)+H_{0,\xi}(\xi))^2}\nonumber \\ 
& \leq \frac{\sqrt{y_{\xi}V_{\xi}(\xi,t)}}{(y_{\xi}(\xi,t)+H_{0,\xi}(\xi))^2}
	\leq \frac 12\frac{1}{y_{\xi}(\xi,t)+H_{0,\xi}(\xi)},
\end{align}
Applying Gronwall's inequality, we obtain 
\begin{equation}
\label{eq:gronwall}
(y_\xi(\xi,0)+H_{0,\xi}(\xi))e^{-\frac12 t}\leq (y_\xi(\xi,t)+H_{0,\xi}(\xi))
\end{equation}
for a.e. $\xi\in\R$, i.e. property $(ii)$ in Definition~\ref{def_F} .

Thus it is left to show that $y_\xi(\xi,t) V_\xi(\xi,t)=U_\xi^2(\xi,t)+r^2(\xi,t)$ for a.e. $\xi\in\R$ and $t>0$. For a.e. $\xi\in\R$ we have for $t<\tau(\xi)$,
\begin{align*}
U_{\xi}(\xi, t)^2 =&\: \frac{1}{4}H_{0,\xi}^2(\xi)t^2 +U_{0,\xi}(\xi)H_{0,\xi}(\xi)t + U_{0,\xi}(\xi)^2, \\
r(\xi,t)^2 =& r_0^2(\xi), \\
y_{\xi}V_{\xi}(\xi,t) =&\: \frac{1}{4}H_{0,\xi}^2(\xi)t^2 +U_{0,\xi}H_{0,\xi}(\xi)t + U_{0,\xi}^2 (\xi)+ r_0^2(\xi),
\end{align*}
where we used that $V_{0,\xi}y_{0,\xi}(\xi)=U_{0,\xi}^2(\xi)+r_0^2(\xi)$ for a.e. $\xi\in\R$. For $t\geq \tau(\xi)$, we have 
\begin{align*}
r_t(t,\xi)& =0,\\
U_\xi(\xi,t)& =\frac12 V_\xi(\xi,\tau(\xi))(t-\tau(\xi)),\\
y_\xi(\xi,t)& = \frac 14 V_\xi(\xi,\tau(\xi))(t-\tau(\xi))^2,
\end{align*}
where we used that $U_\xi(\xi,\tau(\xi))=0=y_\xi(\xi,\tau(\xi))$. Thus, $y_\xi V_\xi(\xi,t)=U_\xi^2(\xi,t)+r^2(\xi,t)$ for a.e $\xi\in\R$ also in this case and Definition~\ref{def_F} $(iii)$ holds.
\end{proof}

A natural question, that is going to play a major role when establishing the Lipschitz stability, is if one can estimate the size of the set where wave breaking has already taken place in Lagrangian coordinates.  

\begin{corollary}
\label{cor:size of broken set}
Let $X_0\in\F_i^\alpha$, and denote by $X(t)$ the solution to \eqref{characteristic_system_a_dissipative} with initial data $X_0$. Then the set of all points where wave breaking has taken place within the time interval $[0,t]$ satisfies
\begin{equation}
\label{size of broken set F}
\mathrm m\left(\{\xi\mid \tau(\xi)\leq t\}\right) \leq \frac{\left(1+\frac 14t^2\right)}{c}\|H_{0}\|_\infty,
\end{equation}
where $c$ is given by $0<c\leq y_{0,\xi}(\xi)+H_{0,\xi}(\xi)$ for a.e. $\xi\in\R$ (cf. Definition \ref{def_F}).
\end{corollary}

\begin{proof}
Given $\xi\in\R$ such that $\tau(\xi)\leq t$, then \eqref{characteristic_system_a_dissipative} implies that 
\begin{equation}
0=y_\xi(\xi,\tau(\xi))=y_{0,\xi}(\xi)+U_{0,\xi}(\xi)\tau(\xi)+\frac14 H_{0,\xi}(\xi)\tau(\xi)^2
\end{equation}
and 
\begin{equation}
0=U_\xi(\xi, \tau(\xi))=U_{0,\xi}(\xi)+\frac12 H_{0,\xi}(\xi)\tau(\xi).
\end{equation}
Combining these two equations finally yields $0=y_{0,\xi}(\xi)-\frac14 H_{0,\xi}(\xi)\tau(\xi)^2$ and hence  
\begin{equation}
\label{eq:H_0xi tau}
1= \frac{H_{0,\xi}(\xi)(1+\frac14 \tau(\xi)^2)}{y_{0,\xi}(\xi)+H_{0,\xi}(\xi)}\leq\frac{H_{0,\xi}(\xi)(1+\frac14 \tau(\xi)^2)}{c} ,
\end{equation}
where we used that there exists $c>0$ such that $y_{0,\xi}(\xi)+H_{0,\xi}(\xi)\geq c>0$ for a.e. $\xi\in\R$.
Thus one obtains
\begin{align}
\mathrm m\left(\{\xi\mid \tau(\xi)=t\}\right)& \leq \mathrm m\left(\{\xi\mid  \tau(\xi)\leq t \}\right) 
          = \int_\R\mathbf{1}_{\{\tau(\xi)\leq t\}}\:d\xi\nonumber \\
	&\leq \frac{1+\frac 14t^2}{c}\int_\R H_{0,\xi}(\xi)\:d\xi=\frac{1+\frac14 t^2}{c}\|H_0\|_\infty.
\end{align}
\end{proof}

Note, that if $y_{0\xi}(\xi)+H_{0,\xi} (\xi)= 1$ for all $\xi\in\R$, the above estimate reads
\begin{equation}
\label{size of broken set F_0^M}
\mathrm m\left(\{\xi\mid \tau(\xi)\leq t\}\right) \leq \left(1+\frac 14t^2\right)\|H_{0}\|_{\infty}.
\end{equation}

\section{Existence of solutions in Eulerian coordinates}\label{section:exist euler}

In Section \ref{section:exist lagrange} $\alpha$-dissipative solutions of the reformulation of the 2HS system in Lagrangian coordinates, \eqref{characteristic_system_a_dissipative}, have been constructed. The goal of this section is to establish the connection between Eulerian and Lagrangian coordinates, which will enable us to define $\alpha$-dissipative solutions in Eulerian coordinates. The strategy is to map the initial data in Eulerian coordinates to the space $\F^\alpha$, then solve \eqref{characteristic_system_a_dissipative}, and map the solution back to Eulerian coordinates. In order to do so the space $\D^\alpha$, of solutions in Eulerian coordinates, will be defined, and mappings between $\F^\alpha$ and $\D^\alpha$ will be established. A definition of $\alpha$-dissipative solutions of \eqref{hunter-saxton-system} will be given and the proposed solutions will be shown to satisfy this definition. 

The solution space in Eulerian coordinates is given in the next definition.
\begin{definition}[Eulerian coordinates]
\label{def D}
Let $\D^\alpha$ be the set of all $(u,\rho,\nu,\mu)$ such that
\begin{align*}
(i)&\;\;\; u\in L^{\infty}(\R), u_x\in L^2(\R),\\
(ii)&\;\;\; \rho\in L^2(\R),\\
(iii)&\;\;\; \mu\leq\nu \in \mathcal{M}^+(\R),\\
(iv)&\;\;\; \mu_{ac} = (u_x^2+\rho^2)\:d x,\\
(v)&\;\;\; \text{if } 0\leq\alpha(x)<1 \text{ for all } x\in\R, \text{ then } \frac{d\mu}{d\nu}>0,\\
	&\;\;\; \text{and } \frac{d\mu}{d\nu}(x) = 1 \text{ whenever } u_x(x)<0 \text{ or }\rho(x)\neq 0,\\
(vi)& \;\;\; \text{if } \alpha(x)=1 \text{ for all } x\in \R, \text{ then } \mu=\nu_{ac},
\end{align*}
where $\mathcal{M}^+(\R)$ denotes the set of positive, finite Radon measures on $\R$.
\end{definition}

Next, the mappings between the elements in $\D^\alpha$ and the elements of $\F^\alpha$ are introduced.  
When the mapping $L:\D^\alpha\rightarrow\F^\alpha$ is to be defined, the critical part is how to define the characteristics $y$ in nice way, since all the other variables will follow naturally.

\begin{definition}
Let the mapping $L:\D^\alpha\rightarrow\F^\alpha$ be defined by $L((u,\rho,\nu,\mu)) = (y,U,H,r,V)$ where
\begin{subequations}
\label{definition L}
\begin{align}
y(\xi) &= \sup\{x\mid\nu((-\infty,x))+x < \xi\},\\
H(\xi) &= \xi - y(\xi),\\
U(\xi) &= u \circ y(\xi),\\
r(\xi) &= (\rho\circ y(\xi))y_{\xi}(\xi),\\
V(\xi) &= \int_{-\infty}^{\xi}\frac{d\mu}{d\nu}\circ y(\eta)H_{\xi}(\eta)\:d\eta.
\end{align}
\end{subequations}
\end{definition}

When going from $\F^\alpha$ to $\D^\alpha$ one has to be careful, since the presence of wave breaking implies that there can be $\xi$ such that $y_\xi(\xi) = 0$. However, $y_\xi(\xi) = 0$ implies that $U_\xi(\xi) = 0$ and $r(\xi) = 0$ by property $(iii)$ in Definition \ref{def_F} and hence one can define the mapping $M:\F^\alpha\rightarrow\D^\alpha$ as follows. 

\begin{definition}
\label{definition M}
Let the mapping $M:\F^\alpha\rightarrow\D^\alpha$ be defined by $M((y,U,H,r,V)) = (u,\rho,\nu,\mu)$ where
\begin{subequations}
\begin{align}
u(x) &= U(\xi), \text{ for any }\xi \text{ such that } x = y(\xi),\\
\rho\:d x &= y_{\#}(r\:d\xi),\\
\nu &= y_{\#}(H_{\xi}\:d\xi),\\
\mu &= y_{\#}(V_{\xi}\:d\xi).
\end{align}
\end{subequations}
\end{definition}

That $L$ and $M$ are well-defined follows from combining the proofs of \cite[Proposition 2.6 and 2.9]{N}, \cite[Definitions 3.9 and 3.10, Theorem 3.11]{GHR15}, and \cite[Theorem 3.8 and 3.11]{HR07}, where slightly different mappings have been considered. We therefore leave the details to the interested reader.

There are five variables in Lagrangian coordinates, while in Eulerian coordinates there are only four. 
Hence there cannot be a one-to-one correspondence. 
Indeed, there is an equivalence relation $\sim$ on $\F^\alpha$ such that $X\sim\bar X$ implies that $M(X) = M(\bar X)$. This equivalence relation is induced by a group action $\bullet$ by a group $G$ given in the next definition.

\begin{definition}
\label{definition_G}
Let the group $G$ and the group action $\bullet$ of $G$ on $\F^\alpha$ be defined as follows.
\begin{itemize}
\item[(i)] Define $G$ as the group of homeomorphisms $f:\R\rightarrow\R$ such that $$f - \Id \in W^{1,\infty}(\R)\text{, } f^{-1}-\Id \in W^{1,\infty}(\R)\text{, and } f_{\xi}-1\in L^2(\R).$$
\item[(ii)] Define a group action $\bullet:\F^\alpha\times G \rightarrow \F^\alpha$ by $$(X,f)\mapsto (y\circ f, U\circ f, H\circ f, (r\circ f) \cdot f_\xi, V\circ f)=X\bullet f.$$
\end{itemize}
We will call the action of $f\in G$ on $X\in\F^\alpha$ the relabeling of $X$ by $f$.
\end{definition}

That the group action $\bullet$ is well-defined follows the same lines as \cite[Proposition 3.4]{HR07} and is left to the interested reader.

As indicated above $M$ respects equivalence classes identified via relabeling. 
\begin{proposition}
\label{M and f}
We have for any $X\in\F^\alpha$ and $f\in G$ that
\begin{equation}
M(X\bullet f) = M(X).
\end{equation}
\end{proposition}
\begin{proof}
The proof is a combination of the ones of \cite[Theorem 3.11]{HR07}, \cite[Proposition 2.12]{N}, and \cite[Definition 3.19]{GHR15}.
\end{proof}

Let us look if the time evolution respects relabeling. To that end define $S_t$ the solution operators that maps any initial data $X_0\in\F_i^\alpha$ to its unique $\alpha$-dissipative solution $X(t)$ at time $t$.
\begin{definition}
\label{def:S_t}
For any $t\geq 0$ and $X_0\in\F_i^\alpha$ define $S_t(X_0)$ to be the unique $\alpha$-dissipative solution $X(t)$ at time $t$ that satisfies $X(0)=X_0$.
\end{definition}
For our solution concept to be well-defined we have to prove that $S_t$ commutes with relabeling.

\begin{proposition}
\label{Semigroup commutes with relabeling}
For any $X_0\in\F_i^\alpha$, $f\in G$ it holds that $S_t(X_0\bullet f) = S_t(X_0)\bullet f$.
\end{proposition}
\begin{proof}
Let $f\in G$ and $X_0\in \F_i^\alpha$ be given and denote by $X(t)$ the corresponding $\alpha$-dissipative solution, i.e. $X(t)=S_t(X_0)$. According to Definition~\ref{definition_G}, $\hat X_0=X_0\bullet f\in \F_i^\alpha$ and the corresponding solution is given by $\hat X(t)=S_t(\hat X_0)=S_t(X_0\bullet f)$. Thus we have to show that $\hat X(t)=X(t)\bullet f$ for all $t\geq 0$. At initial time $t=0$ we have $\hat X(0)=X_0\bullet f=X(0)\bullet f$. For $t>0$ recall that both $X(t)$ and $\hat X(t)$ can be seen as the limit of sequences $X_n(t)$ and $\hat X_n(t)$, respectively, as defined in the proof of Lemma~\ref{solution_operator_lagrangian}. In particular one has that
\begin{equation}\label{ind:aim}
\hat X(t)=\lim_{n\to\infty}\hat X_n(t) \quad \text{ and } \quad X(t)\bullet f=\lim_{n\to\infty} (X_n(t)\bullet f).
\end{equation}
Thus it suffices to show by induction that $\hat X_n(t)=X_n(t)\bullet f$ for all $n\in \mathbb{N}$ and all $t\geq 0$. First of all note that
\begin{equation}\label{ind:vor2}
\hat \tau(\xi)=\tau(f(\xi))\quad \text{ for all }\xi
\end{equation}
 and
by definition $\hat X_1(t)=\hat X(0)=X_0\bullet f= X_1(t)\bullet f$ for all $t\geq 0$. For $n=2$ direct calculations yield that $\hat X_2(t)=X_2(t)\bullet f$. To conclude the argument, assume that
\begin{equation}\label{ind:vor}
\hat X_n(t)=X_n(t)\bullet f \quad \text{for all } t\geq 0.
\end{equation}
 Combining \eqref{approx alpha}, \eqref{ind:vor2} and \eqref{ind:vor} yields
\begin{equation}
\hat \beta_{n+1}(\xi,t)=\beta_{n+1}(f(\xi), t) \quad \text{ for all }t \geq 0,
\end{equation}
and subsequently
\begin{align}
\hat V_{n+1}(\xi,t)& =\int_{-\infty}^\xi (1-\hat \beta_{n+1}(\eta,t)\mathbf{1}_{\{t\geq\hat\tau(\eta)\}})\hat H_{0,\xi}(\eta)d\eta\\
& = \int_{-\infty}^\xi (1-\beta_{n+1}(f(\eta),t)\mathbf{1}_{\{t\geq \tau(f(\eta))\}}) H_{0,\xi}(f(\eta))f_\xi(\eta)d\eta\\
& = \int_{-\infty}^{f(\xi)} (1-\beta_{n+1}(\eta,t)\mathbf{1}_{\{t\geq \tau(\eta)\}})H_{0,\xi}(\eta)d\eta= V_{n+1}(f(\xi),t).
\end{align}
Thus \eqref{approx X} implies that $\hat X_n(t)=X_n(t)\bullet f$ for all $n\in\mathbb{N}$ and all $t\geq 0$.
\end{proof}

To show that there is a one-to-one correspondence between $\D^\alpha$ and equivalence classes in $\F^\alpha$, we need to identify a subset of $\F^\alpha$ containing one representative of each equivalence class.  
To that end define $\F_0^\alpha$ by
\begin{equation}
\label{eq:def F_0}
\F_0^\alpha =\left\{X\in\F^\alpha\mid y + H = \Id\right\}.
\end{equation}
The following proposition states that the mappings $M$ and $L$ are inverses of each other when $M$ is restricted to $\F_0^\alpha$.

\begin{proposition}
\label{prop: M and L}
The mappings $M$ and $L$ satisfy
\begin{align}
M\circ L &= \Id_{\D^\alpha},\\
L\circ M &= \Id_{\F_0^\alpha}.
\end{align}
\end{proposition}
\begin{proof}
See \cite[Theorem 3.11]{GHR15} and \cite[Theorem 3.12]{HR07}.
\end{proof}

Observe that $L$ maps $\D^\alpha$ to $\F_0^\alpha$. In addition, any initial data  $(u_0,\rho_0, \mu_0,\mu_0)$ is mapped to an element in $\F_{i,0}^\alpha=\F_0^\alpha\cap \F_i^\alpha$, which is advantageous because, among other things, \eqref{size of broken set F_0^M} applies in this setting.  

If $\bullet$ is to induce the desired equivalence relation, then to each $X\in\F^\alpha$ there must exist $f\in G$ such that $X\bullet f\in\F_0^\alpha$.

\begin{definition}
Define the map $\Pi:\F^\alpha\rightarrow\F_0^\alpha$ by 
\begin{equation}
\Pi X = X\bullet (y+H)^{-1}, \quad X\in\F^\alpha.
\end{equation}
To ease the notation we write $\Pi X$, despite the fact that $\Pi$ is not a linear operator.
\end{definition}

For the map $\Pi$ to be well-defined we need that $(y+H)^{-1}\in G$.

\begin{proposition}
\label{y plus H in G}
Let $X\in\F^\alpha$, then $y+H\in G$ and $(y+H)^{-1}\in G$. Moreover if $X_0\in\F_{i,0}^\alpha$, and $X(t)=S_t(X_0)$, then for all $t\geq 0$ we have that $e^{-\frac 12t}\leq y_{\xi}(\xi,t) + H_{\xi}(\xi,t)\leq \frac 14t^2+t+1$ for almost every $\xi\in\R$.
\end{proposition}
\begin{proof}
Let $y+H = f$, we show that both $f$ and $f^{-1}$ belong to $G$. From the definition of $\F^\alpha$ we have that $f-\Id\in W^{1,\infty}(\R)$ with $f_{\xi}-1\in L^2(\R)$, with $c\leq f_{\xi}\leq C$ for some positive numbers $c$ and $C$. Thus there exists a Lipschitz continuous inverse $f^{-1}$ such that $f^{-1}-\Id\in W^{1,\infty}(\R)$ and $(f^{-1})_\xi-1\in L^2(\R)$. Hence $f$, $f^{-1}\in G$. 

The lower bound on $y_{\xi}(\xi,t)+H_{\xi}(\xi,t)$ is given by \eqref{eq:gronwall}, while the upper bound is a result of \eqref{characteristic_system_a_dissipative} combined with $|y_{0,\xi}(\xi)|$, $|H_{0,\xi}(\xi)|$, and $|U_{0,\xi}(\xi)|$ all being less than or equal to $1$.
\end{proof}

The set of admissible initial data in Lagrangian coordinates $\F_i^\alpha$ corresponds to a set of admissible initial data $\D_0^\alpha = M(\F_i^\alpha)$ in Eulerian coordinates as already hinted before. We want to characterize the set $\D_0^\alpha$ in terms of conditions on the measures $\mu$ and $\nu$.

\begin{proposition}
Let $\D_0^\alpha = M(\F_i^\alpha)$. Then $(u,\rho,\nu,\mu)\in\D_0^\alpha$ if and only if
\begin{equation}
\nu = \mu.
\end{equation}
\end{proposition}
\begin{proof}
Assume that $X\in\F_i^\alpha$ and let $(u,\rho,\nu,\mu)=M(X)$. Then combining Definition~\ref{def:F_init} and Definition~\ref{definition M} yields
\begin{equation}
\nu = y_\#(H_\xi\:d\xi) = y_\#(V_\xi\:d\xi) = \mu.
\end{equation}

Let $(u,\rho,\mu,\nu)\in\D^\alpha$ such that $\mu=\nu$. Denoting $X=L((u,\rho,\nu,\mu))$, then we have $\frac{d\mu}{d\nu} = 1$ and subsequently
\begin{equation}
V(\xi) = \int_{-\infty}^\xi H_\xi(\eta)\:d\eta = H(\xi).
\end{equation}
That is $X\in \F_i^\alpha$
and since $M\circ L = \Id_{\D^\alpha}$ we have that $(u,\rho,\nu,\mu)\in\D_0^\alpha$.
\end{proof}

The set up is complete, but before defining $\alpha$-dissipative solutions in Eulerian coordinates with the help of the mappings $L$ and $M$, a remark on the properties imposed on $\alpha$.

\begin{remark}
We assume, see \eqref{cond:alpha}, that either $\alpha:\R\to [0,1)$ or $\alpha:\R\to 1$ instead of $\alpha:\R\rightarrow [0,1]$. The reason for this restriction is that in the latter case there exist solutions in Lagrangian coordinates for which Proposition~\ref{prop: M and L} is not satisfies at all times. This is the case if there exist
$t^\star\in [0,\infty)$ and $\xi_1$, $\xi_2$, $\xi_3\in\R$ with $\xi_1<\xi_2<\xi_3$ such that
\begin{itemize}
\item $y_\xi(\xi, t^\star)=0$ for all $\xi\in(\xi_1,\xi_2)$,
\item $V_\xi(\xi, t^\star)=0$ for all $\xi\in(\xi_1,\xi_3)$,
\item $V_\xi(\xi,t^\star)\not =0$ for all $\xi\in(\xi_3,\xi_2)$,
\end{itemize}
as illustrated in Example~\ref{example degenerate nu}.
\end{remark} 

We give the definition of $\alpha$-dissipative solutions of \eqref{hunter-saxton-system}.

\begin{definition}
\label{def: weak sol}
We say that $(u,\rho,\nu,\mu)$ is a weak solution of \eqref{hunter-saxton-system} with initial data $(u_0,\rho_0,\nu_0,\mu_0)\in\D_0^\alpha$ if
\begin{subequations}
\label{eq:regularity}
\begin{align} \label{eq:regularity:6}
u &\in C^{0,\frac 12}\left(\R\times[0,T],\R\right),\qquad \text{ for any finite } T\geq 0,\\ 
\rho\:d x &\in C_{weak *}\left([0,\infty),\mathcal M(\R)\right),\\ \label{eq:regularity:5}
\nu &\in C_{weak *}\left([0,\infty),\mathcal M^+(\R)\right),\\ \label{eq:regularity:d}
\big(u(t),\rho(t),\nu(t),\mu(t)\big) &\in \D^\alpha \quad\text{ for all } t\geq 0,\\ \label{eq:regularity:d2}
\left.\big(u,\rho,\nu,\mu\big)\right|_{t=0} &=(u_0,\rho_0,\nu_0,\mu_0),\\ \label{eq:regularity:d3}
\nu(t)(\R) &= \nu_0(\R)\quad \text{ for all } t\geq 0,
\end{align}
\end{subequations}
for each compactly supported test function $\phi\in C_0^\infty(\R\times [0,\infty))$,
\begin{subequations}
\label{eq: weak sol}
\begin{align}
\int_0^{\infty}\int_{\R}\big(u\phi_t+\frac{1}{2}u^2\phi_x+\frac 14\big(\int_{-\infty}^{x} d\mu- \int_{x}^{\infty} d\mu\big)\phi\big)\:d xd t
&= -\int_{\R} u_0\phi|_{t=0}\:d x,\\
\int_0^{\infty}\int_{\R}\rho\phi_t + (\rho u)\phi_x\:d xd t &= -\int_{\R}\rho_0\phi|_{t=0}\:d x,
\end{align}
and for each non-negative compactly supported test function $\phi\in C_0^\infty(\R\times [0,\infty))$,
\begin{equation}
\int_0^{\infty}\int_{\R}(\phi_t+u\phi_x)\:d\mu(t)d t \geq - \int_{\R} \phi|_{t=0}\:d\mu_0.
\end{equation}
\end{subequations}
If in addition for each $t\geq 0$ we have
\begin{subequations}
\label{eq: alpha weak}
\begin{align}
d\mu(t) &= d\mu(t)^-_{ac}+(1-\alpha(x))\:d\mu(x,t)^-_s,\\
\mu(s) &\overset{*}{\rightharpoonup} \mu(t) \text{ as } s\downarrow t,\\
\mu(s) &\overset{*}{\rightharpoonup} \mu(t)^- \text{ as } s\uparrow t,
\end{align}
\end{subequations}
we call $t\mapsto (u(t),\rho(t),\nu(t),\mu(t))$ an $\alpha$-dissipative weak solution.
\end{definition}

Now, $\alpha$-dissipative solutions can be constructed. The solution operator $T_t:\D_0^\alpha\times [0,\infty)\rightarrow\D^\alpha$ is then given by $T_t = M\circ S_t\circ L$.  

\begin{theorem}
\label{existence theorem}
To any initial data $(u_0,\rho_0,\nu_0,\mu_0)\in\D_0^\alpha$ there exists a globally defined $\alpha$-dissipative solution,  $\big(u(t),\rho(t),\nu(t),\mu(t)\big)$, in the sense of Definition \ref{def: weak sol} such that $\big(u(0),\rho(0),\nu(0),\mu(0)\big) = (u_0,\rho_0,\nu_0,\mu_0)$.
\end{theorem}
\begin{proof}
First we show that the operator $T_t = M\circ S_t\circ L$ is indeed a solution operator on $\D^\alpha$. Let $(u(t),\rho(t),\nu(t),\mu(t)) = T_t((u_0,\rho_0,\nu_0,\mu_0))$. Then we have $(u(t),\rho(t),\nu(t),\mu(t))\in\D^\alpha$, and \eqref{eq:regularity:d}--\eqref{eq:regularity:d3} are satisfied. 

To show \eqref{eq:  weak sol}, we apply a change of variables in the integrals by letting $x=y(\xi,t)$ for each $t$. Strictly speaking we have to integrate over the set $\{\xi\in\R\mid y_{\xi}(\xi,t)>0\}$, but since $y_{\xi}(\xi,t)= 0$ implies that both $r(\xi,t)=0$ and $U_{\xi}(\xi,t)=0$ the integrands vanish on the complement of this set. We have that
\begin{align}
&\int_0^{\infty}\int_{\R}\big(u\phi_t+\frac{1}{2}u^2\phi_x+\frac 14\big(\int_{-\infty}^{x} d\mu - \int_{x}^{\infty} d\mu\big)\phi\big)\:d xd t \nonumber\\
	&= \int_0^{\infty}\int_\R U(\phi_t\circ y+\frac 12 U\phi_x\circ y)y_{\xi}(\xi)\:d\xi d t \nonumber\\
	&\qquad + \frac 14\int_0^{\infty}\int_{\R}\big(\int_{-\infty}^{y(\xi)} d \mu - \int_{y(\xi)}^{\infty} d \mu\big)\phi\circ y y_{\xi}\:d\xi d t \nonumber\\
	&= \int_0^{\infty}\int_{\R}\big[ U(\xi,t)\big(\frac{d}{d t}\phi(y(\xi,t),t)-\frac 12 U\phi_x(y(\xi,t),t)\big)\nonumber\\ 
	&\qquad +\frac 12\big(\int_{-\infty}^{\xi}V_{\xi}(\eta,t)\:d\eta-\frac 12 \int_{-\infty}^{\infty}V_{\xi}(\eta,t)\:d\eta\big)\phi \circ y\big]y_{\xi}\:d\xi d t \nonumber\\
	&= \int_0^{\infty}\int_{\R}\big[ \frac{d}{d t}(U\phi\circ y) - \frac 12 U^2\phi_x\circ y\big]y_{\xi}\:d\xi d t \nonumber\\
	&= \int_0^{\infty}\int_{\R}\big[\frac{d}{d t}(U\phi\circ y y_{\xi}) - \frac 12 U^2\phi_{x}\circ y y_{\xi}-U\phi\circ y y_{t,\xi}\big]\:d\xi d t \nonumber\\
	&= \int_0^{\infty}\int_{\R}\big[\frac{d}{d t}(U\phi\circ y y_{\xi}) - (\frac 12 U^2\phi\circ y)_{\xi}\big]\:d\xi d t \nonumber\\
	&= -\int_{\R} u_0\phi|_{t=0}\:d x.
\end{align}
The equation for $\rho$ is treated in the same way,
\begin{align}
\int_0^{\infty}\int_{\R}\rho\phi_t + (\rho u)\phi_x\:d xd t &= \int_0^{\infty}\int_{\R}\big(\phi_t +  u\phi_x\big)\rho\:d xd t \nonumber\\
&= \int_0^{\infty}\int_{\R}\big(\phi_t\circ y +  U\phi_x\circ y\big)r\:d\xi d t \nonumber\\
&= \int_0^{\infty}\int_{\R}\frac{d}{d t}\phi(y(\xi,t),t)r(\xi,t)\:d\xi d t \nonumber\\
&= \int_0^{\infty}\int_{\R}\frac{d}{d t}\big(\phi(y(\xi,t),t)r(\xi,t)\big)\:d\xi d t \nonumber\\
&= -\int_{\R}\phi(y(\xi,0),0)r(\xi,0)\:d\xi\nonumber\\
&= -\int_{\R}\rho_0(x)\phi(x,0)\:d x.
\end{align}
The inequality for $\mu$ is proved following the same lines
\begin{align}
\int_0^{\infty}\int_{\R}(\phi_t+u\phi_x)\:d\mu(t)d t &= \int_0^{\infty}\int_{\R}(\phi_t+u\phi_x)\circ y V_{\xi}\:d\xi d t \nonumber\\
	&= \int_0^{\infty}\int_{\R}V_{\xi}\frac{d}{d t}(\phi\circ y)\:d\xi d t \nonumber\\
	&= \int_\R \Big(\int_0^{\tau}H_{0,\xi}\frac{d}{d t}(\phi\circ y)\:d t\nonumber\\
	&\qquad + \int_{\tau}^\infty(1-\alpha(y(\tau)))H_{0,\xi}\frac{d}{d t}(\phi\circ y)\:d t\Big)\:d\xi\nonumber\\
	&= \int_{\R}\Big(\alpha(y(\tau))H_{0,\xi}\phi\circ y|_{t=\tau}- H_{0,\xi}(\phi\circ y)|_{t=0}\Big)\:d\xi\nonumber\\
	&\geq -\int_{\R} H_{0,\xi}(\phi\circ y)|_{t=0}\:d\xi \nonumber\\
	&= - \int_{\R} \phi|_{t=0}\:d\mu_0.
\end{align}

We are going to show that \eqref{eq:regularity:6}--\eqref{eq:regularity:5} hold. First we show that $u$ is H{\"o}lder continuous. To that end let $(x_1,t_1)$, $(x_2,t_2)\in \R\times[0,\infty)$. Then there exist $\xi_1$ and $\xi_2\in \R$ such that $y(\xi_1,t_1)=x_1$ and $y(\xi_2,t_2)=x_2$. Thus
\begin{align}
|u(x_1,t_1)-u(x_2,t_2)| &= |u(y(\xi_1,t_1),t_1)-u(y(\xi_2,t_2),t_2)|\nonumber\\
	&\leq |\int_{y(\xi_1,t_1)}^{y(\xi_2,t_1)}u_x(x,t_1)\:d x| + |u(y(\xi_2,t_1),t_1)-u(y(\xi_2,t_2),t_2)|\nonumber\\
	&\leq \|u_x(\cdot,t_1)\|_2\sqrt{|y(\xi_2,t_1)-y(\xi_1,t_1)|} + |\int_{t_1}^{t_2}U_t(\xi_2,t)\:d t|\nonumber\\
	&\leq \sqrt{\nu_0(\R)}\sqrt{|x_2-x_1|+|y(\xi_2,t_2)-y(\xi_2,t_1)|}\nonumber\\
	&\quad + \frac 14\nu_0(\R)|t_2-t_1|\nonumber\\
	&\leq \sqrt{\nu_0(\R)}\sqrt{|x_2-x_1|+\|u_0\|_\infty|t_2-t_1|+\frac 18\nu_0(\R)|t_2-t_1|^2}\nonumber\\
	&\quad + \frac 14\nu_0(\R)|t_2-t_1|,
\end{align}
and thus $u$ is H{\"o}lder continuous. Moreover, if $t_1,t_2\leq T$ for some finite $T$, then 
\begin{equation}
|u(x_1,t_1)-u(x_2,t_2)|\leq C(1+\sqrt T)\sqrt{|x_2-x_1|+|t_2-t_1|},
\end{equation}
where $C$ depends on $\|u_0\|_\infty$ and $\nu_0(\R)$ only. Hence $u$ is locally H\" older continuous both with respect to time and space with H\"older exponent one half. 

Let $t\geq 0$ be given, and choose a sequence $t_n$ in $[0,\infty)$ that converges to $t$. In order to show the continuity of the measure $\nu$ let $\phi\in C_c^\infty(\R)$ be given. Then
\begin{equation}
\int_\R\phi(x)\:d\nu(t_n) = \int_\R\phi(y(\xi,t_n))H_{0,\xi}(\xi)\:d\xi.
\end{equation}
Since $\phi$ is bounded and of compact support, and $y(\xi,t_n)$ tends to $y(\xi,t)$ for almost every $\xi$ as $t_n$ tends to $t$, we have that $\phi(y(\xi,t_n))H_{0,\xi}(\xi)\leq \|\phi\|_\infty H_{0,\xi}(\xi)$. Thus by the Lebesgue dominated convergence theorem $\nu(t_n)$ converges star weakly to $\nu(t)$ in the sense of measures.

The weak star continuity of $\rho\:d x$ is proved in the same manner as for $\nu$.

Next, we show that \eqref{eq: alpha weak} holds. The proof is similar to the one of \cite[Theorem 4.3]{GHR15}. Let $t>0$ be given. Then we have that
\begin{align}
\lim_{s\uparrow t}V_\xi(\xi,s) &= 
	\begin{cases}
	H_{0,\xi}(\xi), &\text{if } t\leq\tau(\xi),\\
	V_\xi(\xi,t), &\text{if } \tau(\xi) < t.
	\end{cases}\\
\lim_{s\downarrow t}V_\xi(\xi,s) &= V_\xi(\xi,t).
\end{align}
We want to prove that
\begin{subequations}
\label{eq:nu m and nu p}
\begin{align}
\lim_{s\uparrow t} \mu(s) &= y(t)_{\#}\big(\lim_{s\uparrow t} V_\xi(\xi,s)\:d\xi\big),\\
\lim_{s\downarrow t} \mu(s) &= y(t)_{\#}\big(V_\xi(\xi,s)\:d\xi\big) = \mu(t),
\end{align}
\end{subequations}
where the limit is in the weak star sense in $\mathcal M^+(\R)$. Let $\phi\in C_c^\infty(\R)$ be given, then
\begin{equation}
\int_\R\phi(x)\:d\mu(s) = \int_\R\phi(y(\xi,s))V_\xi(\xi,s)\:d\xi.
\end{equation}
Since $\phi$ is bounded and of compact support, and $y(\xi,s)$ tends to $y(\xi,t)$ for almost every $\xi$ as $s$ tends to $t$, we have that $\phi(y(\xi,s))V_\xi(\xi,s)\leq \|\phi\|_\infty H_{0,\xi}(\xi)$. Thus by the Lebesgue dominated convergence theorem \eqref{eq:nu m and nu p} hold. Let us prove that
\begin{equation}
d\mu(t) = d\mu(t)^-_{ac}+(1-\alpha(x))d\mu(t)^-_s,
\end{equation}
where $\mu(t)^- = \lim_{s\uparrow t}\mu(s)$.
We have that
\begin{align}
\label{eq:nu m nu}
\mu(t)^--\mu(t) &= y(t)_{\#}\big(\lim_{s\uparrow t} V_\xi(\xi,s)\:d\xi\big)-y(t)_{\#}\big(V_\xi(\xi,t)\:d\xi\big)\nonumber\\
	&= y(t)_{\#}\big((\lim_{s\uparrow t} V_\xi(\xi,s)-V_\xi(\xi,t))\:d\xi\big)\nonumber\\
	&= y(t)_{\#}\big(\alpha(y(\xi,t))H_{0,\xi}(\xi)\mathbf 1_{\{\xi\mid y_\xi(\xi,t) = 0\}}d\xi\big).
\end{align}
Define for each $t$ the sets
\begin{align}
B(t) &= \{\xi\mid y_\xi(\xi,t) > 0\},\\
A(t) &= y(B,t).
\end{align}
Let us show that $A(t)$ is of full measure. Since $y(\cdot,t)$ is surjective $A(t)^c\subseteq y(B(t)^c,t)$, and thus
\begin{equation}
\mathrm m(A(t)^c) \leq \int_{y(B(t)^c,t)}\:d\xi = \int_{B(t)^c}y_\xi(\xi,t)\:d\xi = 0.
\end{equation}
We prove that $y_\xi(\xi,t)>0$ almost everywhere in $y(t)^{-1}(A(t))$. Assume that there exists $\xi\in y(t)^{-1}(A(t))$ such that $y_\xi(\xi,t) = 0$. Then $\xi\in B(t)^c$, and there must be $\xi'$ in $B(t)$ such that $y(\xi',t) = y(\xi,t)$. But then $[\xi,\xi']\subseteq y(t)^{-1}(A(t))$, and $y(\eta,t) = y(\xi',t)$ for $\eta\in (\xi,\xi')$. Thus either $y_\xi(\xi',t) = 0$ or $y_\xi(\xi',t)$ is undefined, which contradicts $\xi'\in B(t)$.

Equation \eqref{eq:nu m nu} implies that on any measurable set $C\subseteq A(t)$ the measure $\mu(t)^--\mu(t)$ is zero. Hence $\mu(t)^--\mu(t)$ is supported on a set of measure zero, and is thus singular with respect to the Lebesgue measure. By the general change of variable formula for any $C\subseteq A(t)^c$ we have that $\left(\mu(t)^--\mu(t)\right)(C) = \int_C\alpha(x)d\mu(t)^-$, and hence $\left(\mu(t)^--\mu(t)\right) = \alpha\:d\big(\left.\mu(t)^-\right|_{A(t)^c}\big)$. To complete the proof we need to prove that
\begin{equation}
\mu(t)^-_{ac} = \left.\mu(t)^-\right|_{A(t)},
\end{equation}
which would imply that $\mu(t)^-_{s} = \left.\mu(t)^-\right|_{A(t)^c}$. First we note that $\mathrm m(A(t)^c) = 0$. Thus for any measurable $C$
\begin{align}
\mu(t)^-_{ac}(C) &\leq \mu(t)^-_{ac}(C\cap A(t))+ \mu(t)^-_{ac}(C\cap A(t)^c)\nonumber\\
	&\leq \mu(t)^-(C\cap A(t)),
\end{align}
and hence $\mu(t)^-_{ac} \leq \mu(t)^-|_{A(t)}$. We must show that $\mu(t)^-|_{A(t)}$ is absolutely continuous. Let $E$ be a set of measure zero and define $K_M =\{\xi\in\R\mid \frac {V_\xi(\xi,t)}{y_\xi(\xi,t)} \leq M\}$. Then $\mathbf{1}_{K_M}$ tends pointwise to one in $y(t)^{-1}(A(t))$, and thus
\begin{align}
\int_{y^{-1}(A(t)\cap E)}V_\xi(\xi,t)\mathbf{1}_{K_M}\:d\xi &\leq M \int_{y^{-1}(A(t)\cap E)}y_\xi(\xi,t)\:d\xi\nonumber\\
	&= M\mathrm m(A(t)\cap E)\nonumber\\
	&= 0.
\end{align}
Hence by the monotone convergence theorem $\left.\mu(t)^-\right|_{A(t)}(E) = 0$ for all sets $E$ of measure zero, and $\left.\mu(t)^-\right|_{A(t)}$ is absolutely continuous.
\end{proof}

See Example \ref{example alpha dissipative} for an example of an $\alpha$-dissipative solution in Eulerian coordinates.

We end this section by a corollary that connects the notion of $\alpha$-dissipative solutions to the construction of dissipative solutions by Bressan and Constantin \cite{BC},  Dafermos \cite{D}, Zhang and Zheng \cite{ZZ}, and Wunsch \cite{W10}.
\begin{corollary}
In the special case $\alpha(x)$ being constant, $\nu_0 =\mu_0= \mu_{0,ac}$, and $\rho_0\equiv 0$ there exists a formula for the solution similar to the formula in \cite[Theorem 2.1]{D} for $\alpha\equiv 1$. More precisely, choosing as initial characteristic the identity function, i.e. $y(\xi,0)=\xi$ (which is always possible in this case), one obtains
\begin{align}
u(x,t) &= u_0(\xi_{x,t}) + \frac 12\int_0^t\Big(\int_{-\infty}^{\xi_{x,t}} \left(1-\alpha\mathbf{1}_{\{z\mid 0\leq\frac{-2}{u_{0,x}(z)}\leq s\}}\right)u_{0,x}(z)^2\:d z\nonumber\\
    &\qquad -\frac 12\int_{-\infty}^{\infty} \left(1-\alpha\mathbf{1}_{\{z\mid 0\leq\frac{-2}{u_{0,x}(z)}\leq s\}}\right)u_{0,x}(z)^2\:d z\Big)\:d s,\\
x &= \xi_{x,t} + u_0(\xi_{x,t})t + \frac 12\int_0^t\int_0^s\Big(\int_{-\infty}^{\xi_{x,t}} \left(1-\alpha\mathbf{1}_{\{z\mid 0\leq\frac{-2}{u_{0,x}(z)}\leq \sigma\}}\right)u_{0,x}(z)^2\:d z\nonumber\\
    &\qquad -\frac 12\int_{-\infty}^{\infty} \left(1-\alpha\mathbf{1}_{\{z\mid 0\leq\frac{-2}{u_{0,x}(z)}\leq \sigma\}}\right)u_{0,x}(z)^2\:d z\Big)\:d\sigma d s,
\end{align}
where $\xi_{x,t}$ is given implicitely through the relation $x=y(\xi_{x,t},t)$.
\end{corollary}

\section{The Lipschitz stability}\label{section lipschitz metric}

In this section we construct a parametrized family of metrics on bounded sets of $\D^\alpha$ that renders the flow Lipschitz continuous with respect to the initial data. The construction is based on the ones in \cite{BHR,LipCHline,LipCH}. The idea is to first create a parametrized family of metrics $\tilde d(t, \cdot, \cdot)$ on $\F^\alpha$, and then to use $\tilde d(t,\cdot, \cdot)$ to construct a parametrized family of metrics on $\D^\alpha$. 

We start by motivating the choice of the set of initial data $\D_0^\alpha$.
Having a close look at the construction of $\alpha$-dissipative solutions in the last two sections, it is natural to require that the initial data $(u_0,\rho_0,\mu_0, \nu_0)\in \D_0^\alpha$. To be more precise, all the important information about the solution is contained in the functions $(u, \rho, \mu)$ while the measure $\nu$ is only added for technical reasons. Correspondingly the important information in Lagrangian coordinates is encoded in the functions $(y,U,r,V)$ while $H$ is a help function. In particular, one has that $M((y,U,H,r,V))$ and $M((y,U,\tilde H, r, V))$ yield the same triplet $(u, \rho,\mu)$ in Eulerian coordinates. Thus if we do not have initial data in $\D_0^\alpha$ we can get solutions which are equal for all practical purposes but far apart in Lagrangian coordinates. In addition, this choice is consistent with the conservative solutions constructed in \cite{N}, that is $\alpha\equiv 0$, since $\nu(t) = \mu(t)$ for all $t\geq 0$ for conservative solutions.

Beside of the choice of the initial data, there is one more major difficulty. As illustrated in Example \ref{example why g}, solutions do not necessarily play nicely with the $B$-norm. In particular, the sudden changes in the energy can lead to jumps in the distance. To model the drop in the energy at wave breaking in a continuous way we introduce a function $g$, which separates points where wave breaking will occur and points where there will be no wave breaking. We start by splitting $\R$ into two regions: One consisting of all points $\xi\in\R$ where wave breaking will occur and one where there will be no wave breaking.

\begin{definition}
\label{def_Omega}
For each $X\in\F^\alpha$ define the sets $\Omega_c(X)$, and $\Omega_d(X)$ by
\begin{align}
\Omega_c(X) &= \{\xi\in\R\mid U_{\xi}(\xi)\geq 0\text{ or } r(\xi)\neq 0\},\\
\Omega_d(X) &= \{\xi\in\R\mid U_{\xi}(\xi) < 0\text{ and } r(\xi) = 0\}.
\end{align}
\end{definition}
For each $X\in\F^\alpha$ the real line splits into mutually disjoint parts $\Omega_{\iota}(X)$, $\iota = c,d$. 
Given $X_0\in \F^\alpha_i$ and $\xi \in \R$, observe that if you start initially in $\Omega_c(X_0)$, $c$ for conservative (or continuous), you will remain there for all later times. That is, $\Omega_c(X_0)\subseteq\Omega_c(X(t))$ for all $t\geq 0$.  If you, on the other hand, start initially in $\Omega_d(X_0)$, $d$ for dissipative (or discontinuous), you will at the wave breaking time $\tau$, given by \eqref{def:tau}, enter $\Omega_c(X(\tau))$. In other words, $\cap_{t\geq 0}\Omega_d(X(t)) = \emptyset$.

To model continuously the drop in the energy we introduce the function $g$ and for technical reasons in addition $g_2$ and $g_3$.
\begin{definition}
For each $X\in\F^\alpha$ and $\xi\in\R$ define
\begin{equation}
g(X(\xi)) =
\begin{cases}
y_{\xi}(\xi)+H_{\xi}(\xi)-\alpha(y(\xi))H_\xi(\xi), &\xi\in\Omega_d(X),\\
y_{\xi}(\xi) + V_{\xi}(\xi), &\xi\in\Omega_c(X).
\end{cases}
\end{equation}
Furthermore define $g_2$ by
\begin{equation}
g_2(X(\xi)) = 
	\begin{cases}
	\|\alpha'\|_\infty H_\infty U_\xi(\xi), &\xi\in\Omega_d(X),\\
	0, &\xi\in\Omega_c(X).
	\end{cases}
\end{equation}
and $g_3$ by
\begin{equation}
g_3(X(\xi)) = 
	\begin{cases}
	\|\alpha'\|_\infty U(\xi) U_\xi(\xi), &\xi\in\Omega_d(X),\\
	0, &\xi\in\Omega_c(X),
	\end{cases}
\end{equation}
where $\displaystyle H_\infty=\lim_{\xi\to\infty}H(\xi)$.
\end{definition}

By construction we have that  $g(X)-1$, $g_2(X)$, and $g_3(X)\in L^2(\R)$. We want to show that $t\mapsto g(X(t)),t\mapsto g_2(X(t)),t\mapsto g_3(X(t))$ are well-defined functions, which are continuous in time, for any solution $X(t) = S_t(X_0)$.
\begin{proposition}
Let $X(\xi,t) = S_t(X_0(\xi))$ for some $X_0\in\F_i^\alpha$ and $\xi \in \R$. Then, the mappings $t\mapsto g(X(\xi,t))-1$, $t\mapsto g_2(X(\xi,t))$, $t\mapsto g_3(X(\xi,t))$ are continuous with respect to the $L^2(\R)$-norm.
\end{proposition}
\begin{proof}
For $t\neq\tau(\xi)$ the functions are pointwise continuous since for $\xi$ given $U$, $U_\xi$, $y_\xi$, $H_\xi$, $V_\xi$, $y$ and $\alpha$ are continuous functions when $t\neq\tau(\xi)$. What remains is to prove that they are pointwise continuous at $t=\tau(\xi)$. We look at $g$ first. By \eqref{eq:lim U_xi y_xi} and \eqref{def:V},
\begin{align}
\lim_{t\uparrow\tau(\xi)}g(X(\xi,t)) & = \lim_{t\uparrow\tau(\xi)}\big( y_\xi(\xi,t)+(1-\alpha(y(\xi,t)))H_{0,\xi}(\xi)\big)\nonumber\\
	&= y_\xi(\xi,\tau(\xi))+\big(1-\alpha(y(\xi,\tau(\xi)))\big)H_{0,\xi}(\xi)\nonumber\\
	&= y_\xi(\xi,\tau(\xi))+V_\xi(\xi,\tau(\xi))\nonumber\\
	&= \lim_{t\downarrow\tau(\xi)}\left[y_\xi(\xi,t) + V_\xi(\xi,t)\right].
\end{align}
From \eqref{eq:lim U_xi y_xi} we have that $U_\xi(\xi,t)\rightarrow 0$ as $t\rightarrow\tau(\xi)$, and hence $t\mapsto g_2(X(\xi,t))$ and $t\mapsto g_3(X(\xi,t))$ are continuous at $t=\tau(\xi)$. 

Since $g(X(\xi,t'))\rightarrow g(X(\xi,t))$ pointwise as $t'$ tends to $t$, we have by the dominated convergence theorem \begin{align}
\lim_{t'\rightarrow t}\int_\R \left(g(X(\xi,t'))-g(X(\xi,t))\right)^2\:d\xi = \int_\R \lim_{t'\rightarrow t}\left(g(X(\xi,t'))-g(X(\xi,t))\right)^2\:d\xi = 0,
\end{align}
if we can find a function $k(\xi)\in L^1(\R)$ such that $|g(X(\xi,t'))-g(X(\xi,t))|^2\leq k(\xi)$ for a.e $\xi$. Therefore recall \eqref{characteristic_system_a_dissipative}, which implies that 
\begin{equation}
\vert U_\xi(\xi,t)\vert \leq \vert U_{0,\xi}(\xi)\vert +\frac12 tH_{0,\xi}(\xi)
\end{equation}
and 
\begin{equation}
\vert y_\xi(\xi,t)-1\vert \leq \vert y_{0,\xi}(\xi)-1\vert +t\vert U_{0,\xi}(\xi)\vert +\frac14 t^2H_{0\xi}(\xi).
\end{equation}
Thus one possible choice for $k(\xi)$ is
\begin{equation}
k(\xi)= (2\vert y_{0,\xi}(\xi)\vert + 2T\vert U_{0,\xi}(\xi)\vert + (2+\frac12T^2)H_{0,\xi}(\xi))^2
\end{equation}
where $T=2t$.
The same argumentation holds for $g_2$ and $g_3$.
\end{proof}

To simplify the notation we make the following definition.
\begin{definition}
Given $X=(y,U,H,r,V)\in \F^\alpha$, define $Z=(y_{\xi},U_{\xi},H_{\xi},r)$.
\end{definition}
We start by introducing a natural, preliminary metric on $\F^\alpha$ which will form the basis for establishing the Lipschitz stability  later on.
\begin{definition}
Let $\tilde{d}:\F^\alpha\times \F^\alpha\rightarrow[0,\infty)$ be defined by
\begin{align}
\tilde{d}(X,\bar{X}) &= \|y-\bar{y}\|_{\infty} +\|U-\bar{U}\|_{\infty} + \|\alpha'\|_\infty\|UH_\xi-\bar U\bar H_\xi\|_2 \nonumber\\&\quad + \|H_{\xi}-\bar{H}_{\xi}\|_1 + \|Z-\bar{Z}\|_{2} + \|g(X)-g(\bar{X})\|_2 \nonumber\\&\quad + \|g_2(X)-g_2(\bar X)\|_2 + \|g_3(X)-g_3(\bar X)\|_2.
\end{align}
\end{definition}
Note that $\tilde{d}$ defines a metric on $\F^\alpha$ since $g$, $g_2$, and $g_3$ are well defined. The metric $\tilde d$ will turn out to be a Lipschitz continuous with respect to time for carefully selected initial data in $\F^\alpha$. The Lipschitz constant will depend on the total energy of the initial data and hence we will have to restrict our attention to subsets of $\F^\alpha$ whose elements have bounded total energy. 

\begin{definition}
Let $\F^{\alpha,M}$ and $\F_0^{\alpha,M}$ be the following closed subsets of $\F^\alpha$
\begin{align}
\F^{\alpha,M} &= \left\{X\in\F^\alpha\mid \|H\|_{\infty} \leq M\right\},\\
\F_0^{\alpha,M} &= \F_0^\alpha \cap\F^{\alpha,M}.
\end{align}
Correspondingly, denote by $\D^{\alpha,M}$ the subset of $\D^\alpha$ given by,
\begin{equation}
\D^{\alpha,M} = \left\{(u,\rho,\nu,\mu)\in\D^\alpha\mid\nu(\R)\leq M\right\}.
\end{equation}
\end{definition}
Note that the mappings $M$ and $L$ respect the energy bound,
\begin{align*}
L(\D^{\alpha,M}) &= \F_0^{\alpha,M},\\
M(\F^{\alpha,M}) &= \D^{\alpha,M}.
\end{align*}

The construction $\tilde d$ might seem arbitrary, but it is connected to Eulerian coordinates as shown in the following remark.
\begin{remark}
The terms in the definition of $\tilde d$ are natural in the sense that for smooth $(u,\rho,\nu,\mu)$ the various terms can be translated to Eulerian coordinates. In particular for $x=y(\xi)$ we have
\begin{align*}
y_\xi(\xi) &= \frac{1}{1+u_x(x)^2+\rho(x)^2},\\
H_\xi(\xi) &= \frac{u_x(x)^2+\rho(x)^2}{1+u_x(x)^2+\rho(x)^2},\\
U_\xi(\xi) &= \frac{u_x(x)}{1+u_x(x)^2+\rho(x)^2},\\
r(\xi) &= \frac{\rho(x)}{1+u_x(x)^2+\rho(x)^2},
\end{align*}
and thus
\begin{align*}
U(\xi)H_\xi(\xi) &= u(x)\frac{u_x(x)^2+\rho(x)^2}{1+u_x(x)^2+\rho(x)^2},\\
\alpha(y(\xi))H_\xi &= \alpha(x)\frac{u_x(x)^2+\rho(x)^2}{1+u_x(x)^2+\rho(x)^2},\\
\|\alpha'\|_\infty H_\infty U_\xi(\xi) &= \|\alpha'\|_\infty \nu(\R)\frac{u_x(x)}{1+u_x(x)^2+\rho(x)^2},\\
\|\alpha'\|_\infty U(\xi) U_\xi(\xi) &= \|\alpha'\|_\infty u(x)\frac{u_x(x)}{1+u_x(x)^2+\rho(x)^2}.
\end{align*}
\end{remark}

We will now construct a metric on $\F_0^\alpha$ that renders the flow Lipschitz continuous with respect to the initial data. The metric $\tilde d$ will unfortunately give a positive distance between $X$ and $X\bullet f$, even though they will map to the same element in Eulerian coordinates via the mapping $M$. Following \cite{LipCHline}, we minimize the distance over all possible relabelings and define $J:\F^\alpha\times \F^\alpha\rightarrow \R$ by
\begin{equation}
J(X,\bar{X}) = \inf_{f,g\in G}\big(\tilde d(X \bullet f,\bar{X})+ \tilde d(X,\bar{X}\bullet g)\big).
\end{equation}
Now, $J$ will not separate $X$ and $X\bullet f$. However $J$ is not a metric, as the triangle inequality fails. One can obtain a metric by summing over finite sequences in $\F_0^{\alpha,M}$ as follows.
\begin{definition}
Let $X,\bar{X}\in\F_{i,0}^{\alpha,M}$, then define $d_M:[0,\infty)\times\F_{i,0}^{\alpha,M}\times\F_{i,0}^{\alpha,M}\rightarrow \R$ by
\begin{equation}
d_M(t,X,\bar{X}) = \inf\sum_{n=1}^N J(\Pi S_t(X_{n-1}),\Pi S_t(X_n)),
\end{equation}
where the infimum is taken over all finite sequences $\{X_n\}_{n=0}^N$ in $\F_{i,0}^{\alpha,M}$ such that the endpoints $X_0$ and $X_N$ satisfy
\begin{align}
X &=X_0,\\
\bar{X} &= X_N.
\end{align}
\end{definition}
It is not at all clear that $d_M(t,X,\bar X)$ only vanishes when $X=\bar{X}$. The purpose of the next lemma is to assert that we have a positive lower bound on $d_M(t,X,\bar{X})$ when $X$ differs from $\bar{X}$.
\begin{lemma}[{\cite[Lemma 3.2]{LipCHline}}]
\label{lower_bound_d_lemma}
Given $X$, $\bar X\in \F_{i,0}^{\alpha,M}$ and $t\geq 0$, let $X_1(t)=\Pi S_t(X)$ and $\bar X_1(t)=\Pi S_t(\bar X)$, then 
\begin{equation}
\label{lower_bound_d}
\|y_1(t)-\bar{y}_1(t)\|_{\infty}+\|U_1(t)-\bar{U}_1(t)\|_{\infty}+\|H_1(t)-\bar{H}_1(t)\|_{\infty} \leq 2d_M(t,X,\bar{X}).
\end{equation}
\end{lemma}
Lemma \ref{lower_bound_d_lemma} states that if the distance between $X_1(t)=\Pi S_t(X)$ and $\bar X_1(t)=\Pi S_t(\bar{X})$ equals zero, then $(y_1,U_1,H_1)$ and $(\bar{y}_1,\bar{U}_1,\bar{H}_1)$ coincide. Still, $r_1$ and $\bar{r}_1$, and $V_1$ and $\bar V_1$ could, in principle, differ. The following lemmas shows that this cannot be the case, and consequently $d_M$ is a metric on $\Pi S_t(\F_{i,0}^{\alpha,M})$. 

\begin{lemma}[{\cite[A weaker form of Lemma 6.4]{LipCH}}]
\label{lemma:r bar r}
Given $X,\bar{X}$ in $\F_{i,0}^{\alpha,M}$ and $t\geq 0$, let $X_1(t)=\Pi S_t(X)$ and $\bar X_1(t)=\Pi S_t(\bar X)$, then $d_M(t,X,\bar{X}) = 0$ implies that $r_1(t)=\bar{r}_1(t)$.
\end{lemma} 
\begin{lemma}[]
Let $X_0,\bar X_0 \in \F_0^\alpha$, and $X_1(t) = \Pi S_t(X_0), \bar X_1(t) = \Pi S_t(\bar X_0)$. If $d_M(t,X_0,\bar X_0) = 0$ then $V_1(t)=\bar V_1(t)$.
\end{lemma}
\begin{proof}
From Lemma \ref{lower_bound_d_lemma} and \ref{lemma:r bar r} we have that
\begin{equation}
(y_1(t),U_1(t),H_1(t),r_1(t)) = (\bar y_1(t),\bar U_1(t),\bar H_1(t),\bar r_1(t)),
\end{equation}
and thus for almost all $\xi\in\R$
\begin{align}
y_{1,\xi}(\xi,t) V_{1,\xi}(\xi,t) &= U_{1,\xi}(\xi,t)^2+r_1(\xi,t)^2\nonumber\\
	&= \bar U_{1,\xi}(\xi,t)^2+\bar r_1(\xi,t)^2\nonumber\\
	&= \bar y_{1,\xi}(\xi,t)\bar V_{1,\xi}(\xi,t).
\end{align}
Thus we have that $\bar V_{1,\xi}(t)= V_{1,\xi}(t)$ almost everywhere in the set $\{\xi\in\R\mid y_{1,\xi}(\xi,t)>0\}$. We turn our attention to the set $\{\xi\mid y_{1,\xi}(\xi,t) = 0\}$. Assume that $\alpha=1$. Then $y_{1,\xi}(\xi,t) = 0$ implies that $V_{1,\xi}(\xi,t) = 0$ and thus $V_{1,\xi}(t)=\bar V_{1,\xi}(t)$ almost everywhere. Assume that $0\leq\alpha<1$. Then if $y_{1,\xi}(\xi,t) = 0$ we have that $t=\tau_1(\xi)$. There are two cases to consider. If $t=0$ we have by definition that $V_{1,\xi}(\xi,t) = H_{1,\xi}(\xi,t)= \bar H_{1,\xi}(\xi,t) = \bar V_{1,\xi}(\xi,t)$. If $t>0$, then we have from the time evolution operator $S_t$ that $V_{1,\xi}(\xi,t) = (1-\alpha(y_1(\xi,\tau_1(\xi))))H_{1,\xi}(\xi,t) = \bar V_{1,\xi}(\xi,t)$. Thus $V_{1,\xi}(\xi,t) = \bar V_{1,\xi}(\xi,t)$ for almost all $\xi\in\R$, and since $\displaystyle \lim_{\xi\rightarrow-\infty}V_1(\xi,t) = 0 = \lim_{\xi\rightarrow-\infty}\bar V_1(\xi,t)$ we have that $V_1(\xi,t) = \bar V_1(\xi,t)$.
\end{proof} 

In order to estimate the time evolution of $d_M(t,X_0,\bar X_0)$, the following lemma is essential. 
\begin{lemma}[{\cite[Lemma 4.8]{N}}]
\label{lemma J relabeling f}
If $X_0,\bar{X}_0\in\F_{i,0}^{\alpha,M}$, then
\begin{equation}
J\big(\Pi S_t(X_0),\Pi S_t(\bar{X}_0)\big) \leq e^{\frac 12t} J\big(S_t(X_0),S_t(\bar{X}_0)\big).
\end{equation}
\end{lemma}

Thus we estimate the time evolution of $\tilde d(S_t(X_0),S_t(\bar X_0)\bullet f)$ for $X_0,\bar X_0\in\F_{i,0}^{\alpha,M}$ and $f\in G$. First let us see how we can use the function $g$ to bound the time evolution of $\|Z(t)-\bar Z(t)\|_2$.

\begin{lemma}
Let $X(t)$ and $\bar{X}(t)$ be the solutions with initial data $X_0,\bar X_0\in\F_{i}^{\alpha}$, then
\begin{subequations}\label{time_Z_est}
\begin{align}
\|y_\xi(t)-\bar y_\xi(t)\|_2 &\leq \|y_{0,\xi}-\bar y_{0,\xi}\|_2 + \int_0^t\|U_\xi(s)-\bar U_\xi(s)\|_2\:d s,\\
\|U_\xi(t)-\bar U_\xi(t)\|_2 &\leq \|U_{0,\xi}-\bar U_{0,\xi}\|_2+\frac12t\|H_{0,\xi}-\bar H_{0,\xi}\|_2 \nonumber\\
 & + \frac 12\int_0^t\Big(\|g(X(s))-g(\bar{X}(s))\|_2
 +\|y_\xi(s)-\bar y_\xi(s)\|_2 \Big)\:d s,\\
\|H_\xi(t)-\bar H_\xi(t)\|_2 &= \|H_{0,\xi}-\bar H_{0,\xi}\|_2,\\
\|r(t)-\bar r(t)\|_2 &= \|r_0-\bar r_0\|_2.
\end{align}
\end{subequations}
\label{time_Z}
\end{lemma}
\begin{proof}
To write more concisely we omit $\xi$ from the notation in this proof. We have that for any $\xi\in\R$,
\begin{subequations}\label{obv:est}
\begin{align}
|r(t)-\bar{r}(t)| &= |r_0-\bar{r}_0|,\\
|H_{\xi}(t)-\bar{H}_{\xi}(t)| &= |H_{0,\xi}-\bar{H}_{0,\xi}|,\\
|y_{\xi}(t)-\bar{y}_{\xi}(t)| &\leq |y_{0,\xi}-\bar{y}_{0,\xi}| + \int_0^t|U_{\xi}(s)-\bar{U}_{\xi}(s)|\:d s.
\end{align}
\end{subequations}
There are three cases to consider for the estimate of $|U_\xi(t)-\bar U_\xi(t)|$. 

If $\xi\in\Omega_c(X_0)\cap\Omega_c(\bar{X}_0)$, then 
\begin{equation}
\label{Z_simple_bound}
|U_{\xi}(t)-\bar{U}_{\xi}(t)| \leq |U_{0,\xi}-\bar{U}_{0,\xi}| + \frac 12\int_0^t|H_{0,\xi}-\bar{H}_{0,\xi}|\:d s.
\end{equation}
Assume that $\xi\in\Omega_d(X_0)\cap\Omega_c(\bar X_0)$. Then we have for $t<\tau$ that
\begin{equation}
|U_{\xi}(t) -\bar{U}_{\xi}(t)| \leq |U_{0,\xi} -\bar{U}_{0,\xi}| + \frac 12\int_0^t|H_{0,\xi}-\bar{H}_{0,\xi}|\:d s.
\end{equation}
If $t\geq \tau$ we have
\begin{align}
|U_{\xi}(t) -\bar{U}_{\xi}(t)| &\leq |U_{0,\xi}-\bar U_{0,\xi}| + \frac 12\int_0^\tau|V_{\xi}(s)-\bar{H}_{0,\xi}|\:d s+\frac12 \int_{\tau}^t |H_{0,\xi}-\bar H_{0,\xi}|\:d s\nonumber\\
	&\leq |U_{0,\xi}-\bar U_{0,\xi}| +\frac12 t|H_{0,\xi}-\bar H_{0,\xi}|\nonumber\\Ê
	& \qquad + \frac 12\int_0^t\left(|g(X(s))-g(\bar{X}(s))|+|y_{\xi}(s)-\bar{y}_{\xi}(s)|\right)\:d s.
\end{align}
Let $\xi\in\Omega_d(X_0)\cap\Omega_d(\bar{X}_0)$. If $\tau\leq t<\bar\tau$, then $V_{\xi}(s) \leq H_{0,\xi}$, $U_{\xi}(t)\geq 0\geq\bar{U}_{\xi}(t)$, and both $U_{\xi}(t),\bar{U}_{\xi}(t)$ are non-decreasing. Hence  
$$0\leq U_{\xi}(t)-\bar{U}_{\xi}(t)\leq |U_{0,\xi}-\bar{U}_{0,\xi}|+\frac 12t|H_{0,\xi}-\bar{H}_{0,\xi}|.$$
If $\tau\leq \bar{\tau}\leq t$ we have
\begin{align}
|U_{\xi}(t)-\bar{U}_{\xi}(t)| &\leq |U_{\xi}(\bar\tau)-\bar{U}_{\xi}(\bar\tau)| + \frac 12 \int_{\bar\tau}^t|V_{\xi}(s)-\bar{V}_{\xi}(s)|\:d s\nonumber\\
	&\leq |U_{\xi}(\bar\tau)-\bar{U}_{\xi}(\bar\tau)| + \frac 12\int_{\bar\tau}^t\left(|g(X(s))-g(\bar{X}(s))|+|y_{\xi}(s)-\bar{y}_{\xi}(s)|\right)d s\nonumber\\
	&\leq |U_{0,\xi}-\bar{U}_{0,\xi}| +\frac12 t|H_{0,\xi}-\bar H_{0,\xi}| \nonumber\\
	& \quad +  \frac 12\int_0^t\left(|g(X(s))-g(\bar{X}(s))|+|y_{\xi}(s)-\bar{y}_{\xi}(s)|\right)\:d s.
\end{align}
Thus we have for any $\xi\in\R$ that,
\begin{align*}
|U_\xi(t)-\bar U_\xi(t)| &\leq |U_{0,\xi}-\bar U_{0,\xi}| + \frac12 t |H_{0,\xi}-\bar H_{0,\xi}|\nonumber\\
&\quad +\frac 12\int_0^t\Big(|g(X(s))-g(\bar{X}(s))| +|y_\xi(s)-\bar y_\xi(s)|\Big)\:d s.
\end{align*}
Applying the $L^2(\R)$-norm on both sides to the above inequality and \eqref{obv:est}, we obtain \eqref{time_Z_est}.
\end{proof}

To simplify the remaining proofs we split for each time $t$ the real line into three parts: a continuous part where no wave breaking occured so far, a part where one solution already experienced wave breaking, and a part where both solutions have experienced wave breaking.
\begin{definition}
\label{real line split}
Define the disjoint sets
\begin{subequations}
\begin{align}
R_{cont}(t) &= \{\xi\in\R\mid t<\tau(\xi),\bar{\tau}(\xi)\},\\
R_{mix}(t) &= \{\xi\in\R\mid \bar{\tau}(\xi)\leq t<\tau(\xi) \text{ or } \tau(\xi)\leq t<\bar{\tau}(\xi)\},\\
R_{disc}(t) &= \{\xi\in\R\mid \tau(\xi),\bar{\tau}(\xi)\leq t\},
\end{align}
\end{subequations}
where the convention $\tau(\xi) = \infty$ for $\xi\in\Omega_c(X_0)$ is used.
\end{definition}

Before the remaining terms in $\tilde d(S_t(X_0), S_t(\bar X_0)\bullet f)$ with $X_0$, $\bar X_0\in \F_{i,0}^{\alpha,M}$ and $f\in G$ will be handled, an estimate for the integral $\int_0^t\|V_\xi(s)-\bar V_\xi(s)\|_1\:d s$ is established. The following observations play an essential role therein. If $X_0\in\F_{i,0}^{\alpha,M}$, then $U_{0,\xi}^2\leq y_{0,\xi}V_{0,\xi} \leq H_{0,\xi}$ since $0\leq y_{0,\xi}\leq 1$. Hence
\begin{equation}
\label{bound U_0xi norm}
\|U_{0,\xi}\|_2 \leq \sqrt{\|H_{0,\xi}\|_1} \leq \sqrt M,
\end{equation}
Moreover, from \eqref{characteristic_system_a_dissipative} we have that $U_t(\xi,t) = \frac 12V_\xi(\xi,t)-\frac 14 V_\infty(t)$, and thus
\begin{equation}
\label{bound U_t}
\|U_t\|_\infty \leq \frac 14\|V\|_\infty \leq \frac 14 M,
\end{equation}
since $V_{\xi}(t,\xi)\leq H_{0,\xi}(\xi)$ for all $\xi\in\R$.

\begin{lemma}
\label{lemma V_xi and bar V_xi}
Given any two solutions $X(t)$ and $\bar X(t)$ with initial data $X_0\in \F_i^{\alpha,M}$ and $\bar X_0\in \F_{i,0}^{\alpha,M}$, respectively, one has
\begin{align}
\int_0^t\int_\R|V_\xi(\xi,s)& -\bar V_\xi(\xi,s)|\:d\xi d s
	\leq (8+3t+2t^2+\frac 34t^3)\|H_{0,\xi}-\bar H_{0,\xi}\|_1\nonumber\\
	&\quad + 2\left(\sqrt{1+\frac 14t^2}+(1+\frac 14t^2)\right)\sqrt M\|U_{0,\xi}-\bar U_{0,\xi}\|_2\nonumber\\
	&\quad + \left(4 + t + t^2 + \frac 14t^3\right)\sqrt M\|y_{0,\xi}-\bar y_{0,\xi}\|_2\nonumber\\
	&\quad + \left(\sqrt{1+\frac 14t^2}+1+\frac 14t^2\right)\sqrt M\nonumber \\
	& \qquad\qquad \times \int_0^t\|g(X(s))-g(\bar X(s))\|_2+\|y_\xi(s)-\bar y_\xi(s)\|_2\:d s
\end{align}
\end{lemma}

\begin{proof}
We omit $\xi$ from the notation in this proof. Let $R_{cont}(t)$, $R_{mix}(t)$, and $R_{disc}(t)$ as in Definition \ref{real line split}. We want to estimate $\int_0^t\int_\R|V_\xi(\xi,s)-\bar V_\xi(\xi,s)|\:d\xi d s$ by the terms of $\tilde d(X_0,\bar X_0)$ and $\tilde d(X(t),\bar X(t))$. Since $R_{cont}(t)$, $R_{mix}(t)$, and $R_{disc}(t)$ are disjoint, we can look at each region separately and add the results.

If $\xi\in R_{cont}(t)$, then
\begin{equation}
\label{estimate R_cont}
\int_0^t|V_\xi(\xi,s)-\bar V_\xi(\xi,s)|\:d s = t|H_{0,\xi}-\bar H_{0,\xi}|.
\end{equation}

Let $\xi\in R_{mix}(t)$, then there are two possibilities: either $\tau\leq t$, or $\bar\tau\leq t$. The two possibilities must be treated differently. To that end we split $R_{mix}(t)$ into two disjoint sets,
\begin{align}
A(t) &= \{\xi\in R_{mix}(t)\mid \tau(\xi)\leq t\},\\
\bar A(t) &= \{\xi\in R_{mix}(t)\mid \bar\tau(\xi)\leq t\}.
\end{align}

Assume first that $\xi\in\bar A(t)$, then
\begin{align}
\int_0^t|V_\xi(s)-\bar V_\xi(s)|\:d s &= \int_0^{\bar\tau}|V_\xi(s)-\bar V_\xi(s)|\:d s+\int_{\bar\tau}^t|V_\xi(s)-\bar V_\xi(s)|\:d s\nonumber\\
	&= \bar\tau|H_{0,\xi}-\bar H_{0,\xi}| + \int_{\bar\tau}^t|H_{0,\xi}-\bar V_\xi(s)|\:d s.
\end{align}
If $\bar V_\xi(s)\geq H_{0,\xi}$, then $|\bar V_\xi(s)- H_{0,\xi}| \leq |H_{0,\xi}-\bar H_{0,\xi}|$, and we can bound the integral by 
\begin{equation}
\int_0^t |V_\xi(s)-\bar V_\xi(s)|\:d s=t | H_{0,\xi}-\bar H_{0,\xi}|.
\end{equation}
If $H_{0,\xi}>\bar V_\xi(s)$, then there are two cases to consider, namely $\xi\in\Omega_c(X_0)$ and $\xi\in\Omega_d(X_0)$. If $\xi\in\Omega_c(X_0)$ we have for $s\geq \bar\tau$,
\begin{equation}
\label{eq:R_mix bar A 1}
|H_{0,\xi}-\bar V_\xi(s)| \leq |g(X(s))-g(\bar X(s))| + |y_\xi(s)-\bar y_\xi(s)|.
\end{equation}
If, on the other hand, $\xi\in\Omega_d(X_0)$ we have for $s\geq \bar\tau$
\begin{align}
H_{0,\xi} - \bar V_\xi(s) &= H_{0,\xi} - \bar H_{0,\xi} + \alpha(\bar y(\bar\tau))\bar H_{0,\xi}\nonumber\\
	&= \big(1-\alpha(\bar y(\bar \tau))\big)\left(H_{0,\xi} - \bar H_{0,\xi}\right) + \alpha(\bar y(\bar\tau))H_{0,\xi}
\end{align}
and hence by \eqref{eq:sm diff} and \eqref{eq:lim U_xi y_xi},
\begin{align}
\label{eq:R_mix bar A 2}
\int_{\bar\tau}^t|H_{0,\xi} - \bar V_\xi(s)|\:d s &\leq (t-\bar\tau)|H_{0,\xi} - \bar H_{0,\xi}| + \alpha(\bar y(\bar\tau))\int_{\bar\tau}^t H_{0,\xi}\:d s \nonumber\\
	&\leq (t-\bar\tau)|H_{0,\xi} - \bar H_{0,\xi}| + 2 U_\xi(t)-2 U_\xi(\bar\tau) \nonumber\\
	&\leq (t-\bar\tau)|H_{0,\xi} - \bar H_{0,\xi}| + 2\left(\bar U_\xi(\bar\tau)-U_\xi(\bar\tau)\right) \nonumber\\
	&\leq (t-\bar\tau)|H_{0,\xi} - \bar H_{0,\xi}| + 2\left(|U_{0,\xi}-\bar U_{0,\xi}| + \frac 12\bar\tau|H_{0,\xi}-\bar H_{0,\xi}|\right).
\end{align}
Thus for $\xi\in\bar A(t)$ we have
\begin{align}
\label{eq:R_mix bar A}
\int_0^t|V_\xi(s)-\bar V_\xi(s)|\:d s &\leq 2t|H_{0,\xi}-\bar H_{0,\xi}| + 2|U_{0,\xi}-\bar U_{0,\xi}| \nonumber\\
	&\quad + \int_0^t|g(X(s))-g(\bar X(s))|+|y_\xi(s)-\bar y_\xi(s)|\:d s.
\end{align}
Since \eqref{size of broken set F_0^M} implies that the measure of $\bar A(t)$ is bounded by $(1+\frac 14t^2)M$, we obtain $L^2(\R)$-estimates on the right hand side of \eqref{eq:R_mix bar A} when integrating over $\bar A(t)$.

Assume now that $\xi\in A(t)$. Then
\begin{equation}\label{4.45}
\int_0^t|V_\xi(s)-\bar V_\xi(s)|\:d s = \tau|H_{0,\xi}-\bar H_{0,\xi}| + \int_\tau^t|V_\xi(s)-\bar H_{0,\xi}|\:d s.
\end{equation}
If $V_\xi(s)\geq \bar H_{0,\xi}$, then $|V_\xi(s)-\bar H_{0,\xi}| \leq |H_{0,\xi}-\bar H_{0,\xi}|$, and we can bound the integral by \begin{equation}\label{eq:basis}
\int_0^t |V_\xi(s)-\bar V_\xi(s)|\:d s\leq t | H_{0,\xi}-\bar H_{0,\xi}|.
\end{equation}

Assume that $\bar H_{0,\xi}>V_\xi(s)$ and $s\geq\tau$. Since the measure of $A(t)$ is not bounded in terms of $M$ and $t$ only, see Corollary \ref{cor:size of broken set}, we cannot use the same argument as for $\xi\in \bar A(t)$ to get $L^2(\R)$-estimates. For $X_0\in\F_i^{\alpha,M}$ there exists $\hat X_{0}\in\F_{i,0}^{\alpha,M}$ such that $\hat X_0\bullet f = X_0$, where $f = y_0+H_0\in G$. Then \eqref{eq:H_0xi tau} gives
\begin{equation}
\label{eq:tau relabeled}
\hat H_{0,\xi}\circ f = \frac{1}{1+\frac 14\hat\tau^2\circ f} = \frac{1}{1+\frac 14\tau^2},
\end{equation}
and in particular for $t\geq\tau$,
\begin{equation}
1\leq (1+\frac 14t^2)\hat H_{0,\xi}\circ f.
\end{equation}
We want to estimate $\bar H_{0,\xi}- V_\xi(s)$. Assume that $\xi\in\Omega_d(\bar X_0)$, then
\begin{align}
\bar H_{0,\xi} - V_\xi(s) &= \bar H_{0,\xi} - H_{0,\xi} + \alpha(y(\tau))H_{0,\xi} \nonumber\\
	&= \left(1-\alpha(y(\tau))\right)\left(\bar H_{0,\xi}-H_{0,\xi}\right) + \alpha(y(\tau))\bar H_{0,\xi},
\end{align}
and hence by \eqref{eq:sm diff} and \eqref{eq:lim U_xi y_xi},
\begin{align}
\label{eq:R_mix A 1}
\int_\tau^t |V_\xi(s)-\bar H_{0,\xi}|\:d s &\leq (t-\tau)|H_{0,\xi}-\bar H_{0,\xi}| + \int_\tau^t\bar H_{0,\xi}\:d s\nonumber\\
	&= (t-\tau)|H_{0,\xi}-\bar H_{0,\xi}| + 2\left(\bar U_\xi(t)-\bar U_\xi(\tau)\right)\nonumber\\
	&\leq (t-\tau)|H_{0,\xi}-\bar H_{0,\xi}| + 2\left(U_\xi(\tau)-\bar U_\xi(\tau)\right)\nonumber\\
	&\leq (t-\tau)|H_{0,\xi}-\bar H_{0,\xi}| + 2|U_{0,\xi}-\bar U_{0,\xi}| + \tau|H_{0,\xi}-\bar H_{0,\xi}|\nonumber\\
	& \leq t|H_{0,\xi}-\bar H_{0,\xi}| +2|U_{0,\xi}-\bar U_{0,\xi}|.
\end{align}
Since $t\geq s\geq\tau$ we have that
\begin{align}
|U_{0,\xi}-\bar U_{0,\xi}| &\leq \left(1+\frac 14t^2\right)\hat H_{0,\xi}\circ f|U_{0,\xi}-\bar U_{0,\xi}|\nonumber\\
	&= \left(1+\frac 14t^2\right)\left|\hat U_{0,\xi}\circ ff_\xi\hat H_{0,\xi}\circ f - \bar U_{0,\xi} \hat H_{0,\xi}\circ f\right|\nonumber\\
	& = \left(1+\frac14 t^2\right)\left| \hat U_{0,\xi}\circ f(\hat H_{0,\xi}\circ ff_\xi-\bar H_{0,\xi})\right. \nonumber\\
	&\qquad\qquad\qquad\quad   +\left.\hat U_{0,\xi}\circ f\bar H_{0,\xi}-\bar U_{0,\xi}\hat H_{0,\xi}\circ f\right| \nonumber \\
       & \leq \left(1+\frac14 t^2\right)\left( |\hat U_{0,\xi}\circ f||\hat H_{0,\xi}\circ f f_\xi-\bar H_{0,\xi}|\right.\nonumber\\
    &\qquad + \left.|\hat U_{0,\xi}\circ f\bar H_{0,\xi}-\bar U_{0,\xi}\hat H_{0,\xi}\circ f|\right).
\end{align}
By definition $\hat H_{0,\xi}\circ ff_\xi = H_{0,\xi}$ and $|\hat U_{0,\xi}\circ f|\leq 1$, and thus
\begin{equation}
\label{estimate U_0xi bar U_0xi temp}
|U_{0,\xi}-\bar U_{0,\xi}| \leq \left(1+\frac 14t^2\right)\left(\left|\hat U_{0,\xi}\circ f \bar H_{0,\xi}-\bar U_{0,\xi}\hat H_{0,\xi}\circ f\right| + |H_{0,\xi}-\bar H_{0,\xi}|\right).
\end{equation}
To estimate the first term on the right hand side of \eqref{estimate U_0xi bar U_0xi temp} we note that $1 = 1-f_\xi+f_\xi$, and hence
\begin{align}
\left|\hat U_{0,\xi}\circ f \bar H_{0,\xi}-\bar U_{0,\xi}\hat H_{0,\xi}\circ f \right| 
&= \left|\hat U_{0,\xi}\circ f \bar H_{0,\xi}-\bar U_{0,\xi}\hat H_{0,\xi}\circ f \right|\left((1-f_\xi)+f_\xi\right)\nonumber\\
&\leq\left(|\hat U_{0,\xi}\circ f\bar H_{0,\xi}|+|\bar U_{0,\xi} \hat H_{0,\xi}\circ f |\right)(1-f_\xi)\nonumber\\
&\qquad + \left|\hat U_{0,\xi}\circ ff_\xi\bar H_{0,\xi}-\bar U_{0,\xi} \hat H_{0,\xi}\circ ff_\xi\right|\nonumber\\
&\leq (|\bar U_{0,\xi}|+\bar H_{0,\xi})(1-f_\xi) + |\bar U_{0,\xi}||H_{0,\xi}-\bar H_{0,\xi}|\nonumber\\&\quad + \bar H_{0,\xi}|U_{0,\xi}-\bar U_{0,\xi}|.
\end{align}
Since $\bar X_0\in\F_{i,0}^{\alpha,M}$, and $f = y_0+H_0$ we have that
\begin{equation}
1-f_\xi = \bar y_{0,\xi}+\bar H_{0,\xi}- y_{0,\xi}- H_{0,\xi},
\end{equation}
and thus
\begin{align}
\label{eq:R_mix A 2}
|U_{0,\xi}-\bar U_{0,\xi}| &\leq \left(1+\frac 14t^2\right)\big(4|H_{0,\xi}-\bar H_{0,\xi}|+(|\bar U_{0,\xi}|+\bar H_{0,\xi})|y_{0,\xi}-\bar y_{0,\xi}|\nonumber\\
	&\qquad \qquad \qquad \qquad +\bar H_{0,\xi}|U_{0,\xi}-\bar U_{0,\xi}| \big).
\end{align}
Inserting \eqref{eq:R_mix A 2} into \eqref{eq:R_mix A 1} and subsequently into \eqref{4.45}, we end up with
\begin{align}
\label{eq:R_mix A 3}
\int_0^t |V_\xi(s)-\bar H_{0,\xi}|\:d s &\leq \left(8+2t+2t^2\right)|H_{0,\xi}-\bar H_{0,\xi}|\nonumber\\ 
	&\quad+ 2\left(1+\frac 14t^2\right)\Big((|\bar U_{0,\xi}|+\bar H_{0,\xi})|y_{0,\xi}-\bar y_{0,\xi}| \nonumber\\
	&\qquad\qquad\qquad\qquad\quad   +\bar H_{0,\xi}|U_{0,\xi}-\bar U_{0,\xi}|\Big).
\end{align}

Assume that $\xi\in\Omega_c(\bar X_0)$. Then, by a similar argument as for \eqref{estimate U_0xi bar U_0xi temp},
\begin{align}
\bar H_{0,\xi}- V_\xi(s) &\leq (\bar H_{0,\xi}- V_\xi(s))\left(1+\frac 14t^2\right)\hat H_{0,\xi}\circ f\nonumber\\
        & = \left(1+\frac14 t^2\right) \left(\bar H_{0,\xi}\hat H_{0,\xi}\circ f-\hat V_\xi\circ f(s) \hat H_{0,\xi}\circ f f_\xi\right) \nonumber\\
        & =\left(1+\frac14 t^2\right) \left( \bar H_{0,\xi}(\hat H_{0,\xi}\circ f-\hat V_\xi\circ f(s))+\hat V_\xi\circ f(s)(\bar H_{0,\xi}- H_{0,\xi})\right) \nonumber\\
        & \leq \left(1+\frac14 t^2 \right)\left( | H_{0,\xi}-\bar H_{0,\xi}| +\bar H_{0,\xi}(\hat H_{0,\xi}\circ f-\hat V_\xi\circ f(s))\right).
 \end{align}
Since $\bar X_0\in\F_{i,0}^{\alpha,M}$, and $f = y_0+H_0$ we have that
\begin{equation}
1-f_\xi = \bar y_{0,\xi}+\bar H_{0,\xi}- y_{0,\xi}- H_{0,\xi},
\end{equation}
and thus,
\begin{align}
\left(\hat H_{0,\xi}\circ f-\hat V_\xi\circ f(s)\right ) \bar H_{0,\xi}
& = ( \hat H_{0,\xi}\circ f-\hat V_\xi\circ f(s)) \bar H_{0,\xi} (1-f_\xi)\nonumber\\
&\quad + ( H_{0,\xi}-V_\xi(s))\bar H_{0,\xi} \nonumber\\
& \leq \bar H_{0,\xi}\vert 1-f_\xi\vert + \vert H_{0,\xi}-\bar H_{0,\xi}\vert + \bar H_{0,\xi}(\bar H_{0,\xi}-V_\xi(s))  \nonumber \\
& \leq (\bar H_{0,\xi} |y_\xi(s)-\bar y_\xi(s)| + \bar H_{0,\xi}\vert g(X(s))-g(\bar X(s))\vert ) \nonumber\\
	&\qquad +2 \vert H_{0,\xi}-\bar H_{0,\xi}\vert + \bar H_{0,\xi}\vert y_{0,\xi}-\bar y_{0,\xi}\vert.
 \end{align}
Hence,
\begin{align}
\label{eq:R_mix A 4}
\int_0^t |V_\xi(s)-\bar H_{0,\xi}|\:d s &\leq 3t\left(1+\frac 14t^2\right)|H_{0,\xi}-\bar H_{0,\xi}| \nonumber\\
	&+ t\left(1+\frac 14t^2\right)\bar H_{0,\xi}\vert y_{0,\xi}-\bar y_{0,\xi}\vert\nonumber\\
	&+ \left(1+\frac 14t^2\right)\bar H_{0,\xi}\int_0^t|g(X(s))-g(\bar X(s))|+|y_\xi(s)-\bar y_\xi(s)|\:d s.
\end{align}

From \eqref{eq:basis}, \eqref{eq:R_mix A 3}, and \eqref{eq:R_mix A 4} we get that for $\xi\in A(t)$ the estimate
\begin{align}
\label{eq:R_mix A}
\int_0^t|V_\xi(s)-\bar V_\xi(s)|\:d s &\leq \left(8+3t+2t^2+\frac 34t^3\right)|H_{0,\xi}-\bar H_{0,\xi}|\nonumber\\
	&+ 2\left(1+\frac 14t^2\right)\bar H_{0,\xi}|U_{0,\xi}-\bar U_{0,\xi}|\nonumber\\
	&+ \left(1+\frac 14t^2\right)\left(2|\bar U_{0,\xi}| + (2+t)|\bar H_{0,\xi}|\right)|y_{0,\xi}-\bar y_{0,\xi}|\nonumber\\
	&+ \left(1+\frac 14t^2\right)\bar H_{0,\xi}\int_0^t|g(X(s))-g(\bar X(s))|\:d s \nonumber\\
	&+ \left(1+\frac 14t^2\right)\bar H_{0,\xi}\int_0^t|y_\xi(s)-\bar y_\xi(s)|\:d s,
\end{align}
holds.

Assume finally that $\xi\in R_{disc}(t)$. Recall that according to \eqref{size of broken set F_0^M} the measure of $R_{disc}(t)$ is less than or equal to $\left(1+\frac 14t^2\right)M$. Since \eqref{eq:R_mix bar A 2} and \eqref{eq:R_mix A 1} give the same estimate we can assume without loss of generality that $\tau\leq\bar\tau\leq t$. By  \eqref{eq:R_mix A 1} we have
\begin{align}
\label{estimate R_disc}
\int_0^t|V_\xi(s)-\bar V_\xi(s)|\:d s &= \int_0^\tau|H_{0,\xi}-\bar H_{0,\xi}|\:d s+\int_\tau^{\bar\tau}|V_\xi(s)-\bar H_{0,\xi}|\:d s\nonumber\\
	&\quad + \int_{\bar\tau}^t |V_\xi(s)-\bar V_\xi(s)|\:d s\nonumber\\
	&= (\bar\tau+\tau)|H_{0,\xi}-\bar H_{0,\xi}| + 2|U_{0,\xi}-\bar U_{0,\xi}| \nonumber\\
	&\quad + \int_{\bar\tau}^t|V_\xi(s)-\bar V_\xi(s)|\:d s.
    \end{align}
For $s\geq\bar\tau$ we have
\begin{equation}
\label{eq:R disc t big}
|V_\xi(s)-\bar V_\xi(s)| \leq |g(X(s))-g(\bar X(s))| + |y_\xi(s)-\bar y_\xi(s)|.
\end{equation}
Thus
\begin{align}
\label{eq:R_disc}
\int_0^t|V_\xi(s)-\bar V_\xi(s)|\:d s &\leq 2t|H_{0,\xi}-\bar H_{0,\xi}| + 2|U_{0,\xi}-\bar U_{0,\xi}|\nonumber\\
	&\quad + \int_0^t|g(X(s))-g(\bar X(s))| + |y_\xi(s)-\bar y_\xi(s)|\:d s.
\end{align}
If we collect the estimates \eqref{estimate R_cont}, \eqref{eq:R_mix bar A}, \eqref{eq:R_mix A}, and \eqref{eq:R_disc} we get
\begin{align}
\int_0^t \int_\R|V_\xi(\xi,s)-\bar V(\xi,s)|d\xi\:d s 
	&\leq \left(8+3t+2t^2+\frac 34t^3\right)|H_{0,\xi}-\bar H_{0,\xi}|\nonumber\\
	&\quad + 2\int_{\bar A(t)\cup R_{disc}(t)}|U_{0,\xi}-\bar U_{0,\xi}|\:d\xi \nonumber\\
	&\quad + \int_{\bar A(t)\cup R_{disc}(t)}\int_0^t|g(X(s))-g(\bar X(s))|\:d sd\xi\nonumber\\
	&\quad + \int_{\bar A(t)\cup R_{disc}(t)}\int_0^t|y_\xi(s)-\bar y_\xi(s)|\:d sd\xi\nonumber\\
	&\quad + 2\left(1+\frac 14t^2\right)\int_{A(t)}|U_{0,\xi}-\bar U_{0,\xi}|\bar H_{0,\xi}\:d\xi\nonumber\\
	&\quad + 2\left(1+\frac 14t^2\right)\nonumber\\
	&\qquad\times\int_{A(t)}(|\bar U_{0,\xi}|+(1+\frac12t)|\bar H_{0,\xi}|)|y_{0,\xi}-\bar y_{0,\xi}|\:d\xi\nonumber\\
	&\quad + \left(1+\frac 14t^2\right)\bar H_{0,\xi}\nonumber\\
	&\qquad \times\int_{A(t)}\int_0^t|g(X(s))-g(\bar X(s))|\:dsd\xi\nonumber\\
	&\quad + \left(1+\frac 14t^2\right)\bar H_{0,\xi}\int_{A(t)}\int_0^t|y_\xi(s)-\bar y_\xi(s)|\:d s d\xi
\end{align}

From \eqref{size of broken set F_0^M} we have that $\mathrm{m}(\bar A(t)\cup R_{disc}(t))\leq (1+\frac 14t^2)M$, while \eqref{bound U_0xi norm} ensures that $\|\bar U_{0,\xi}\|_2\leq \sqrt M$. Thus
\begin{align}
\int_0^t\int_\R|V_\xi(\xi,s)&-\bar V(\xi,s)|\:d\xi d s \leq
         (8+3t+2t^2+\frac 34t^3)\|H_{0,\xi}-\bar H_{0,\xi}\|_1\nonumber\\
	&\quad + 2\left(\sqrt{1+\frac 14t^2}+(1+\frac 14t^2)\right)\sqrt M\|U_{0,\xi}-\bar U_{0,\xi}\|_2\nonumber\\
	&\quad + \left(4 + t + t^2 + \frac 14t^3\right)\sqrt M\|y_{0,\xi}-\bar y_{0,\xi}\|_2\nonumber\\
	&\quad + \left(1+\frac 14t^2+\sqrt{1+\frac 14t^2}\right)\sqrt M\nonumber\\
	& \qquad \qquad \times\int_0^t\|g(X(s))-g(\bar X(s))\|_2+\|y_\xi(s)-\bar y_\xi(s)\|_2\:d s
\end{align}
\end{proof}

The system \eqref{characteristic_system_a_dissipative} and Lemma \ref{lemma V_xi and bar V_xi} can now be used to estimate the time evolution of $\|y(t)-\bar y(t)\|_\infty$, $\|U(t)-\bar U(t)\|_\infty$, and $\|U(t)H_{0,\xi}-\bar U(t)\bar H_{0,\xi}\|_2$. Note that it is not necessary to derive an estimate for $\|H_\xi(t)-\bar H_\xi(t)\|_1$, since $H_\xi(\xi,t) = H_{0,\xi}(\xi)$ and $\bar H_\xi(\xi,t) = \bar H_{0,\xi}(\xi)$.

\begin{lemma}
\label{lemma y and U and UH_xi}
Let $X(t),\bar X(t)$ be solutions with initial data $X_0\in\F_i^{\alpha,M}$ and $\bar X_0\in\F_{i,0}^{\alpha,M}$, respectively. Then
\begin{align}
\|U(t)-\bar{U}(t)\|_{\infty} &\leq \|U_0-\bar U_0\|_\infty+\frac 14\int_0^t\int_{\R} |V_{\xi}(\xi,s)-\bar{V}_{\xi}(\xi,s)|\:d\xi\:d s,\\ \label{differenz:y}
\|y(t)-\bar{y}(t)\|_{\infty} &\leq \|y(0)-\bar y(0)\|_\infty + \int_0^t\|U(s)-\bar U(s)\|_\infty\:d s\\
\|U(t)H_\xi-\bar U(t)\bar H_\xi\|_2 &\leq \|U_0H_{0,\xi}-\bar U_0\bar H_{0,\xi}\|_2 + \frac 14Mt\|H_{0,\xi}-\bar H_{0,\xi}\|_2\nonumber \\
&\quad + \frac 14\sqrt{M}\int_0^t\int_{\R} |V_{\xi}(\xi,s)-\bar{V}_{\xi}(\xi,s)|\:d\xi\:d s .
\end{align}
\end{lemma}
\begin{proof}
We omit $\xi$ from the notation in this proof. It suffices to show the inequalities for $\|U(t)-\bar U(t)\|_\infty$ and $\|U(t)H_{0,\xi}-\bar U(t)\bar H_{0,\xi}\|_2$, since \eqref{differenz:y} follows immediately from \eqref{characteristic_system_a_dissipative}. We have
\begin{align}
\frac{d}{d t}|U(t)-\bar{U}(t)| &\leq \left|\frac 12\left(V(t)-\bar{V}(t)\right)+\frac 14\left(V_{\infty}(t)-\bar{V}_{\infty}(t)\right)\right| \nonumber\\ \label{last}
	&\leq \frac 14\int_{-\infty}^{\infty} |V_{\xi}(t)-\bar{V}_{\xi}(t)|\:d\xi,
\end{align}
and from \eqref{bound U_t} there is
\begin{align}
\frac{d}{d t}|U(t)H_{0,\xi}-\bar{U}(t)\bar H_{0,\xi}| &\leq |U_t(t)-\bar U_t(t)|\bar H_{0,\xi} + |U_t(t)||H_{0,\xi}-\bar H_{0,\xi}|\nonumber\\ \label{lastlast}
	&\leq \frac 14\int_{-\infty}^{\infty} |V_{\xi}(t)-\bar{V}_{\xi}(t)|\:d\xi \bar H_{0,\xi} + \frac 14M|H_{0,\xi}-\bar H_{0,\xi}|
\end{align}
Since $\|\bar H_{0,\xi}\|_2\leq \sqrt{\|\bar H_{0,\xi}\|_1} \leq \sqrt M$ the lemma follows by applying the $L^2(\R)$-norm to both sides of \eqref{last} and \eqref{lastlast}.
\end{proof}

It remains to estimate the norms $\|g(X(t))-g(\bar X(t))\|_2$, $\|g_2(X(t))-g_2(\bar X(t))\|_2$, and $\|g_3(X(t))-g_3(\bar X(t))\|_2$.

\begin{lemma}
\label{lemma g-functions}
Let $X(t),\bar X(t)$ be the solutions with initial data $X_0\in\F_i^{\alpha,M}$ and $\bar X_0\in \F_{i,0}^{\alpha,M}$, respectively. Then
\begin{align}
\|g_2(X(t))-g_2(\bar X(t))\|_2 &\leq \|g_2(X_0)-g_2(\bar X_0)\|_2 + \sqrt M\|\alpha'\|_\infty\|H_{0,\xi}-\bar H_{0,\xi}\|_1 \nonumber\\
\label{est:g2}
	&\quad+ \|\alpha'\|_\infty M(\|U_{0,\xi}-\bar U_{0,\xi}\|_2+\frac 12t\|H_{0,\xi}-\bar H_{0,\xi}\|_2),\\	
\|g_3(X(t))-g_3(\bar X(t))\|_2 
&\leq \|g_3(X_0)-g_3(\bar X_0)\|_2+\frac14 t \| g_2(X_0)-g_2(\bar X_0)\|_2\nonumber\\
&\quad+ \frac 12 t\|\alpha'\|_\infty\|U_0 H_{0,\xi}-\bar U_0\bar H_{0,\xi}\|_2 \nonumber\\
& \quad +\frac14\|\alpha'\|_\infty Mt \| U_{0,\xi}-\bar U_{0,\xi}\|_2 \nonumber\\
& \quad + \frac18 \|\alpha'\|_\infty Mt^2 \| H_{0,\xi}-\bar H_{0,\xi}\|_2 \nonumber\\
&\quad + \frac 14t\|\alpha'\|_\infty\sqrt M\|H_{0,\xi}-\bar H_{0,\xi}\|_1 \nonumber\\ \label{est:g3}
& \quad + \frac14 \|\alpha'\|_\infty \sqrt{M} \int_0^t\int_{\R}|V_\xi(\xi,s)-\bar V_\xi(\xi,s)|\:d\xi d s,\\
\|g(X(t))-g(\bar X(t))\|_2 &\leq  \|g(X_0)-g(\bar X_0)\|_2 + \|y_{0,\xi}-\bar y_{0,\xi}\|_2\nonumber\\
	&\quad + \frac 12t\|\alpha'\|_\infty M\|U_{0,\xi}-\bar U_{0,\xi}\|_2\nonumber\\
	&\quad +  (1+\frac 14\|\alpha'\|_\infty Mt^2)\|H_{0,\xi}-\bar H_{0,\xi}\|_2\nonumber\\
	&\quad + \|\alpha'\|_\infty \sqrt M\|y_0-\bar y_0\|_\infty\nonumber\\
	&\quad + t\|\alpha'\|_\infty \|U_0H_{0,\xi}-\bar U_0\bar H_{0,\xi}\|_2\nonumber\\
	&\quad + 2\|g_3(X_0)-g_3(\bar X_0)\|_2\nonumber\\
	&\quad + \frac 14t\|g_2(X_0)-g_2(\bar X_0)\|_2\nonumber\\
	&\quad + \frac 14t\|\alpha'\|_\infty \sqrt M\|H_{0,\xi}-\bar H_{0,\xi}\|_1\nonumber\\
	&\quad + \int_0^t\|U_\xi(s)-\bar U_\xi(s)\|_2\:d s\nonumber\\ \label{est:g}
	&\quad + \|\alpha'\|_\infty \sqrt M \int_0^t\|U(s)-\bar U(s)\|_\infty\:d s.
\end{align} 
\end{lemma}
\begin{proof}
We omit $\xi$ from the notation in this proof. Recall the splitting of $\R$ into $R_{cont}(t)$, $R_{mix}(t)$, and $R_{disc}(t)$ given in Definition \ref{real line split}.

Let us prove the estimate for $g_2$. Assume that $\xi\in R_{cont}(t)$, then if $\xi\in\Omega_c(X_0)\cap\Omega_c(\bar X_0)$ we have
\begin{equation}
\label{eq:g_2 cont cc}
g_2(X(t))-g_2(\bar X(t)) = 0-0 = 0.
\end{equation}
We can assume without loss of generality that $\xi\in \Omega_c(X_0)\cap\Omega_d(\bar X_0)$ in the mixed case since the argument does not use the property $\bar y_0+\bar H_0=id$. We have
\begin{equation}
|g_2(X(t))-g_2(\bar X(t))| = \|\alpha'\|_\infty \bar H_\infty|\bar U_\xi(t)|,
\end{equation}
and from \eqref{eq:lim U_xi y_xi} we know that $|\bar U_\xi(t)|\leq |\bar U_{0,\xi}|$, and hence
\begin{equation}
\label{eq:g_2 cont cd}
|g_2(X(t))-g_2(\bar X(t))|\leq |g_2(X_0)-g_2(\bar X_0)|.
\end{equation}

If $\xi\in\Omega_d(X_0)\cap\Omega_d(\bar X_0)$, then \eqref{eq:lim U_xi y_xi} implies that
\begin{align}
\label{eq:g_2 cont dd}
|g_2(X(t))-g_2(\bar X(t))| &\leq \|\alpha'\|_\infty|H_\infty-\bar H_\infty||\bar U_{\xi}(t)| + \|\alpha'\|_\infty H_\infty|U_\xi(t)-\bar U_\xi(t)|\nonumber\\
	&\leq \|\alpha'\|_\infty|\bar U_{0,\xi}||H_\infty-\bar H_\infty| \nonumber\\
	&\quad + \|\alpha'\|_\infty M(|U_{0,\xi}-\bar U_{0,\xi}|+\frac 12t|H_{0,\xi}-\bar H_{0,\xi}|).
\end{align}

Assume that $\xi\in R_{mix}(t)$. Then if $\xi\in\Omega_c(X_0)\cap\Omega_d(\bar X_0)$ or $\xi\in\Omega_d(X_0)\cap\Omega_c(\bar X_0)$ we have that
\begin{equation}
\label{eq:g_2 mix cd}
g_2(X(t))-g_2(\bar X(t)) = 0-0 = 0.
\end{equation}
If $\xi\in \Omega_d(X_0)\cap\Omega_d(\bar X_0)$, and if we assume without loss of generality that $\tau\leq t < \bar\tau$, then we have
\begin{equation}
|g_2(X(t))-g_2(\bar X(t))| = \|\alpha'\|_\infty \bar H_\infty|\bar U_\xi(t)|,
\end{equation}
and from \eqref{eq:lim U_xi y_xi} we know that $|\bar U_\xi(t)|\leq |\bar U_{\xi}(\tau)|$, and hence
\begin{equation}
|g_2(X(t))-g_2(\bar X(t))|\leq \|\alpha'\|_\infty\bar H_\infty|\bar U_\xi(\tau)- U_\xi(\tau)|.
\end{equation}
Thus following the same line as in \eqref{eq:g_2 cont dd}, we end up with
\begin{equation}
\label{eq:g_2 mix dd}
|g_2(X(t))-g_2(\bar X(t))|\leq \|\alpha'\|_\infty M(|U_{0,\xi}-\bar U_{0,\xi}|+\frac 12t|H_{0,\xi}-\bar H_{0,\xi}|).
\end{equation}

Assume that $\xi\in R_{disc}(t)$. Then
\begin{equation}
\label{eq:g_2 disc}
g_2(X(t))-g_2(\bar X(t)) = 0-0 = 0.
\end{equation}

From \eqref{eq:g_2 cont cc}, \eqref{eq:g_2 cont cd}, \eqref{eq:g_2 cont dd}, \eqref{eq:g_2 mix cd}, \eqref{eq:g_2 mix dd}, and \eqref{eq:g_2 disc} we get
\begin{align}
|g_2(X(t))-g_2(\bar X(t))| &\leq |g_2(X_0)-g_2(\bar X_0)| \nonumber\\
	&\quad + \|\alpha'\|_\infty|\bar U_{0,\xi}||H_\infty-\bar H_\infty|\nonumber\\
	&\quad + \|\alpha'\|_\infty M|U_{0,\xi}-\bar U_{0,\xi}|\nonumber\\
	&\quad + \frac 12\|\alpha'\|_\infty Mt|H_{0,\xi}-\bar H_{0,\xi}|.
\end{align}
The estimate \eqref{est:g2} is obtained after some direct computations by taking the $L^2(\R)$-norm on both sides.

Next we show the inequality for $g_3$. Assume that $\xi\in R_{cont}(t)$, then if $\xi\in\Omega_c(X_0)\cap\Omega_c(\bar X_0)$ we have
\begin{equation}
\label{eq:g_3 cont cc}
g_3(X(t))-g_3(\bar X(t)) = 0-0 = 0.
\end{equation}
For the mixed case assume without loss of generality $\xi\in \Omega_c(X_0)\cap\Omega_d(\bar X_0)$ since the argument does not depend on the property $\bar y_0+\bar H_0=id$. We have
\begin{equation}
|g_3(X(t))-g_3(\bar X(t))| = \|\alpha'\|_\infty |\bar U(t)||\bar U_\xi(t)|,
\end{equation}
and from \eqref{eq:lim U_xi y_xi} we know that $|\bar U_\xi(t)|\leq |\bar U_{0,\xi}|$, and hence by \eqref{bound U_t} we have
\begin{align}
\label{eq:g_3 cont cd}
|g_3(X(t))-g_3(\bar X(t))| &= \|\alpha'\|_\infty|\bar U(t)||\bar U_\xi(t)| \nonumber\\
	&\leq \|\alpha'\|_\infty|\bar U(t)||\bar U_{0,\xi}|\nonumber\\
	&\leq \|\alpha'\|_\infty|\bar U_0||\bar U_{0,\xi}| + \frac 14t\|\alpha'\|\bar H_\infty|\bar U_{0,\xi}|\nonumber\\
	&\leq |g_3(X_0)-g_3(\bar X_0)| + \frac 14t |g_2(X_0)-g_2(\bar X_0)|.
\end{align}
If $\xi\in\Omega_d(X_0)\cap\Omega_d(\bar X_0)$, then
\begin{align}
g_3(X(t))-g_3(\bar X(t)) &= \|\alpha'\|_\infty U(t)U_\xi(t)-\|\alpha'\|_\infty\bar U(t)\bar U_\xi(t)\nonumber\\
	&= \|\alpha'\|_\infty\left(U_0+\int_0^tU_t(s)\:d s\right)\left(U_{0,\xi}+\frac 12t H_{0,\xi}\right)\nonumber\\&\quad - \|\alpha'\|_\infty\left(\bar U_0+\int_0^t\bar U_t(s)\:d s\right)\left(\bar U_{0,\xi}+\frac 12t \bar H_{0,\xi}\right)\nonumber\\
	&= g_3(X_0)-g_3(\bar X_0) + \frac 12 t\|\alpha'\|_\infty\left(U_0H_{0,\xi}-\bar U_0\bar H_{0,\xi}\right)\nonumber\\&\quad + \left(\bar U_{0,\xi}+ \frac 12t \bar H_{0,\xi}\right)\|\alpha'\|_\infty\int_0^t\left(U_t(s)-\bar U_t(s)\right)\:d s\nonumber\\
	\label{eq:g_3 cont ddd}
	&\quad + \|\alpha'\|_\infty\int_0^t U_t(s)\:d s\left(U_{0,\xi}+\frac 12t H_{0,\xi}-\bar U_{0,\xi}-\frac 12 t\bar H_{0,\xi}\right).
\end{align}
Since \eqref{eq:lim U_xi y_xi} implies that $|\bar U_\xi(t)|=|\bar U_{0,\xi}+\frac12 t\bar H_{0,\xi}|$ is decreasing in time and due to \eqref{bound U_t} we have that
\begin{align}
\label{eq:g_3 cont dd}
|g_3(X(t))-g_3(\bar X(t))| &
\leq |g_3(X_0)-g_3(\bar X_0)| + \frac 12t\|\alpha'\|_\infty|U_0H_{0,\xi}-\bar U_0\bar H_{0,\xi}| \nonumber\\
&\quad+ \frac14\|\alpha'\|_\infty Mt|U_{0,\xi}-\bar U_{0,\xi}| + \frac 18\|\alpha'\|_\infty Mt^2|H_{0,\xi}-\bar H_{0,\xi}| \nonumber\\
&\quad+ \frac 14\|\alpha'\|_\infty|\bar U_{0,\xi}+\frac12 t\bar H_{0,\xi}|\int_0^t\int_{\R}|V_\xi(\xi,s)-\bar V_\xi(\xi,s)|\:d\xi d s\\ \label{g_3 cont dd}
&\leq |g_3(X_0)-g_3(\bar X_0)| + \frac 12t\|\alpha'\|_\infty|U_0H_{0,\xi}-\bar U_0\bar H_{0,\xi}| \nonumber\\&\quad+ \frac14\|\alpha'\|_\infty Mt|U_{0,\xi}-\bar U_{0,\xi}| + \frac 18\|\alpha'\|_\infty Mt^2|H_{0,\xi}-\bar H_{0,\xi}| \nonumber\\&\quad+ \frac 14\|\alpha'\|_\infty|\bar U_{0,\xi}|\int_0^t\int_{\R}|V_\xi(\xi,s)-\bar V_\xi(\xi,s)|\:d\xi d s.
\end{align}
Note that we can rewrite \eqref{eq:g_3 cont ddd} in the following way,
\begin{align}
g_3(X(t))-g_3(\bar X(t)) &= g_3(X_0)-g_3(\bar X_0) + \frac 12 t\|\alpha'\|_\infty\left(U_0H_{0,\xi}-\bar U_0\bar H_{0,\xi}\right)\nonumber\\&\quad + \|\alpha'\|_\infty\left(U_{0,\xi}+ \frac 12t H_{0,\xi}\right)\int_0^t\left(U_t(s)-\bar U_t(s)\right)\:d s\nonumber\\&\quad + \|\alpha'\|_\infty\int_0^t \bar U_t(s)\:d s\left(U_{0,\xi}+\frac 12t H_{0,\xi}-\bar U_{0,\xi}-\frac 12 t\bar H_{0,\xi}\right),
\end{align}
and hence
\begin{align}
\label{eq:g_3 sym}
|g_3(X(t))-g_3(\bar X(t))| &\leq |g_3(X_0)-g_3(\bar X_0)| + \frac 12t\|\alpha'\|_\infty|U_0H_{0,\xi}-\bar U_0\bar H_{0,\xi}| \nonumber\\&\quad+ \frac14\|\alpha'\|_\infty Mt|U_{0,\xi}-\bar U_{0,\xi}| + \frac 18\|\alpha'\|_\infty Mt^2|H_{0,\xi}-\bar H_{0,\xi}| \nonumber\\&\quad+ \frac 14\|\alpha'\|_\infty|U_{0,\xi}+\frac12 t H_{0,\xi}|\int_0^t\int_{\R}|V_\xi(\xi,s)-\bar V_\xi(\xi,s)|\:d\xi d s.
\end{align}

Assume that $\xi\in\R_{mix}(t)$. If $\xi\in\Omega_c(X_0)\cap\Omega_d(\bar X_0)$ or $\xi\in\Omega_d(X_0)\cap\Omega_c(\bar X_0)$, then
\begin{equation}
\label{eq:g_3 mix cd}
g_3(X(t))-g_3(\bar X(t)) = 0-0 = 0.
\end{equation}
Let $\xi\in \Omega_d(X_0)\cap\Omega_d(\bar X_0)$ and assume without loss of generality that $\tau\leq t<\bar\tau$. Then from \eqref{eq:lim U_xi y_xi} we know that $|\bar U_\xi(t)|\leq |\bar U_{\xi}(\tau)|$, and hence by \eqref{bound U_t}
\begin{align}
|g_3(X(t))-g_3(\bar X(t))| &= \|\alpha'\|_\infty|\bar U(t)||\bar U_\xi(t)| \nonumber\\
	&\leq \|\alpha'\|_\infty|\bar U(t)||\bar U_{\xi}(\tau)|\nonumber\\
	&\leq \|\alpha'\|_\infty|\bar U(\tau)||\bar U_{\xi}(\tau)| + \frac 14(t-\tau)\|\alpha'\|\bar H_\infty|\bar U_{\xi}(\tau)|\nonumber\\
	&\leq |g_3(X(\tau))-g_3(\bar X(\tau))| + \frac 14(t-\tau) |g_2(X(\tau))-g_2(\bar X(\tau))|.
\end{align}
Combining \eqref{eq:g_2 cont dd} and \eqref{g_3 cont dd} then yields
\begin{align}
|g_3(X(t))-g_3(\bar X(t))| &\leq |g_3(X_0)-g_3(\bar X_0)| + \frac 12\tau\|\alpha'\|_\infty|U_0H_{0,\xi}-\bar U_0\bar H_{0,\xi}| \nonumber\\
	&\quad+ \frac14\|\alpha'\|_\infty M\tau|U_{0,\xi}-\bar U_{0,\xi}| + \frac 18\|\alpha'\|_\infty M\tau^2|H_{0,\xi}-\bar H_{0,\xi}| \nonumber\\
	&\quad+ \frac 14\|\alpha'\|_\infty|\bar U_{0,\xi}|\int_0^\tau\int_{\R}|V_\xi(\xi,s)-\bar V_\xi(\xi,s)|\:d\xi d s\nonumber\\
	&\quad + \frac14\|\alpha'\|_\infty|\bar U_{0,\xi}|(t-\tau)|H_\infty-\bar H_\infty|\nonumber\\
	&\quad + \frac14\|\alpha'\|_\infty M(t-\tau)(|U_{0,\xi}-\bar U_{0,\xi}|+\frac 12\tau|H_{0,\xi}-\bar H_{0,\xi}|),
\end{align}
and thus
\begin{align}
\label{eq:g_3 mix dd}
|g_3(X(t))-g_3(\bar X(t))| &\leq |g_3(X_0)-g_3(\bar X_0)| + \frac 12t\|\alpha'\|_\infty|U_0H_{0,\xi}-\bar U_0\bar H_{0,\xi}| \nonumber\\
	&\quad + \frac14\|\alpha'\|_\infty Mt|U_{0,\xi}-\bar U_{0,\xi}|+\frac 18\|\alpha'\|_\infty Mt^2|H_{0,\xi}-\bar H_{0,\xi}|\nonumber\\
	&\quad + \frac14\|\alpha'\|_\infty|\bar U_{0,\xi}|t|H_\infty-\bar H_\infty|\nonumber\\
	&\quad + \frac 14\|\alpha'\|_\infty|\bar U_{0,\xi}|\int_0^t\int_{\R}|V_\xi(\xi,s)-\bar V_\xi(\xi,s)|\:d\xi d s.
\end{align}

Assume that $\xi\in R_{disc}(t)$. Then
\begin{equation}
\label{eq:g_3 disc}
g_3(X(t))-g_3(\bar X(t)) = 0-0 = 0.
\end{equation}

From \eqref{eq:g_3 cont cc}, \eqref{eq:g_3 cont cd}, \eqref{g_3 cont dd}, \eqref{eq:g_3 mix cd}, \eqref{eq:g_3 mix dd}, and \eqref{eq:g_3 disc} we get
\begin{align}
|g_3(X(t))-g_3(\bar X(t))| &\leq |g_3(X_0)-g_3(\bar X_0)| + \frac 14t|g_2(X_0)-g_2(\bar X_0)|\nonumber\\
	&\quad + \frac 12t\|\alpha'\|_\infty|U_0H_{0,\xi}-\bar U_0\bar H_{0,\xi}|\nonumber\\
	&\quad + \frac 14\|\alpha'\|_\infty Mt|U_{0,\xi}-\bar U_{0,\xi}|\nonumber\\
	&\quad + \frac 18\|\alpha'\|_\infty Mt^2|H_{0,\xi}-\bar H_{0,\xi}|\nonumber\\
	&\quad + \frac 14\|\alpha'\|_\infty|\bar U_{0,\xi}|t|H_\infty-\bar H_\infty|\nonumber\\
	&\quad + \frac 14\|\alpha'\|_\infty|\bar U_{0,\xi}|\int_0^t\int_\R|V_\xi(\xi,s)-\bar V_\xi(\xi,s)|\:d\xi d s.
\end{align}
The estimate \eqref{est:g3} is obtained after some direct computations by taking the $L^2(\R)$-norm on both sides.

We prove the estimate for $g$. Assume that $\xi\in R_{cont}(t)$, and let $\xi\in \Omega_c(X_0)\cap\Omega_c(\bar X_0)$. Then
\begin{align}
\label{eq:g cont cc}
|g(X(t))-g(\bar X(t))| &= |H_{0,\xi}-\bar H_{0,\xi} + y_\xi(t)-\bar y_\xi(t)|\nonumber\\
	&\leq |g(X_0)-g(\bar X_0)| + \int_0^t|U_\xi(s)-\bar U_\xi(s)|\:d s.
\end{align}
For the mixed case assume without loss of generality $\xi\in \Omega_d(X_0)\cap\Omega_c(\bar X_0)$, then
\begin{equation}
g(X(t))-g(\bar X(t)) = g(X_0)-g(\bar X_0) + \int_0^t\left(U_\xi(s)-\bar U_\xi(s)\right)\:d s - \int_0^t\frac d{d s}\alpha(y(s))H_{0,\xi}\:d s.
\end{equation}
We have
\begin{align}
\int_0^t\frac d{d s}\alpha(y(s))H_{0,\xi}\:d s &\leq \int_0^t\|\alpha'\|_\infty|U(s)|H_{0,\xi}\:d s \nonumber\\
	&\leq \int_0^t\|\alpha'\|_\infty|U_0 + \int_0^sU_t(\sigma)\:d\sigma|H_{0,\xi}\:d s \nonumber\\
	&\leq \|\alpha'\|_\infty \left(|U_0|t+\frac 18H_\infty t^2\right) H_{0,\xi} \nonumber\\
	&\leq 2\|\alpha'\|_\infty|U_0|(U_\xi(t)-U_{0,\xi})\nonumber\\
	&\quad + \frac 14\|\alpha'\|_\infty tH_\infty(U_\xi(t)-U_{0,\xi}) \nonumber\\
	&\leq - 2\|\alpha'\|_\infty|U_0|U_{0,\xi} - \frac 14t\|\alpha'\|_\infty H_\infty U_{0,\xi} \nonumber\\ \label{alphaest}
	&\leq 2|g_3(X_0)-g_3(\bar X_0)| + \frac 14t|g_2(X_0)-g_2(\bar X_0)|,
\end{align}
hence
\begin{align}
\label{eq:g cont cd}
|g(X(t))-g(\bar X(t))| 
	&\leq |g(X_0)-g(\bar X_0)| + \int_0^t|U_\xi(s)-\bar U_\xi(s)|\:d s\nonumber\\
	&\quad + 2|g_3(X_0)-g_3(\bar X_0)| + \frac 14t|g_2(X_0)-g_2(\bar X_0)|.
\end{align}
Here we used that $U_{0,\xi}\leq U_\xi(t)\leq 0$.

If $\xi\in\Omega_d(X_0)\cap\Omega_d(\bar X_0)$ we have
\begin{align}
\label{eq:g cont dd}
|g(X(t)) - g(\bar X(t))| &= |y_\xi(t)-\bar y_\xi(t) +(1- \alpha(y(t)))H_{0,\xi}-(1-\alpha(\bar y(t)))\bar H_{0,\xi}| \nonumber\\	
	&\leq  |y_{0,\xi}-\bar y_{0,\xi}| + \int_0^t|U_\xi(s)-\bar U_\xi(s)|\:d s \nonumber\\
	&\quad+|H_{0,\xi}-\bar H_{0,\xi}|+  \|\alpha'\|_\infty \bar H_{0,\xi}\|y_0-\bar y_0\|_\infty  \nonumber\\
	& \quad   + \int_0^t\|\alpha'\|_\infty \bar H_{0,\xi}\|U(s)-\bar U(s)\|_\infty\:d s.
\end{align}

Assume that $\xi\in R_{mix}(t)$, and let $\xi\in\Omega_c(X_0)\cap\Omega_d(\bar X_0)$. Then
\begin{align}
|g(X(t))-g(\bar X(t))| &= | y_\xi(t)-\bar y_\xi(t)+H_{0,\xi}-\bar V_\xi(t) |\nonumber\\
	&\leq |g(X(\bar\tau))-g(\bar X(\bar\tau))| + \int_{\bar\tau}^t|U_\xi(s)-\bar U_\xi(s)|\:d s\nonumber\\
&\leq |g(X_0)-g(\bar X_0)| + 2|g_3(X_0)-g_3(\bar X_0)| \nonumber\\ \label{eq:g mix cd}
	&\quad + \frac 14t|g_2(X_0)-g_2(\bar X_0)| + \int_0^t|U_\xi(s)-\bar U_\xi(s)|\:d s,
\end{align}
where we used \eqref{eq:g cont cd} in the last step.
Both the argument and the result are the same in the case $\xi\in\Omega_d(X_0)\cap\Omega_c(\bar X_0)$. 

If $\xi\in\Omega_d(X_0)\cap\Omega_d(\bar X_0)$, then if $\bar\tau\leq t<\tau$,
\begin{align}
g(X(t))-g(\bar X(t)) &= y_\xi(t)-\bar y_\xi(t) + H_{0,\xi} -\bar H_{0,\xi} + \alpha(\bar y(\bar\tau))\bar H_{0,\xi} - \alpha(y(t))H_{0,\xi}\nonumber\\
	&= g(X(\bar\tau))-g(\bar X(\bar\tau)) + \int_{\bar\tau}^t\left(U_\xi(s)-\bar U_\xi(s)\right)\:d s\nonumber\\
	&\quad - \int_{\bar\tau}^t\frac{d}{d s}\alpha(y(s))H_{0,\xi}\:d s.
\end{align}
Then we can proceed as in \eqref{alphaest}, and apply \eqref{eq:g_2 cont dd}, \eqref{eq:g_3 cont dd}, and \eqref{eq:g cont dd}, to obtain
\begin{align}
\label{eq:g mix dd}
|g(X(t))-g(\bar X(t))| &\leq |g(X(\bar\tau))-g(\bar X(\bar\tau))| + 2|g_3(X(\bar\tau))-g_3(\bar X(\bar\tau))|\nonumber\\
	&\quad + \frac 14(t-\bar\tau)|g_2(X(\bar\tau))-g_2(\bar X(\bar\tau))| + \int_{\bar\tau}^t|U_\xi(s)-\bar U_\xi(s)|\:d s\nonumber\\
	&\leq |y_{0,\xi}-\bar y_{0,\xi}|+\int_0^{\bar \tau} | U_\xi(s)-\bar U_\xi(s)|ds +| H_{0,\xi}-\bar H_{0,\xi}| \nonumber \\
	& \quad + \|\alpha'\|_\infty\bar H_{0,\xi}\|y_0-\bar y_0\|_\infty \nonumber\\
	&\quad +\|\alpha'\|_\infty \bar H_{0,\xi}\int_0^{\bar \tau}  \| U(s)-\bar U(s)\|_\infty  ds\nonumber\\
	&\quad + 2|g_3(X_0)-g_3(\bar X_0)| + \bar\tau\|\alpha'\|_\infty|U_0H_{0,\xi}-\bar U_0\bar H_{0,\xi}| \nonumber\\
	&\quad + \frac12\|\alpha'\|_\infty M\bar\tau|U_{0,\xi}-\bar U_{0,\xi}| + \frac 14\|\alpha'\|_\infty M\bar\tau^2|H_{0,\xi}-\bar H_{0,\xi}| \nonumber\\
	&\quad + \frac 12\|\alpha'\|_\infty|\bar U_{0,\xi}+\frac 12 \bar\tau\bar H_{0,\xi}|\int_0^{\bar\tau}\int_{\R}|V_\xi(\xi,s)-\bar V_\xi(\xi,s)|\:d\xi d s\nonumber\\
	&\quad + \frac 14\|\alpha'\|_\infty|\bar U_{0,\xi}|(t-\bar\tau)|H_\infty-\bar H_\infty|\nonumber\\
	&\quad + \frac 14\|\alpha'\|_\infty M(t-\bar\tau)(|U_{0,\xi}-\bar U_{0,\xi}|+\frac 12\bar\tau|H_{0,\xi}-\bar H_{0,\xi}|) \nonumber\\
	&\quad + \int_{\bar\tau}^t|U_\xi(s)-\bar U_\xi(s)|\:d s \nonumber\\
	&\leq |y_{0,\xi}-\bar y_{0,\xi}| + \frac 12\|\alpha'\|_\infty Mt|U_{0,\xi}-\bar U_{0,\xi}|\nonumber\\
	&\quad + \left(1+\frac 14\|\alpha'\|_\infty Mt^2\right)|H_{0,\xi}-\bar H_{0,\xi}| \nonumber\\
	&\quad + 2|g_3(X_0)-g_3(\bar X_0)| + t\|\alpha'\|_\infty|U_0H_{0,\xi}-\bar U_0\bar H_{0,\xi}| \nonumber\\
	&\quad + \|\alpha'\|_\infty\bar H_{0,\xi}\|y_0-\bar y_0\|_\infty \nonumber\\
	&\quad + \frac 14\|\alpha'\|_\infty|\bar U_{0,\xi}|t|H_\infty-\bar H_\infty| \nonumber\\
	&\quad + \|\alpha'\|_\infty\bar H_{0,\xi} \int_0^t\|U(s)-\bar U(s)\|_\infty\:d s \nonumber\\ 
	&\quad + \int_0^t|U_\xi(s)-\bar U_\xi(s)|\:d s. 
\end{align}
The term
\begin{equation}
\frac 12\|\alpha'\|_\infty|\bar U_{0,\xi} + \frac 12\bar \tau \bar H_{0,\xi}|\int_0^t\int_\R|V_\xi(\xi,s)-\bar V_\xi(\xi,s)|\:d\xi d s,
\end{equation}
vanishes since $\bar U_{0,\xi} + \frac 12\bar\tau \bar H_{0,\xi}=0$ by the definition of $\bar\tau$. 

The case $\tau\leq t<\bar\tau$ is similar. Combining \eqref{eq:g_3 sym} and $U_{0,\xi}+ \frac 12\tau H_{0,\xi}= 0$, we get that
\begin{align}
|g_3(X(\tau))-g_3(\bar X(\tau))| &\leq |g_3(X_0)-g_3(\bar X_0)| + \frac 12 \tau\|\alpha'\|_\infty\left|U_0H_{0,\xi}-\bar U_0\bar H_{0,\xi}\right|\nonumber\\
	&\quad + \frac 14\|\alpha'\|_\infty M\tau\left(|U_{0,\xi}-\bar U_{0,\xi}|+\frac 12\tau |H_{0,\xi}-\bar H_{0,\xi}|\right).
\end{align}
Thus, following the same lines as in \eqref{eq:g mix dd} yields,
\begin{align}
\label{eq:g mix bar dd}
|g(X(t))-g(\bar X(t))| &\leq |y_{0,\xi}-\bar y_{0,\xi}| + \frac 12\|\alpha'\|_\infty Mt|U_{0,\xi}-\bar U_{0,\xi}|\nonumber\\
	&\quad + \left(1+\frac 14\|\alpha'\|_\infty Mt^2\right)|H_{0,\xi}-\bar H_{0,\xi}| \nonumber\\
	&\quad + 2|g_3(X_0)-g_3(\bar X_0)| + t\|\alpha'\|_\infty|U_0H_{0,\xi}-\bar U_0\bar H_{0,\xi}| \nonumber\\
	&\quad + \|\alpha'\|_\infty\bar H_{0,\xi}\|y_0-\bar y_0\|_\infty \nonumber\\
	&\quad + \frac 14\|\alpha'\|_\infty|\bar U_{0,\xi}|t|H_\infty-\bar H_\infty| \nonumber\\
	&\quad + \|\alpha'\|_\infty\bar H_{0,\xi}\int_0^t\|U(s)-\bar U(s)\|_\infty\:d s  \nonumber\\
	&\quad + \int_0^t|U_\xi(s)-\bar U_\xi(s)|\:d s.
\end{align}

Assume that $\xi\in R_{disc}(t)$, and assume without loss of generality that $\bar\tau\leq\tau$, then
\begin{equation}
|g(X(t))-g(\bar X(t))| \leq |g(X(\tau))-g(\bar X(\tau))| + \int_{\tau}^t|U_\xi(s)-\bar U_\xi(s)|\:d s.
\end{equation}
Thus, from \eqref{eq:g mix dd} and \eqref{eq:g mix bar dd}, we get that
\begin{align}
\label{eq:g disc}
|g(X(t))-g(\bar X(t))| &\leq |y_{0,\xi}-\bar y_{0,\xi}| + \frac 12\|\alpha'\|_\infty Mt|U_{0,\xi}-\bar U_{0,\xi}|\nonumber\\
	&\quad + \left(1+\frac 14\|\alpha'\|_\infty Mt^2\right)|H_{0,\xi}-\bar H_{0,\xi}| \nonumber\\
	&\quad + 2|g_3(X_0)-g_3(\bar X_0)| + t\|\alpha'\|_\infty|U_0H_{0,\xi}-\bar U_0\bar H_{0,\xi}| \nonumber\\
	&\quad + \|\alpha'\|_\infty\bar H_{0,\xi}\|y_0-\bar y_0\|_\infty \nonumber\\
	&\quad + \frac 14\|\alpha'\|_\infty|\bar U_{0,\xi}|t|H_\infty-\bar H_\infty| \nonumber\\
	&\quad + \|\alpha'\|_\infty\bar H_{0,\xi}\int_0^t\|U(s)-\bar U(s)\|_\infty\:d s \nonumber\\
	&\quad + \int_0^t|U_\xi(s)-\bar U_\xi(s)|\:d s.
\end{align}

Combining \eqref{eq:g cont cc}, \eqref{eq:g cont cd}, \eqref{eq:g cont dd}, \eqref{eq:g mix cd}, \eqref{eq:g mix dd}, \eqref{eq:g mix bar dd}, and \eqref{eq:g disc}  yields
\begin{align}
|g(X(t))-g(\bar X(t))| &\leq  |y_{0,\xi}-\bar y_{0,\xi}|\nonumber\\
	&\quad + \frac 12\|\alpha'\|_\infty Mt|U_{0,\xi}-\bar U_{0,\xi}|\nonumber\\
	&\quad +  (1 + \frac 14\|\alpha'\|_\infty Mt^2)|H_{0,\xi}-\bar H_{0,\xi}|\nonumber\\
	&\quad + \|\alpha'\|_\infty \bar H_{0,\xi}\|y_0-\bar y_0\|_\infty\nonumber\\
	&\quad + t\|\alpha'\|_\infty |U_0H_{0,\xi}-\bar U_0\bar H_{0,\xi}|\nonumber\\
	&\quad + |g(X_0)-g(\bar X_0)| + 2|g_3(X_0)-g_3(\bar X_0)|\nonumber\\
	&\quad + \frac 14t|g_2(X_0)-g_2(\bar X_0)|\nonumber\\
	&\quad + \int_0^t|U_\xi(s)-\bar U_\xi(s)|\:d s\nonumber\\
	&\quad + \frac 14\|\alpha'\|_\infty |\bar U_{0,\xi}|t\|H_{0,\xi}-\bar H_{0,\xi}\|_1\nonumber\\
	&\quad + \|\alpha'\|_\infty \bar H_{0,\xi}\int_0^t\|U(s)-\bar U(s)\|_\infty\:d s.
\end{align}
The result is obtained after some direct computations by applying the $L^2(\R)$-norm to both sides of the above estimate and using \eqref{bound U_0xi norm}.
\end{proof}

A Lipschitz theorem can now be stated for the flow in Lagrangian coordinates. The idea is to apply the estimates in Lemmas \ref{time_Z}, \ref{lemma V_xi and bar V_xi}, \ref{lemma y and U and UH_xi}, and \ref{lemma g-functions} together with Gronwall's inequality.
\begin{theorem}
\label{theorem lipschitz lagrangian}
For solutions $X(t),\bar X(t)$ of \eqref{characteristic_system_a_dissipative} with initial data $X_0\in\F_i^{\alpha,M}$ and $\bar X_0\in\F_{i,0}^{\alpha,M}$ the estimate
\begin{equation}
\tilde d(X(t),\bar X(t))\leq C_{M,\alpha}(t)\tilde d(X_0,\bar X_0),
\end{equation}
holds, with  $C_{M,\alpha}(t) = \tilde C_{M,\alpha}(t)e^{t\bar C_{M,\alpha}(t)}$, where
\begin{align}
\tilde C_{M,\alpha}(t) &= 3+\frac 32t+\frac 12t^2+\frac 3{16}t^3+\sqrt M\left(1+\frac 14t+\frac 14t^2+\frac 1{16}t^3\right)\nonumber\\
	&\quad + \|\alpha'\|_\infty\sqrt M\left(5+2t+t^2+\frac 38t^3\right) + \|\alpha'\|_\infty M\left(3+\frac 54t+\frac 12t^2+\frac 18t^3\right) ,\\
\bar C_{M,\alpha}(t) &= 2+\|\alpha'\|_\infty \sqrt M +\sqrt M\left(\frac 12+\frac 18t+\frac 1{16}t^2\right)+\|\alpha'\|_\infty M\left(1+\frac 14t+\frac 18t^2\right).
\end{align}
\end{theorem}
\begin{proof}
Combining Lemma \ref{time_Z}, \ref{lemma y and U and UH_xi}, and \ref{lemma g-functions} we end up with 
\begin{align}
\tilde d(X(t),\bar X(t)) &\leq \tilde d(X_0,\bar X_0)\nonumber\\
	&\quad + \|y_{0,\xi}-\bar y_{0,\xi}\|_2\nonumber\\
	&\quad + \|\alpha'\|_\infty M\left(1+ \frac 34t\right)\|U_{0,\xi}-\bar U_{0,\xi}\|_2\nonumber\\
	&\quad + \left(1+\frac12 t+\|\alpha'\|_\infty M\left(\frac 34t+\frac 38t^2\right)\right)\|H_{0,\xi}-\bar H_{0,\xi}\|_2\nonumber\\
	&\quad + \|\alpha'\|_\infty \sqrt M \|y_0-\bar y_0\|_\infty\nonumber\\
	&\quad + \|\alpha'\|_\infty\sqrt M\left(1+\frac 12t\right)\|H_{0,\xi}-\bar H_{0,\xi}\|_1\nonumber\\
	&\quad + \frac 32t\|\alpha'\|_\infty\|U_0H_{0,\xi}-\bar U_0\bar H_{0,\xi}\|_2\nonumber\\
	&\quad + \frac 12t\|g_2(X_0)-g_2(\bar X_0)\|_2\nonumber\\
	&\quad + 2\|g_3(X_0)-g_3(\bar X_0)\|_2\nonumber\\
	&\quad + 2\int_0^t\|U_\xi(s)-\bar U_\xi(s)\|_2\:d s\nonumber\\
	&\quad + \left(1+\|\alpha'\|_\infty\sqrt M\right)\int_0^t\|U(s)-\bar U(s)\|_\infty\:d s\nonumber\\
	&\quad + \frac 12\int_0^t(\|g(X(s))-g(\bar X(s))\|_2+\|y_\xi(s)-\bar y_\xi(s)\|_2)\:d s\nonumber\\
	&\quad + \left(\frac 14+\frac 12\|\alpha'\|_\infty\sqrt M\right)\int_0^t\int_\R|V_\xi(\xi,s)-\bar V_\xi(\xi,s)|d\xi\:d s.
\end{align}
Inserting the estimate from Lemma \ref{lemma V_xi and bar V_xi} yields
\begin{equation}
\tilde d(X(t),\bar X(t)) \leq \tilde C_{M,\alpha}(t)\tilde d(X_0,\bar X_0) + \bar C_{M,\alpha}(t)\int_0^t \tilde d(X(s),\bar X(s))\:d s,
\end{equation}
with
\begin{align}
\tilde C_{M,\alpha}(t) &= 3+\frac 32t+\frac 12t^2+\frac 3{16}t^3+\sqrt M\left(1+\frac 14t+\frac 14t^2+\frac 1{16}t^3\right)\nonumber\\
	&\quad + \|\alpha'\|_\infty\sqrt M\left(5+2t+t^2+\frac 38t^3\right) + \|\alpha'\|_\infty M\left(3+\frac 54t+\frac 12t^2+\frac 18t^3\right) ,\\
\bar C_{M,\alpha}(t) &= 2+\|\alpha'\|_\infty \sqrt M +\sqrt M\left(\frac 12+\frac 18t+\frac 1{16}t^2\right)+\|\alpha'\|_\infty M\left(1+\frac 14t+\frac 18t^2\right).
\end{align}
The theorem follows from Gronwall's inequality with $C_{M,\alpha}(t) = \tilde C_{M,\alpha}(t)e^{t\bar C_{M,\alpha}(t)}$.
\end{proof}

We are now ready to prove the Lipschitz continuity on $\F_{i,0}^{\alpha,M}$. The ingredients for the proof are Theorem \ref{theorem lipschitz lagrangian} and Lemma \ref{lemma J relabeling f}.
\begin{theorem}
\label{lipschitz_theorem}
Let $X,\bar{X}\in\F_{i,0}^{\alpha,M}$, then for all $t\geq 0$ it holds that
\begin{equation}
\label{lipschitz_equation}
d_M\left(t,X,\bar{X}\right) \leq \hat C_{M,\alpha}(t)d_M(0,X,\bar{X}),
\end{equation}
where $\hat C_{M,\alpha}(t)$ is given by $\hat C_{M,\alpha}=e^{\frac 12t}\tilde C_{M,\alpha}(t)e^{t\bar C_{M,\alpha}(t)}$,
\begin{align}
\tilde C_{M,\alpha}(t) &= 3+\frac 32t+\frac 12t^2+\frac 3{16}t^3+\sqrt M\left(1+\frac 14t+\frac 14t^2+\frac 1{16}t^3\right)\nonumber\\
	&\quad + \|\alpha'\|_\infty\sqrt M\left(5+2t+t^2+\frac 38t^3\right) + \|\alpha'\|_\infty M\left(3+\frac 54t+\frac 12t^2+\frac 18t^3\right) ,\\
\bar C_{M,\alpha}(t) &= 2+\|\alpha'\|_\infty \sqrt M +\sqrt M\left(\frac 12+\frac 18t+\frac 1{16}t^2\right)+\|\alpha'\|_\infty M\left(1+\frac 14t+\frac 18t^2\right).
\end{align}
\end{theorem}
\begin{proof}
Let $1>\epsilon>0$ and $X,\bar{X}\in\F_{i,0}^{\alpha,M}$ be given and choose $\{X_n\}_{n=0}^N$ in $\F_{i,0}^{\alpha,M}$, $\{f_n\}_{n=1}^N$, and $\{g_n\}_{n=0}^{N-1}$ in $G$ such that $X_0=X, X_N=\bar{X}$ and $d_M(0,X,\bar{X})+\epsilon \geq \sum_{n=1}^N \tilde d(X_n\bullet f_n,X_{n-1})+\tilde{d}(X_n,X_{n-1}\bullet g_{n-1})$. Then from the definition of $d_M$ we have
\begin{align}
d_M\big(t,X,\bar{X}\big) &\leq \sum_{n=1}^N J\big(\Pi S_t(X_n),\Pi S_t(X_{n-1})\big) \nonumber\\ 
	&\leq e^{\frac 12t}\sum_{n=1}^N J\big(S_t(X_n), S_t(X_{n-1})\big),
\end{align}
by Lemma \ref{lemma J relabeling f}. From the definition of $J$ and Theorem \ref{theorem lipschitz lagrangian} we get 
\begin{align}
d_M\big(t,X,\bar{X}\big)	&\leq e^{\frac 12t}\sum_{n=1}^N J\big(S_t(X_n),S_t(X_{n-1})\big)\nonumber\\
	&\leq e^{\frac 12t}C_{M,\alpha}(t)\sum_{n=1}^N\big(\tilde d(X_n\bullet f_n,X_{n-1})
	+\tilde d(X_n,X_{n-1}\bullet g_{n-1})\big)\nonumber\\
	&\leq \hat C_{M,\alpha}(t)\big(d_M(0,X,\bar{X})+\epsilon\big).
\end{align}
The inequality holds for each $\epsilon$ in the range $(0,1)$, which implies that
\begin{equation}
d_M\left(\Pi S_t(X(t)),\Pi S_t(\bar{X}(t))\right) \leq \hat C_{M,\alpha}(t)d_M(0,X,\bar{X}).
\end{equation}
\end{proof}
Since $\F_{i,0}^{\alpha,M}$ is in one to one correspondance with $\D_0^{\alpha,M}$ the metric $d_M$ on $\F_{i,0}^{\alpha,M}$ induces a metric $d_{\D_M}$ on $\D_0^{\alpha,M}$.
\begin{definition}
Let $(u, \rho, \nu, \mu)$, $(\bar u, \bar\rho, \bar\nu, \bar \mu)\in \D_0^{\alpha,M}$, then define $d_{\D_0^{\alpha,M}}:[0,\infty)\times \D_0^{\alpha,M}\times \D_0^{\alpha,M}\to \R$ by
\begin{equation}
d_{\D_0^{\alpha,M}}(t, (u, \rho, \nu,\mu), (\bar u, \bar \rho, \bar \nu, \bar \mu))=d_M(t, L((u, \rho, \nu, \mu)), L((\bar u, \bar\rho, \bar \nu, \bar \mu)))
\end{equation}
for any $(u,\rho,\nu,\mu),(\bar{u},\bar{\rho},\bar{\nu},\bar\mu)\in\D_0^{\alpha, M}$ and $t\geq 0$.
\end{definition}
\begin{theorem}
The $\alpha$-dissipative solution operator $T_t$ is Lipschitz continuous in the sense that for any  $(u_0,\rho_0,\nu_0,\mu_0),(\bar{u}_0,\bar{\rho}_0,\bar{\nu}_0,\bar\mu_0)\in\D_0^{\alpha,M}$ the inequality
\begin{align}
&d_{\D_0^{\alpha,M}}\big(t,(u_0,\rho_0,\nu_0,\mu_0),(\bar{u}_0,\bar{\rho}_0,\bar{\nu}_0,\bar\mu_0)\big)\nonumber\\
&\qquad \qquad\leq \hat C_{M,\alpha}(t)d_{\D_0^{\alpha,M}}\big(0,(u_0,\rho_0,\nu_0,\mu_0),(\bar{u}_0,\bar{\rho}_0,\bar{\nu}_0,\bar\mu_0)\big)
\end{align}
holds. The Lipschitz constant $\hat C_{M,\alpha}(t)$ is given inTheorem \ref{lipschitz_theorem}.
\end{theorem}
\begin{proof}
The theorem follows from Theorem \ref{existence theorem} and Theorem \ref{lipschitz_theorem}.
\end{proof}

\appendix
\section{Examples}\label{section examples}
\begin{example}
\label{example alpha dissipative}
In this example we construct an $\alpha$-dissipative solution. Let the initial data $(u_0,\rho_0,\nu_0,\mu_0)\in\D_0^\alpha$ be given by
\begin{align}
u_0(x) &= \begin{cases} 0, &x\leq -1,\\ x+1, &-1\leq x\leq 0,\\ -x+1,&0\leq x\leq 1,\\ 0,&1\leq x,\end{cases}\\
\rho_0(x) &= 0,\\
\nu_0 &= \mu_{0,ac} = u_{0,x}^2(x)\:d x = \mu_0,
\end{align}
and $\alpha(x) = \frac{x^2}{x^2+1}$. Then the initial data in Lagrangian coordinates is given by $(y_0,U_0,H_0,r_0,V_0) = L((u_0,\rho_0,\nu_0,\mu_0))$ where
\begin{align}
y_0(\xi) &= \begin{cases} \xi, &\xi\leq -1,\\ \frac 12(\xi-1), &-1\leq\xi\leq 3,\\ \xi-2, &3\leq\xi,\end{cases}\\
H_0(\xi) &= \begin{cases} 0, &\xi\leq -1,\\ \frac 12(\xi+1), &-1\leq\xi\leq 3,\\ 2, &3\leq\xi,\end{cases}\\
U_0(\xi) &= \begin{cases} 0, &\xi\leq -1,\\ \frac 12(\xi+1), &-1\leq\xi\leq 1,\\ \frac 12(-\xi+3), &1\leq\xi\leq 3,\\ 0, &3\leq\xi,\end{cases}\\
r_0(\xi) &= 0,\\
V_0(\xi) &= H_0(\xi).\\
\end{align}
Moreover the wave breaking time $\tau$ as a function of $\xi$ is given by
\begin{equation}
\tau(\xi) = \begin{cases}\infty, &\xi<1,\\ 2, &1<\xi<3,\\ \infty, & 3<\xi.\end{cases}
\end{equation}
Then for $t<2$ the solution is given by
\begin{align}
y(\xi,t) &=
	\begin{cases}
	\xi-\frac 14t^2, &\xi\leq -1,\\
	\frac 12(\xi-1)+\frac12(\xi+1)t+\frac 18(\xi-1)t^2, &-1\leq\xi\leq 1,\\
	\frac 12(\xi-1)-\frac 12(\xi-3)t+\frac 18(\xi-1)t^2, &1\leq\xi\leq 3,\\
	\xi-2+\frac 14t^2, &3\leq\xi,
	\end{cases}\\
U(\xi,t) &=
	\begin{cases}
	-\frac 12 t, &\xi\leq-1,\\
	\frac 12(\xi+1)+\frac 14(\xi-1)t, &-1\leq\xi\leq 1,\\
	-\frac 12(\xi-1)+\frac 14(\xi-1)t, &1\leq\xi\leq 3,\\
	\frac 12 t, &3\leq\xi,
	\end{cases}\\
H(\xi,t) &= H_0(\xi),\\
r(\xi,t) &= 0,\\
V(\xi,t) &= H_0(\xi).
\end{align}
In Eulerian coordinates the solution for $t<2$ is given by
\begin{align}
u(x,t) &=
	\begin{cases}
	-\frac 12 t, & x\leq -\frac 14t^2-1,\\
	\frac{x-\frac 12t+1}{1+\frac 12t}, & -\frac 14t^2-1\leq x\leq t,\\
	\frac{-x+\frac 12 t+1}{1-\frac 12t}, & t\leq x\leq \frac 14t^2+1,\\
	\frac 12 t, &\frac 14t^2+1\leq x,
	\end{cases}\\
d\nu(t) &= d\mu(t) = u_x^2(t)\:d x.
\end{align}
As $t\rightarrow 2$ we note that $y(\xi,t)\rightarrow 2$ for $\xi\in(1,3)$. Hence all wave breaking takes place at the coordinates $(x=2,t=2)$. We have
\begin{equation}
\alpha(2) = \frac 45.
\end{equation}
At $t=2$ the solution in Lagrangian coordinates reads
\begin{align}
y(\xi,2) &=
	\begin{cases}
	\xi-1, &\xi\leq -1,\\
	2\xi, &-1\leq\xi\leq 1,\\
	2, &1\leq\xi\leq 3,\\
	\xi-1, &3\leq\xi,
	\end{cases}\\
U(\xi,2) &=
	\begin{cases}
	-1, &\xi\leq-1,\\
	\xi, &-1\leq\xi\leq 1,\\
	1, &1\leq\xi\leq 3,\\
	1, &3\leq\xi,
	\end{cases}\\
H(\xi,2) &= H_0(\xi),\\
r(\xi,2) &= 0,\\
V(\xi,2) &=
	\begin{cases}
	0, &\xi\leq -1,\\
	\frac 12(\xi+1), &-1\leq\xi\leq 1,\\
	\frac{1}{10}(\xi+9), &1\leq\xi\leq 3,\\
	\frac 65, &3\leq\xi,
	\end{cases}
\end{align}
and thus the solution in Lagrangian coordinates for $t\geq 2$ is given by
\begin{align}
y(\xi,t) &=
	\begin{cases}
	-\frac{3}{20}t^2-\frac{2}{5}t+\xi+\frac{2}{5}, &\xi\leq -1,\\
	\frac 12(\xi-\frac 15)+\frac 12(\xi+\frac 15)t+\frac 18(\xi-\frac 15)t^2, &-1\leq\xi\leq 1,\\
	\frac{1}{10}(\xi+3)-\frac{1}{10}(\xi-7)t+\frac{1}{40}(\xi+3)t^2, &1\leq\xi\leq 3,\\
	\frac{3}{20}t^2+\frac 25t+\xi-\frac{12}{5}, &3\leq\xi,
	\end{cases}\\
U(\xi,t) &=
	\begin{cases}
	-\frac{3}{10}t-\frac 25, &\xi\leq -1,\\
	\frac 12(\xi+\frac 15)+\frac 14(\xi-\frac 15)t, &-1\leq\xi\leq 1,\\
	-\frac{1}{10}(\xi-7)+\frac{1}{20}(\xi+3)t, & 1\leq\xi\leq 3,\\
	\frac{3}{10}t+\frac 25, & 3\leq\xi,
	\end{cases}\\
H(\xi,t) &= H(\xi,0),\\
r(\xi,t) &= 0,\\
V(\xi,t) &=
	\begin{cases}
	0, &\xi\leq -1,\\
	\frac 12(\xi+1), &-1\leq\xi\leq 1,\\
	\frac{1}{10}(\xi+9), &1\leq\xi\leq 3,\\
	\frac 65, &3\leq\xi.
	\end{cases}
\end{align}
The solution in Eulerian coordinates for $t\geq 2$ is then given by
\begin{align}
u(x,t) &=
	\begin{cases}
	-\frac{3}{10}t-\frac{2}{5}, & x\leq-\frac{3}{20}t^2-\frac 25t-\frac 35,\\
	\frac{x+\frac 15(1-\frac 12t)}{1+\frac 12t}, & -\frac{3}{20}t^2-\frac 25t-\frac 35\leq x\leq \frac{1}{10}t^2+\frac 35 t+ \frac 25,\\
	\frac{x-\frac 12t - 1}{\frac 12t-1}, & \frac{1}{10}t^2+\frac 35 t+ \frac 25\leq x\leq \frac{3}{20}t^2+\frac 25 t+\frac 35,\\
	\frac{3}{10}t+\frac 25, & \frac{3}{20}t^2+\frac 25 t+\frac 35\leq x,
	\end{cases}\\
\rho(x,t) &= 0,\\
d\nu(t) &= u_x(x,t)^2d x \nonumber\\ &\quad + \delta_{(x=2,t=2)} + \frac{4}{(\frac 12t-1)^2}\mathbf{1}_{[\frac{1}{10}t^2+\frac 35t + \frac 25, \frac{3}{20}t^2+\frac 25t+\frac 35]}d x\\
d\mu(t) &= u_x(x,t)^2\:d x + \frac 15\delta_{(x=2,t=2)}.
\end{align}
In Figure \ref{figure example adiss} curves of $y(\xi,t)$ are drawn for $\xi=-1,1,3$. The cumulative energy function for the measures $\mu$ and $\nu$ are plotted for $t=0,2,4$ in Figure \ref{figure energy example adiss}.
\begin{figure}
\includegraphics[scale=0.8 ]{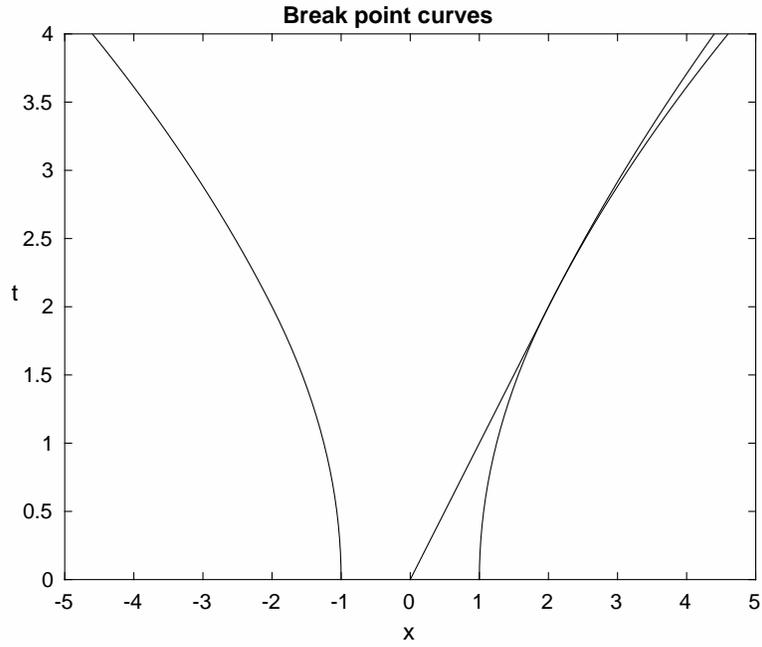}
\caption{Plots of $y(\xi,t)$ in Example \ref{example alpha dissipative} for $\xi=-1,1,3$.}
\label{figure example adiss}
\end{figure}
\begin{figure}
\includegraphics[scale=0.8 ]{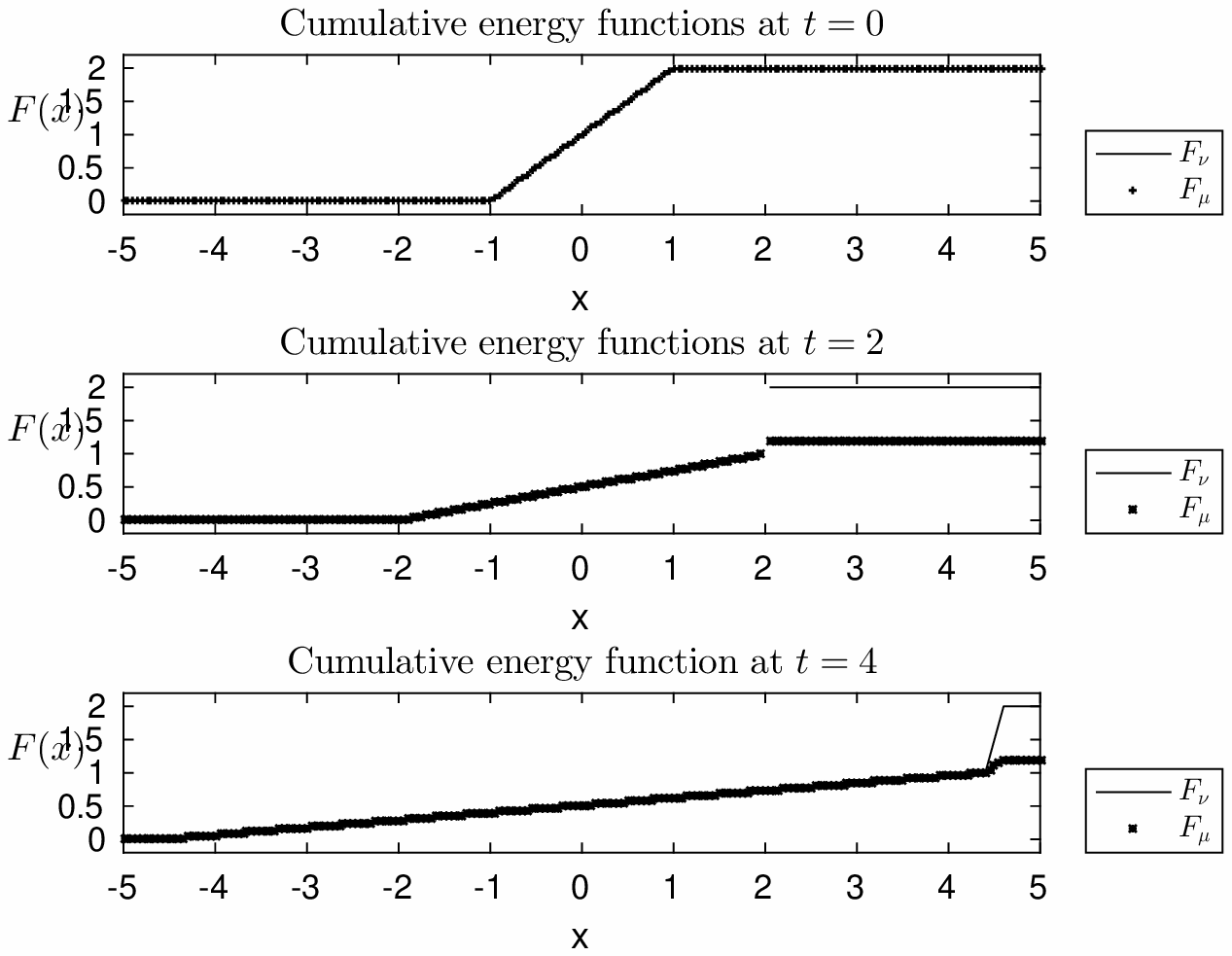}
\caption{Plots of the cumulative energy functions in Example \ref{example alpha dissipative} for $t=0,2,4$.}
\label{figure energy example adiss}
\end{figure}
\end{example}

\begin{example}
\label{example why g}
In this example we show that the function $g$ in the construction of the metric is necessary. Let $\alpha\in [0,1]$ be a given constant, choose $\epsilon \in (0,1)$, and let $(u_0,\rho_0,\nu_0,\mu_0)$ and $(\bar u_0,\bar \rho_0,\bar\nu_0,\bar \mu_0)$ be given by
\begin{align}
u_0(x) &=
	\begin{cases}
	0, &x\leq -1,\\
	x+1, &-1\leq x\leq 0,\\
	1-x,& 0\leq x\leq\epsilon,\\
	1-\epsilon,&\epsilon\leq x,
	\end{cases}\\
\bar u_0(x) &=
	\begin{cases}
	-2\epsilon,&x\leq -1,\\
	x+1-2\epsilon,&-1\leq x\leq\epsilon,\\
	1-\epsilon,&\epsilon\leq x,
	\end{cases}\\
\rho_0 &= \bar\rho_0 = 0,\\
\nu_0 &=\bar\nu_0 = u_{0,x}^2\:d x=\bar \mu_0=\mu_0.
\end{align}
Then $X_0 = L((u_0,\rho_0,\nu_0,\mu_0))$ and $\bar{X}_0 = L((\bar u_0,\bar\rho_0,\bar\nu_0,\bar\mu_0))$ are given by
\begin{align}
y_0(\xi) = \bar{y}_0(\xi) &=
	\begin{cases}
	\xi, &\xi\leq -1,\\
	\frac 12(\xi-1), &-1\leq\xi\leq 1+2\epsilon,\\
	\xi-1-\epsilon, & 1+2\epsilon\leq\xi,
	\end{cases}\\
U_0(\xi) &=
	\begin{cases}
	0, &\xi\leq-1,\\
	\frac 12(\xi+1),&-1\leq\xi\leq 1,\\
	-\frac 12(\xi-3), & 1\leq\xi\leq 1+2\epsilon,\\
	1-\epsilon, & 1+2\epsilon\leq\xi,
	\end{cases}\\
\bar{U}_0(\xi) &=
	\begin{cases}
	-2\epsilon, &\xi\leq-1,\\
	\frac 12(\xi+1-4\epsilon),&-1\leq\xi\leq 1+2\epsilon,\\
	1-\epsilon, & 1+2\epsilon\leq\xi,
	\end{cases}\\
H_0(\xi) = \bar{H}_0(\xi) &= 
	\begin{cases}
	0, &\xi\leq -1,\\
	\frac 12(\xi+1), &-1\leq\xi\leq 1+2\epsilon,\\
	1+\epsilon, & 1+2\epsilon\leq\xi,
	\end{cases}\\
r_0(\xi) = \bar{r}_0(\xi) &= 0.
\end{align} 
Let $X(t)=S_t(X_0)$ and $\bar{X}(t) = S_t(\bar{X}_0)$, then for $0\leq t < 2$ we have
\begin{align}
y(\xi,t) &= 
	\begin{cases}
	\xi-\frac 18(1+\epsilon)t^2, &\xi\leq -1,\\
	\frac 12(\xi-1)+\frac 12(\xi+1)t + \frac 18(\xi-\epsilon)t^2, &-1\leq\xi\leq 1,\\
	\frac 12(\xi-1)-\frac 12(\xi-3)t + \frac 18(\xi-\epsilon)t^2, & 1\leq\xi\leq 1+2\epsilon,\\
	\xi-1-\epsilon+(1-\epsilon)t + \frac 18(1+\epsilon)t^2, & 1+2\epsilon\leq\xi,
	\end{cases}\\
\bar{y}(\xi,t) &=
	\begin{cases}
	\xi-2\epsilon t-\frac 18(1+\epsilon)t^2, &\xi\leq -1,\\
	\frac 12(\xi-1)+\frac 12(\xi+1-4\epsilon)t + \frac 18(\xi-\epsilon)t^2, &-1\leq\xi\leq 1+2\epsilon,\\
	\xi-1-\epsilon + (1-\epsilon)t + \frac 18(1+\epsilon)t^2, & 1+2\epsilon\leq\xi,
	\end{cases}\\
U(\xi,t) &=
	\begin{cases}
	-\frac 14(1+\epsilon)t, &\xi\leq-1,\\
	\frac 12(\xi+1)+\frac 14(\xi-\epsilon)t,&-1\leq\xi\leq 1,\\
	-\frac 12(\xi-3)+\frac 14(\xi-\epsilon)t, & 1\leq\xi\leq 1+2\epsilon,\\
	1-\epsilon + \frac 14(1+\epsilon)t, & 1+2\epsilon\leq\xi,
	\end{cases}\\
\bar{U}(\xi,t) &=
	\begin{cases}
	-2\epsilon-\frac 14(1+\epsilon)t, &\xi\leq-1,\\
	\frac 12(\xi+1)-2\epsilon+\frac 14(\xi-\epsilon)t,&-1\leq\xi\leq 1+2\epsilon,\\
	1-\epsilon+\frac 14(1+\epsilon)t, & 1+2\epsilon\leq\xi.
	\end{cases}
\end{align}
Then $M(X(t))$, and $M(\bar X(t))$, the solutions in Eulerian coordinates, are given by
\begin{align}
x_1(t) &= -1-\frac 18(1+\epsilon)t^2,\\
x_2(t) &= t+\frac 18(1-\epsilon)t^2,\\
x_3(t) &= \epsilon +(1-\epsilon)t + \frac 18(1+\epsilon)t^2,\\
u(x,t) &=
	\begin{cases}
	-\frac 14(1+\epsilon)t, & x\leq x_1(t),\\
	\frac{1}{1+\frac 12t}\left(x+1-\frac 14(1+\epsilon)t\right),& x_1(t)\leq x\leq x_2(t),\\
	\frac{1}{1-\frac 12t}\left(1-x+\frac 14(3-\epsilon)t\right),& x_2(t)\leq x\leq x_3(t),\\
	1-\epsilon+\frac 14(1+\epsilon)t,& x_3(t)\leq x,
	\end{cases}\\
\rho(x,t) &= 0,\\
d\nu(t) &= \frac{1}{1+t+\frac 14t^2}\mathbf 1_{[x_1(t),x_2(t)]}\:d x + \frac{1}{1-t+\frac 14t^2}\mathbf{1}_{[x_2(t),x_3(t)]}\:d x,\\
\mu(t) &= \nu(t),\\
\bar x_1(t) &= -1-2\epsilon t -\frac 18(1+\epsilon)t^2,\\
\bar x_2(t) &= \epsilon +(1-\epsilon)t + \frac 18(1+\epsilon)t^2,\\
\bar u(x,t) &=
	\begin{cases}
	-2\epsilon-\frac 14(1+\epsilon)t, & x\leq \bar x_1(t),\\
	\frac{1}{1+\frac 12t}\left(x+1-2\epsilon -\frac 14t(1-3\epsilon)\right), & \bar x_1(t)\leq x\leq \bar x_2(t),\\ 
	1-\epsilon+\frac 14(1+\epsilon)t, & \bar x_2(t)\leq x,\\
	\end{cases}\\
\bar\rho(x,t) &= 0,\\
d\bar\nu(t) &=  \frac{1}{1+t+\frac 14t^2}\mathbf{1}_{[\bar x_1(t),\bar x_2(t)]}\:d x,\\
\bar\mu(t) &= \bar \nu(t),
\end{align}
For $2\leq t$ the solutions in Lagrangian coordinates are given by
\begin{align}
y(\xi,t) &= 
	\begin{cases}
	\xi+\frac{\alpha\epsilon}{2}-\frac{\alpha\epsilon}{2}t -\frac 18(1+(1-\alpha)\epsilon)t^2, &\xi\leq -1,\\
	\frac 12(\xi-1)+\frac{\alpha\epsilon}{2}+\frac 12(\xi+1-\alpha\epsilon)t+\frac 18(\xi-(1-\alpha)\epsilon)t^2, &-1\leq\xi\leq 1,\\
	\frac{1-\alpha}{2}(\xi-1)+\frac{\alpha\epsilon}{2} +(-\frac{1-\alpha}{2}(\xi-1)+1-\frac{\alpha\epsilon}{2})t\\ \qquad+ \frac 18((1-\alpha)(\xi-1-\epsilon)+1)t^2, & 1\leq\xi\leq 1+2\epsilon,\\
	\xi-1-\epsilon-\frac{\alpha\epsilon}{2}+(1-\epsilon+\frac{\alpha\epsilon}{2})t + \frac 18(1+(1-\alpha)\epsilon)t^2, & 1+2\epsilon\leq\xi,
	\end{cases}\\
\bar{y}(\xi,t) &=
	\begin{cases}
	\xi-2\epsilon t-\frac 18(1+\epsilon)t^2, &\xi\leq -1,\\
	\frac 12(\xi-1)+\frac 12(\xi+1-4\epsilon)t + \frac 18(\xi-\epsilon)t^2, &-1\leq\xi\leq 1+2\epsilon,\\
	\xi-1-\epsilon + (1-\epsilon)t + \frac 18(1+\epsilon)t^2, & 1+2\epsilon\leq\xi,
	\end{cases}\\
U(\xi,t) &=
	\begin{cases}
	-\frac{\alpha\epsilon}{2} -\frac 14(1+(1-\alpha)\epsilon)t, &\xi\leq-1,\\
	\frac 12(\xi+1-\alpha\epsilon)+\frac 14(\xi-(1-\alpha)\epsilon)t,&-1\leq\xi\leq 1,\\
	-\frac{1-\alpha}{2}(\xi-1)+1-\frac{\alpha\epsilon}{2} + \frac 14((1-\alpha)(\xi-1-\epsilon)+1)t, & 1\leq\xi\leq 1+2\epsilon,\\
	1-\epsilon+\frac{\alpha\epsilon}{2} + \frac 14(1+(1-\alpha)\epsilon)t, & 1+2\epsilon\leq\xi,
	\end{cases}\\
\bar{U}(\xi,t) &=
	\begin{cases}
	-2\epsilon-\frac 14(1+\epsilon)t, &\xi\leq-1,\\
	\frac 12(\xi+1)-2\epsilon+\frac 14(\xi-\epsilon)t,&-1\leq\xi\leq 1+2\epsilon,\\
	1-\epsilon+\frac 14(1+\epsilon)t, & 1+2\epsilon\leq\xi,
	\end{cases}\\
V(\xi,t) &=
	\begin{cases}
	0, &\xi\leq -1,\\
	\frac 12(\xi+1), &-1\leq\xi\leq 1,\\
	\frac{1-\alpha}{2}\xi+\frac{1+\alpha}{2}, &1\leq\xi\leq 1+2\epsilon,\\
	1+(1-\alpha)\epsilon, & 1+2\epsilon\leq\xi,
	\end{cases}\\
\bar V(\xi,t) &= \bar H_0(\xi),\\
H(\xi,t) &= H_0(\xi),\\
\bar H(\xi,t) &= \bar H_0(\xi). 
\end{align}
The solutions in Eulerian coordinates for $2\leq t$ are given by
\begin{align}
\label{eq:ex g eulerian}
x_1(t) &= -1+\frac{\alpha\epsilon}{2}-\frac{\alpha\epsilon}{2}t -\frac 18(1+(1-\alpha)\epsilon)t^2,\\
x_2(t) &= \frac{\alpha\epsilon}{2} + (1-\frac{\alpha\epsilon}{2})t+\frac 18(1-(1-\alpha)\epsilon)t^2,\\
x_3(t) &= \epsilon-\frac{\alpha\epsilon}{2}+(1-\epsilon+\frac{\alpha\epsilon}{2})t + \frac 18(1+(1-\alpha)\epsilon)t^2,\\
u(x,t) &=
	\begin{cases}
	-\frac{\alpha\epsilon}{2}-\frac 14(1+(1-\alpha)\epsilon)t, & x\leq x_1(t),\\
	\frac{x-x_1(t)}{(1+\frac 12t)^2}-\frac{\alpha\epsilon}{2}-\frac 14(1+(1-\alpha)\epsilon)t, & x_1(t)\leq x\leq x_2(t),\\
	-\frac{1}{1-\frac 12t}\left(x-2-\frac 12(1-\epsilon)t+\frac 18(1-\epsilon)t^2\right),& x_2(t)\leq x\leq x_3(t),\\
	1-\epsilon+\frac{\alpha\epsilon}{2} + \frac 14(1+(1-\alpha)\epsilon)t, & x_3(t) \leq x,\\ 
		\end{cases}\\
\rho(x,t) &= 0,\\
d\nu(t) &= \frac{1}{1+t+\frac 14t^2}\mathbf 1_{[x_1(t),x_2(t)]}\:d x +\frac{1}{(1-\alpha)(1-t+\frac 14t^2)}\mathbf{1}_{[x_2(t),x_3(t)]}\:d x \nonumber\\ &\quad + \epsilon\delta_{(x=2+\frac{1-\epsilon}{2},t=2)},\\
d\mu(t) &= \frac{1}{1+t+\frac 14t^2}\mathbf 1_{[x_1(t),x_2(t)]}\:d x +\frac{1}{(1-t+\frac 14t^2)}\mathbf{1}_{[x_2(t),x_3(t)]}\:d x \nonumber\\ &\quad + (1-\alpha)\epsilon\delta_{(x=2+\frac{1-\epsilon}{2},t=2)},\\
\bar x_1(t) &= -1-2\epsilon t-\frac 18(1+\epsilon)t^2,\\
\bar x_2(t) &= \epsilon +(1-\epsilon)t + \frac 18(1+\epsilon)t^2,\\
\bar u(x,t) &=
	\begin{cases}
	-2\epsilon-\frac 14(1+\epsilon)t, & x\leq \bar x_1(t),\\
	\frac{1}{1+\frac 12t}\left(x+1-2\epsilon -\frac 14t(1-3\epsilon)\right),& \bar x_1(t)\leq x\leq \bar x_2(t),\\
	1-\epsilon+\frac 14(1+\epsilon)t, &\bar x_2(t)\leq x,\\
		\end{cases}\\
\bar\rho(x,t) &= 0,\\
d\bar\nu(t) &= \frac{1}{1+t+\frac 14t^2}\mathbf{1}_{[\bar x_1(t),\bar x_2(t)]}\:d x,\\
\bar\mu(t) &= \bar\nu(t),
\end{align}
The curves tracking the break points, $x_1$, $x_2$, $x_3$, $\bar x_1$, and  $\bar x_2$, is drawn in Figure \ref{figure example why g}. In Figure \ref{figure energy example why g} the cumulative function of the measures $\mu$ and $\bar\mu$ are plotted for $t=0,2,4$. Note that they do not coincide after wave breaking.

For $t<2$ the norms of $X(t)-\bar{X}(t)$ are given by
\begin{align}
\|y(t)-\bar{y}(t)\|_{\infty} &= 2\epsilon t,\\
\|U(t)-\bar{U}(t)\|_{\infty} &= 2\epsilon,\\
\|y_{\xi}(t)-\bar{y}_{\xi}(t)\|_2 &= \sqrt{2\epsilon}t,\\
\|U_{\xi}(t)-\bar{U}_{\xi}(t)\|_2 &= \sqrt{2\epsilon},\\
\|V_\xi(t)-\bar V_\xi(t)\|_2 &= 0,\\
\|g(X(t))-g(\bar{X})\|_2 &= \sqrt{2\epsilon}\left(t+\frac{\alpha}{2}\right),\\
\|g_2(X(t))-g_2(\bar X(t))\|_2 &= 0,\\
\|g_3(X(t))-g_3(\bar X(t))\|_2 &= 0,\\
\|\alpha'\|_\infty\|U(t)H_{0,\xi}-\bar U(t)\bar H_{0,\xi}\|_2 &= 0,
\end{align}
while for $2\leq t$ we have
\begin{align}
\|y(t)-\bar{y}(t)\|_{\infty} &= \epsilon\left(2t + \frac{\alpha}{8}(t-2)^2\right),\\
\|U(t)-\bar{U}(t)\|_{\infty} &= 2\epsilon + \frac{\alpha\epsilon}{4}(t-2),\\
\|y_{\xi}(t)-\bar{y}_{\xi}(t)\|_2 &= \sqrt{2\epsilon}\left(t+\frac{\alpha}{8}(t-2)^2\right),\\
\|U_{\xi}(t)-\bar{U}_{\xi}(t)\|_2 &= \sqrt{2\epsilon}\left(1+\frac{\alpha}{4}(t-2)\right),\\
\|V_\xi(t)-\bar V_\xi(t)\|_2 &= \sqrt{2\epsilon}\frac{\alpha}{2},\\
\|g(X)-g(\bar{X})\|_2 &= \sqrt{2\epsilon}\left(\frac{\alpha}{2} + t + \frac{\alpha}{8}(t-2)^2\right),\\
\|g_2(X(t))-g_2(\bar X(t))\|_2 &= 0,\\
\|g_3(X(t))-g_3(\bar X(t))\|_2 &= 0,\\
\|\alpha'\|_\infty\|U(t)H_{0,\xi}-\bar U(t)\bar H_{0,\xi}\|_2 &= 0.
\end{align}
That is, at $t=2$ the norm $\|U_{\xi}(t)-\bar{U}_{\xi}(t)\|_2$ suddenly starts to grow, and the growth depends on $\alpha$. Moreover, there is a jump in $\|V_\xi(t)-\bar V_\xi(t)\|_2$ at $t=2$, so that term cannot be a part of the metric. Note that since $H_0 = \bar H_0$ all terms involving differences in $H$ or $H_\xi$ vanish.
\begin{figure}
\includegraphics[scale=0.8  ]{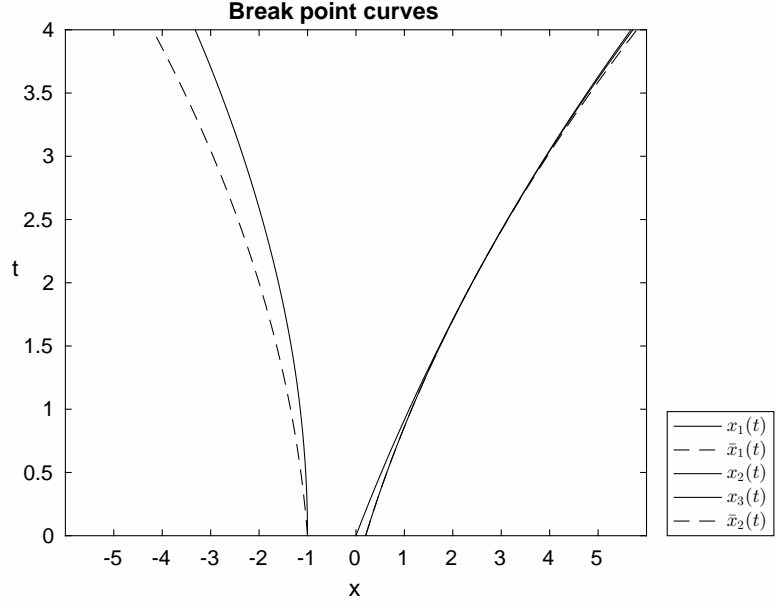}
\caption{A plot of characteristic curves $x_i(\xi,t), i=1,2,3$ and $\bar x_i(t), t=1,2$ from \eqref{eq:ex g eulerian} in Example \ref{example why g}.}
\label{figure example why g}
\end{figure}

\begin{figure}
\includegraphics[scale=0.8  ]{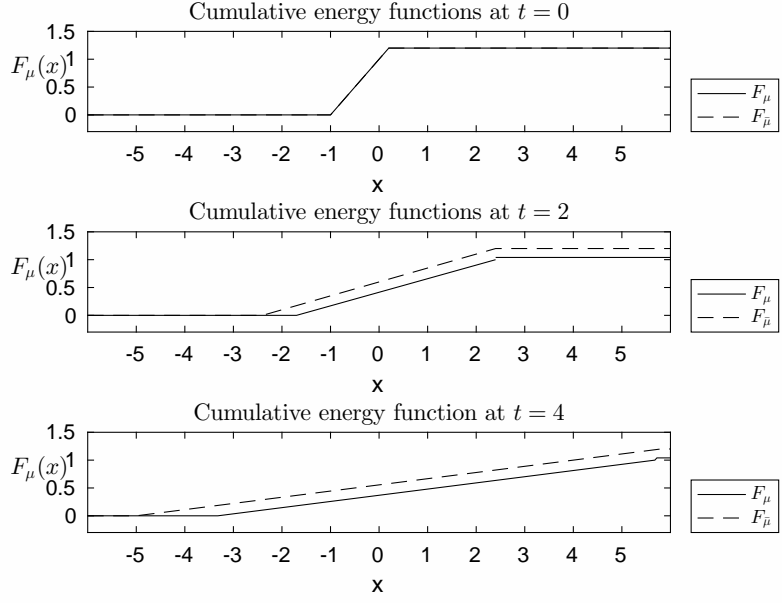}
\caption{Plots of the cumulative energy functions in Example \ref{example why g} for $t=0,2,4$ .}
\label{figure energy example why g}
\end{figure}
\end{example}

\begin{example}
\label{example iteration}
In this example we demonstrate the iteration scheme from the proof of Lemma \ref{solution_operator_lagrangian}. In particular we illustrate how the scheme reaches the limit in a finite number of iterations for multipeakon initial data. Let $X_0\in\F_0$ be given by
\begin{align}
y_0(\xi) &=
	\begin{cases}
	\xi, &\xi\leq -1,\\
	\frac 12(\xi-1), &-1\leq\xi\leq 1,\\
	\frac 45(\xi-1), &1\leq\xi\leq \frac 72,\\
	\xi-\frac 32, &\frac 72\leq\xi,
	\end{cases}\\
U_0(\xi) &=
	\begin{cases}
	1, &\xi\leq -1,\\
	-\frac 12(\xi-1), &-1\leq\xi\leq 1,\\
	-\frac 25(\xi-1), &1\leq\xi\leq \frac 72,\\
	-1, &\frac 72\leq\xi,
	\end{cases}\\
H_0(\xi) &=
	\begin{cases}
	0, &\xi\leq -1,\\
	\frac 12(\xi+1), &-1\leq\xi\leq 1,\\
	\frac 15(\xi+4), &1\leq\xi\leq \frac 72,\\
	\frac 32, &\frac 72\leq\xi,
	\end{cases}\\
V_0 &= H_0,\\
r_0 &= 0,
\end{align}
which gives
\begin{equation}
\tau(\xi) =
	\begin{cases}
	\infty,&\xi\leq -1,\\
	2, &-1\leq\xi\leq 1,\\
	4, &1\leq\xi\leq\frac 72,\\
	\infty, &\frac 72\leq\xi.
	\end{cases}
\end{equation}
Choose $\alpha(x)$ such that
\begin{equation}
\alpha(x) = 
	\begin{cases}
	0, &x\leq 0,\\
	\frac 14x, &0\leq x\leq 3,\\
	\frac 34, & 3\leq x.
	\end{cases}
\end{equation}
Then, in the notation of the proof of Lemma \ref{solution_operator_lagrangian}, we have $X_1(t) = X_0$. For $t<2$ the next term is given by
\begin{align}
y_2(\xi,t) &=
	\begin{cases}
	\xi+t-\frac{3}{16}t^2, &\xi\leq -1,\\
	\frac 12(\xi-1)-\frac 12(\xi-1)t+\frac 18(\xi-\frac 12)t^2, &-1\leq\xi\leq 1,\\
	\frac 45(\xi-1)-\frac 25(\xi-1)t+\frac 1{20}(\xi+\frac 14)t^2, &1\leq\xi\leq \frac 72,\\
	\xi-\frac 32-t+\frac 3{16}t^2, &\frac 72\leq\xi,
	\end{cases}\\
U_2(\xi,t) &=
	\begin{cases}
	1-\frac 38t, &\xi\leq -1,\\
	-\frac 12(\xi-1)+\frac 14(\xi-\frac 12)t, &-1\leq\xi\leq 1,\\
	-\frac 25(\xi-1)+\frac 1{10}(\xi+\frac 14)t, &1\leq\xi\leq \frac 72,\\
	-1+\frac 38t, &\frac 72\leq\xi,
	\end{cases}\\
H_2(\xi,t) &=
	\begin{cases}
	0, &\xi\leq -1,\\
	\frac 12(\xi+1), &-1\leq\xi\leq 1,\\
	\frac 15(\xi+4), &1\leq\xi\leq \frac 72,\\
	\frac 32, &\frac 72\leq\xi,
	\end{cases}\\
V_2(t) &= H_0,\\
r_2(t) &= 0.
\end{align}
The amount of energy dissipated at $t=2$ is given by
\begin{equation}
\alpha(y_1(\xi,\tau(\xi))) = 0,\qquad -1\leq\xi\leq 1,
\end{equation}
and hence we have for $2\leq t < 4$ that
\begin{align}
y_2(\xi,t) &=
	\begin{cases}
	\xi+t-\frac{3}{16}t^2, &\xi\leq -1,\\
	\frac 12(\xi-1)-\frac 12(\xi-1)t+\frac 18(\xi-\frac 12)t^2, &-1\leq\xi\leq 1,\\
	\frac 45(\xi-1)-\frac 25(\xi-1)t+\frac 1{20}(\xi+\frac 14)t^2, &1\leq\xi\leq \frac 72,\\
	\xi-\frac 32-t+\frac 3{16}t^2, &\frac 72\leq\xi,
	\end{cases}\\
U_2(\xi,t) &=
	\begin{cases}
	1-\frac 38t, &\xi\leq -1,\\
	-\frac 12(\xi-1)+\frac 14(\xi-\frac 12)t, &-1\leq\xi\leq 1,\\
	-\frac 25(\xi-1)+\frac 1{10}(\xi+\frac 14)t, &1\leq\xi\leq \frac 72,\\
	-1+\frac 38t, &\frac 72\leq\xi,
	\end{cases}\\
H_2(\xi,t) &=
	\begin{cases}
	0, &\xi\leq -1,\\
	\frac 12(\xi+1), &-1\leq\xi\leq 1,\\
	\frac 15(\xi+4), &1\leq\xi\leq \frac 72,\\
	\frac 32, &\frac 72\leq\xi,
	\end{cases}\\
V_2(t) &= H_0,\\
r_2(t) &= 0.
\end{align}
The amount dissipated at $t=4$ is given by
\begin{equation}
\alpha(y_1(\xi,\tau(\xi))) = \frac 15(\xi-1),\qquad 1\leq\xi\leq \frac 72.
\end{equation}
Thus, for $t\geq 4$,
\begin{align}
y_2(\xi,t) &=
	\begin{cases}
	\xi+1-\frac 12(t-4)-\frac{11}{64}(t-4)^2, &\xi\leq-1,\\
	\frac 12(\xi+1)+\frac 12\xi(t-4)+\frac 18(\xi-\frac{3}{8})(t-4)^2, &-1\leq\xi\leq 1,\\
	1+\frac 12(t-4)+\frac{1}{200}(\frac{37}8 + 12\xi-\xi^2)(t-4)^2, &1\leq\xi\leq \frac 72,\\
	\xi-\frac 52 +\frac 12(t-4)+\frac{11}{64}(t-4)^2, &\frac 72\leq\xi,
	\end{cases}\\
U_2(\xi,t) &=
	\begin{cases}
	-\frac 12-\frac{11}{32}(t-4), &\xi\leq -1,\\
	\frac 12\xi+\frac 14(\xi-\frac{3}{8})(t-4), &-1\leq\xi\leq 1,\\
	\frac 12 + \frac{1}{100}(\frac{37}8 + 12\xi-\xi^2)(t-4), &1\leq\xi\leq \frac 72,\\
	\frac 12+\frac{11}{32}(t-4), &\frac 72\leq\xi,
	\end{cases}\\
H_2(\xi,t) &=
	\begin{cases}
	0, &\xi\leq -1,\\
	\frac 12(\xi+1), &-1\leq\xi\leq 1,\\
	\frac 15(\xi+4), &1\leq\xi\leq \frac 72,\\
	\frac 32, &\frac 72\leq\xi,
	\end{cases}\\
V_2(\xi,t) &= 
	\begin{cases}
	0, &\xi\leq -1,\\
	\frac 12(\xi+1), &-1\leq\xi\leq 1,\\
	\frac 1{50}(39+12\xi-\xi^2), &1\leq\xi\leq \frac 72,\\
	\frac{11}{8}, &\frac 72\leq\xi,
	\end{cases}\\
r_2(t) &= 0.
\end{align}
Then
\begin{equation}
\alpha(y_2(\xi,\tau(\xi))) = 
	\begin{cases}
	\frac 1{16}, &-1\leq\xi\leq 1,\\
	\frac 14, &1\leq\xi\leq\frac 72,
	\end{cases}
\end{equation}
and $X_3(t)$ can be computed as follows. For $t<2$ we have $X_3(t) = X_2(t)$. For $2\leq t<4$ we have,
\begin{align}
y_3(\xi,t) &=
	\begin{cases}
	\xi+\frac 54 + \frac 14(t-2)-\frac{23}{128}(t-2)^2, &\xi\leq -1,\\
	\frac 14+\frac 14(t-2)+\frac 1{128}(15\xi-8)(t-2)^2, &-1\leq\xi\leq 1,\\
	\frac 15(\xi+\frac 14)-\frac 15(\xi-\frac 94)(t-2) + \frac 1{20}(\xi+\frac{3}{32})(t-2)^2, &1\leq\xi\leq \frac 72,\\
	\xi-\frac{11}{4}-\frac 14(t-2)+\frac{23}{128}(t-2)^2, &\frac 72\leq\xi,
	\end{cases}\\
U_3(\xi,t) &=
	\begin{cases}
	\frac 14-\frac{23}{64}(t-2), &\xi\leq -1,\\
	\frac 14+\frac 18(\frac{15}8\xi-1)(t-2), &-1\leq\xi\leq 1,\\
	-\frac 15(\xi-\frac 94)+\frac 1{10}(\xi+\frac{3}{32})(t-2), &1\leq\xi\leq \frac 72,\\
	-\frac 14+\frac{23}{64}(t-2), &\frac 72\leq\xi,
	\end{cases}\\
H_3(\xi,t) &=
	\begin{cases}
	0, &\xi\leq -1,\\
	\frac 12(\xi+1), &-1\leq\xi\leq 1,\\
	\frac 15(\xi+4), &1\leq\xi\leq \frac 72,\\
	\frac 32, &\frac 72\leq\xi,
	\end{cases}\\
V_3(\xi,t) &= 
	\begin{cases}
	0, &\xi\leq -1,\\
	\frac {15}{32}(\xi+1), &-1\leq\xi\leq 1,\\
	\frac 15\xi+\frac{59}{80}, &1\leq\xi\leq \frac 72,\\
	\frac{23}{16}, &\frac 72\leq\xi,
	\end{cases}\\
r_3(t) &= 0.
\end{align}
At $t=4$ there is wave breaking for $1<\xi<\frac 72$, and one fourth of the energy is dissipated. Thus for $t\geq 4$ we have
\begin{align}
y_3(\xi,t) &=
	\begin{cases}
	\xi+\frac{33}{32}-\frac{15}{32}(t-4)-\frac{21}{128}(t-4)^2, &\xi\leq -1,\\
	\frac{15}{32}\xi+\frac 12 + \frac{15}{32}\xi(t-4)+\frac 1{128}(15\xi-6)(t-4)^2, &-1\leq\xi\leq 1,\\
	\frac{31}{32}+\frac{15}{32}(t-4)+\frac 3{80}(\xi+\frac{7}8)(t-4)^2, &1\leq\xi\leq \frac 72,\\
	\xi-\frac{81}{32}+\frac{15}{32}(t-4)+\frac{21}{128}(t-4)^2, &\frac 72\leq\xi,
	\end{cases}\\
U_3(\xi,t) &=
	\begin{cases}
	-\frac{15}{32}-\frac{21}{64}(t-4), &\xi\leq -1,\\
	\frac{15}{32}\xi+\frac 1{64}(15\xi-6)(t-4), &-1\leq\xi\leq 1,\\
	\frac{15}{32}+\frac 3{40}(\xi+\frac{7}8)(t-4), &1\leq\xi\leq \frac 72,\\
	\frac{15}{32}+\frac{21}{64}(t-4), &\frac 72\leq\xi,
	\end{cases}\\
H_3(\xi,t) &=
	\begin{cases}
	0, &\xi\leq -1,\\
	\frac 12(\xi+1), &-1\leq\xi\leq 1,\\
	\frac 15(\xi+4), &1\leq\xi\leq \frac 72,\\
	\frac 32, &\frac 72\leq\xi,
	\end{cases}\\
V_3(\xi,t) &= 
	\begin{cases}
	0, &\xi\leq -1,\\
	\frac {15}{32}(\xi+1), &-1\leq\xi\leq 1,\\
	\frac{3}{20}(\xi+\frac{21}{4}), &1\leq\xi\leq \frac 72,\\
	\frac{21}{16}, &\frac 72\leq\xi,
	\end{cases}\\
r_3(t) &= 0.
\end{align}
Then
\begin{equation}
\alpha(y_3(\xi,\tau(\xi))) = 
	\begin{cases}
	\frac 1{16}, &-1\leq\xi\leq 1,\\
	\frac{31}{128}, &1\leq\xi\leq\frac 72,
	\end{cases}
\end{equation}
and we have that $X_4(t) = X_3(t)$ for $t<4$. When $t\geq 4$, then
\begin{align}
y_4(\xi,t) &=
	\begin{cases}
	\xi+\frac{33}{32}-\frac{15}{32}(t-4)-\frac{337}{2048}(t-4)^2, &\xi\leq -1,\\
	\frac{15}{32}\xi+\frac 12 + \frac{15}{32}\xi(t-4)+\frac 1{2048}(240\xi-97)(t-4)^2, &-1\leq\xi\leq 1,\\
	\frac{31}{32}+\frac{15}{32}(t-4)+\frac 1{10240}(388\xi+327)(t-4)^2, &1\leq\xi\leq \frac 72,\\
	\xi-\frac{81}{32}+\frac{15}{32}(t-4)+\frac{337}{2048}(t-4)^2, &\frac 72\leq\xi,
	\end{cases}\\
U_4(\xi,t) &=
	\begin{cases}
	-\frac{15}{32}-\frac{337}{1024}(t-4), &\xi\leq -1,\\
	\frac{15}{32}\xi+\frac 1{1024}(240\xi-97)(t-4), &-1\leq\xi\leq 1,\\
	\frac{15}{32}+\frac1{5120}(388\xi+327)(t-4), &1\leq\xi\leq \frac 72,\\
	\frac{15}{32}+\frac{337}{1024}(t-4), &\frac 72\leq\xi,
	\end{cases}\\
H_4(\xi,t) &=
	\begin{cases}
	0, &\xi\leq -1,\\
	\frac 12(\xi+1), &-1\leq\xi\leq 1,\\
	\frac 15(\xi+4), &1\leq\xi\leq \frac 72,\\
	\frac 32, &\frac 72\leq\xi,
	\end{cases}\\
V_4(\xi,t) &= 
	\begin{cases}
	0, &\xi\leq -1,\\
	\frac {15}{32}(\xi+1), &-1\leq\xi\leq 1,\\
	\frac{1}{640}(97\xi+503), &1\leq\xi\leq \frac 72,\\
	\frac{337}{256}, &\frac 72\leq\xi,
	\end{cases}\\
r_4(t) &= 0.
\end{align}
Since $X_5(t) = X_4(t)$ for all $t>0$, the $\alpha$-dissipative solution with initial data $X_0$ at time $t$ is given by $X_5(t)$. In Figure \ref{figure iteration} $y_i(\xi,t)$ has been plotted for $i=1,2,3,4$ and $\xi=-1,1,\frac 72$. 
\begin{figure}
\includegraphics[scale=0.8  ]{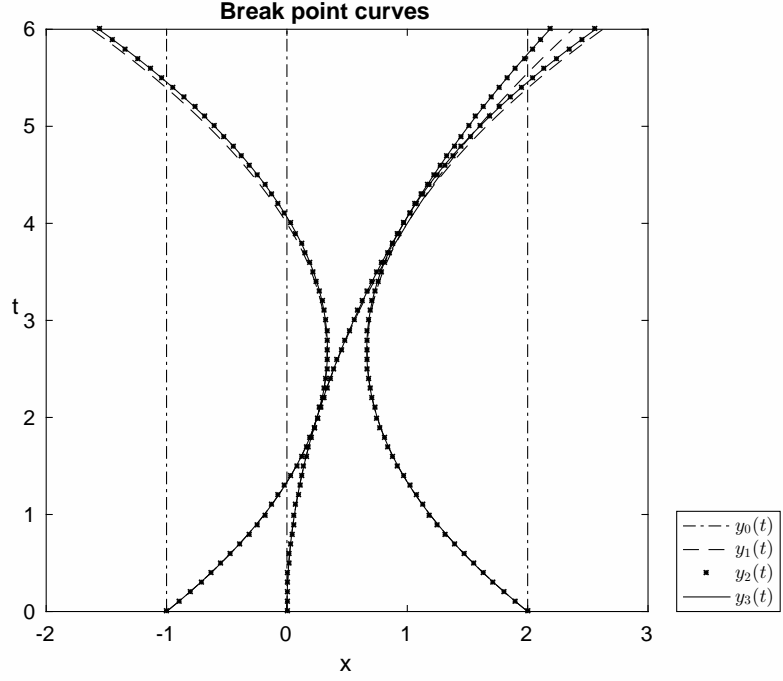}
\caption{A plot of characteristic curves $y_i(\xi,t), i=1,2,3,4$ in Example \ref{example iteration} corresponding to $\xi=-1$, $\xi=1$, and $\xi=\frac 72$ from left to right.}
\label{figure iteration}
\end{figure}
\end{example}

\begin{example}
\label{example degenerate nu}
In this example we demonstrate why it is necessary to restrict $alpha$ to $[0,1)$. Select $\alpha$ such that $\alpha(\frac 14) = 1$, and $\alpha(\frac 12) = \frac 12$. Let $(u_0,\rho_0,\nu_0,\mu_0)$ be given by $\rho_0=0$, $\nu_0=\mu_0= (\nu_0)_{ac}$, and
\begin{equation}
u_0(x) =
	\begin{cases}
	1, & x\leq -1,\\
	-x, & -1\leq x\leq 0,\\
	-\frac 12x, &0\leq x\leq 2,\\
	-1, &2\leq x.
	\end{cases}
\end{equation}
Then $X_0=L((u_0,\rho_0,\nu_0,\mu_0))$ is given by
\begin{align}
\label{equation example degenerate nu}
y_0(\xi) &=
	\begin{cases}
	\xi, &\xi\leq -1,\\
	\frac 12(\xi-1), &-1\leq\xi\leq 1,\\
	\frac 45(\xi-1), &1\leq\xi\leq \frac 72,\\
	\xi-\frac 32, &\frac 72\leq\xi,
	\end{cases}\\
U_0(\xi) &=
	\begin{cases}
	1, &\xi\leq -1,\\
	-\frac 12(\xi-1), &-1\leq\xi\leq 1,\\
	-\frac 25(\xi-1), &1\leq\xi\leq \frac 72,\\
	-1, &\frac 72\leq\xi,
	\end{cases}\\
H_0(\xi) &=
	\begin{cases}
	0, &\xi\leq -1,\\
	\frac 12(\xi+1), &-1\leq\xi\leq 1,\\
	\frac 15(\xi+4), &1\leq\xi\leq \frac 72,\\
	\frac 32, &\frac 72\leq\xi,
	\end{cases}\\
V_0 &= H_0,\\
r_0 &= 0.
\end{align}
We claim that $L\circ M\circ S_4(X_0) \neq \Pi\circ S_4(X_0)$. Note that $S_4(X_0)$ does not belong to the set $\F^\alpha$, since we choose an invalid $\alpha$. However the mappings $L$ and $M$ can still be applied in this more general case.  For $t=4$, $S_4(X_0)$ reads
\begin{align}
y(\xi,4) &=
	\begin{cases}
	\xi+\frac 32, &\xi\leq -1,\\
	\frac 12, &-1\leq\xi\leq \frac 72,\\
	\xi-3, &\frac 72\leq\xi,
	\end{cases}\\
U(\xi,4) &= 0,\\
H(\xi,4) &=
	\begin{cases}
	0, &\xi\leq -1,\\
	\frac 12(\xi+1), &-1\leq\xi\leq 1,\\
	\frac 15(\xi+4), &1\leq\xi\leq \frac 72,\\
	\frac 32, &\frac 72\leq\xi,
	\end{cases}\\
V(\xi,4) &=
	\begin{cases}
	0, &\xi\leq 1,\\
	\frac 1{10}(\xi-1), &1\leq\xi\leq \frac 72,\\
	\frac 14, &\frac 72\leq\xi.
	\end{cases}
\end{align}
ad $\tilde X= \Pi\circ S_4(X_0)$ is given by 
\begin{align}
\tilde y(\xi) &=
	\begin{cases}
	\xi, &\xi\leq \frac 12,\\
	\frac 12, &\frac 12\leq \xi\leq 2,\\
	\xi-\frac 32, &2\leq \xi,
	\end{cases}\\
\tilde U(\xi) &= 0,\\
\tilde H(\xi) &=
	\begin{cases}
	0,&\xi\leq\frac 12,\\
	\xi-\frac 12, &\frac 12\leq \xi\leq 2,\\
	\frac 32, &2\leq \xi,
	\end{cases}\\
\tilde V(\xi) &= 
	\begin{cases}
	0,&\xi\leq\frac 32,\\
	\frac 12\xi-\frac 34, &\frac 32\leq \xi\leq 2,\\
	\frac 14, &2\leq \xi.
	\end{cases}
\end{align}
On the other hand let $\bar X = L\circ M\circ S_4(X_0)$, then we get
\begin{align}
\bar y(\xi) &=
	\begin{cases}
	\xi, &\xi\leq \frac 12,\\
	\frac 12, &\frac 12\leq \xi\leq 2,\\
	\xi-\frac 32, &2\leq \xi,
	\end{cases}\\
\bar U(\xi) &= 0,\\
\bar H(\xi) &=
	\begin{cases}
	0,&\xi\leq\frac 12,\\
	\xi-\frac 12, &\frac 12\leq \xi\leq 2,\\
	\frac 32, &2\leq \xi,
	\end{cases}\\
\bar V(\xi) &= 
	\begin{cases}
	0,&\xi\leq\frac 12,\\
	\frac 16\xi-\frac 1{12}, &\frac 12\leq \xi\leq 2,\\
	\frac 14, &2\leq \xi.
	\end{cases}
\end{align}
Since $\tilde V \neq \bar V$ we have that if we allow that $\alpha(x) = 1$ for some $x$ and $\alpha(\bar x) < 1$ for some other $\bar x$, then we cannot guarantee that $L\circ M = \Id_{\F_0}$. By passing to Eulerian coordinates and back some information about $V$ is lost. In Figure \ref{figure energy degenerate nu} the cumulative function for $\nu$ is plotted at $t=0,2,4$. We see that the jump at $t=4$ contains two jumps, one created at $t=2$, the other at $t=4$. Uniqueness is lost for $t>4$ since we are unable to separate the jumps in a systematic way. This can be seen from Figure \ref{figure energy comp degenerate nu} where cumulative functions for $\nu$ and $\bar\nu$ have been plotted at $t=6$. The measure $\nu$ is obtained by $M\circ S_6\circ L$, while $\bar\nu$ is obtained by $M\circ S_2\circ L\circ M\circ S_4\circ L$. It should be pointed out that non-uniquenss after wave breaking for $\nu$ does not affect $(u,\rho,\mu)$ in any way.

\begin{figure}
\includegraphics[scale=0.7]{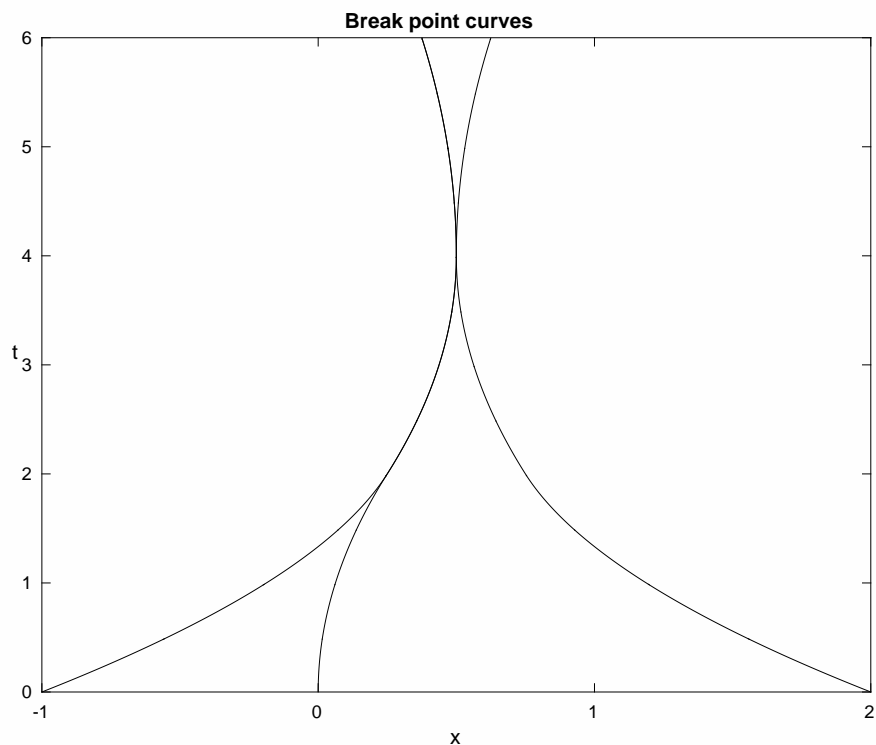}
\caption{A plot of characteristic curves $y(\xi,t)$ corresponding to $\xi=-1$, $\xi=1$, and $\xi=\frac 72$ in \eqref{equation example degenerate nu} in Example \ref{example degenerate nu} from left to right.}
\label{figure degenerate nu}
\end{figure}

\begin{figure}
\includegraphics[scale=0.8]{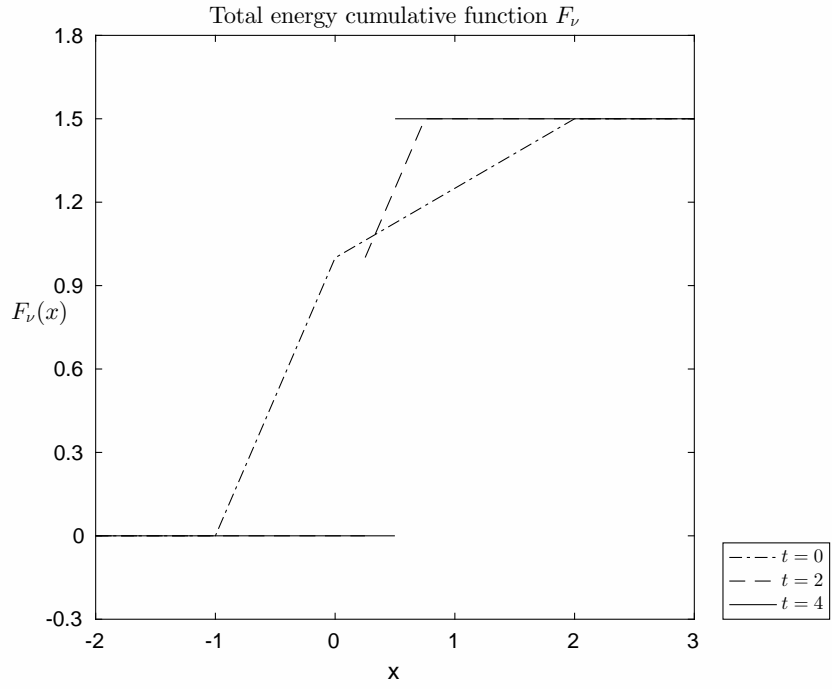}
\caption{A plot of the cumulative total energy function $F_\nu$ in Example \ref{example degenerate nu} at $t=0$, $t=2$, and $t=4$.}
\label{figure energy degenerate nu}
\end{figure}

\begin{figure}
\includegraphics[scale=0.8]{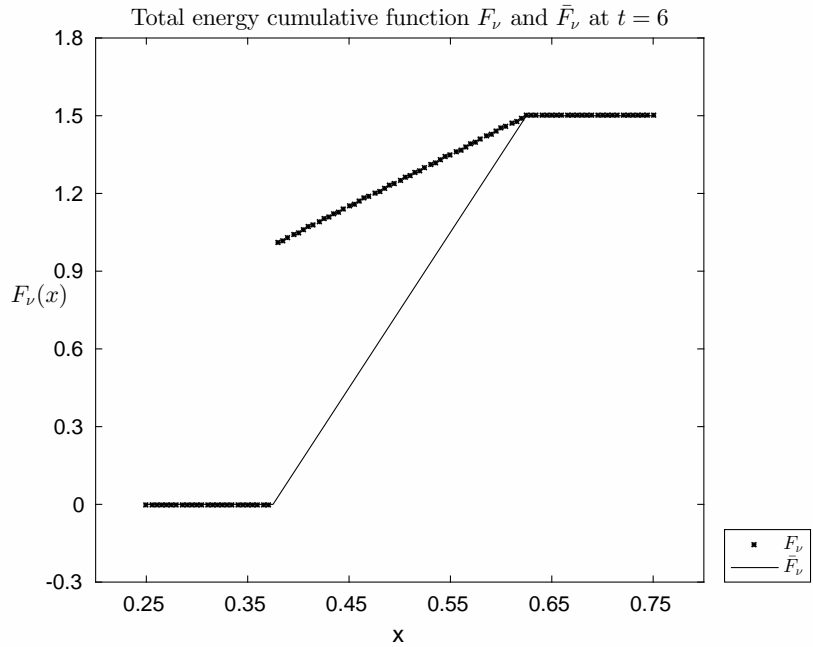}
\caption{A plot of the cumulative total energy functions $F_\nu$ and $F_{\bar\nu}$ from Example \ref{example degenerate nu} at $t=6$.}
\label{figure energy comp degenerate nu}
\end{figure}

\end{example}

\end{document}